\documentclass{amsart}
\usepackage{amsmath,amsthm,amscd,amssymb,amsfonts, amsbsy}
\usepackage{latexsym}
\usepackage{txfonts}

\numberwithin{equation}{section}

\theoremstyle{plain}
\newtheorem{theorem}[equation]{Theorem}
\newtheorem{lemma}[equation]{Lemma}
\newtheorem{corollary}[equation]{Corollary}

\theoremstyle{definition}

\theoremstyle{remark}

\newcommand{\fiint}{\operatornamewithlimits{\fint\!\!\!\!\fint}}
\newcommand{\dv}{\operatorname{div}}

\newcommand{\supp}{\operatorname{supp}}
\newcommand{\dist}{\operatorname{dist}}
\newcommand{\loc}{\operatorname{loc}}
\newcommand{\ok}{\operatorname{``OK"}}
\newcommand{\tr}{\operatorname{tr}}

\begin{document}

\title[Analyticity of layer potentials]{Analyticity of 
layer potentials and $L^{2}$ solvability of boundary value problems for divergence form elliptic
equations with complex $L^{\infty}$ coefficients}

\author[M. Alfonseca]{M. Angeles Alfonseca}
\address[M. Alfonseca]{Department of Mathematics, 
North Dakota State University, Fargo, ND 58105-5075}
\email{maria.alfonseca@ndsu.edu}

\author[P. Auscher]{Pascal Auscher}
\address[P. Auscher]{Universit\'e de Paris-Sud, UMR du CNRS 8628, 91405 Orsay Cedex, France}
\email{pascal.auscher@math.u-psud.fr}

\author[A. Axelsson]{Andreas Axelsson}
\address[A. Axelsson]{Matematiska Institutionen, Stockholms Universitet, 106 91 Stockholm, Sweden}
\email{andax@math.su.se}

\author[S. Hofmann]{Steve Hofmann}
\address[S. Hofmann]{Department of Mathematics, 
University of Missouri, Columbia, Missouri 65211, USA}
\email{hofmann@math.missouri.edu}
\thanks{S. Hofmann was supported by the National Science Foundation}

\author[S. Kim]{Seick Kim}
\address[S. Kim]{Centre for Mathematics and its Applications,
The Australian National University, ACT 0200, Australia}
\email{seick.kim@maths.anu.edu.au}

\begin{abstract} We consider 
divergence form elliptic operators of the form $L=-\dv A(x)\nabla$, defined in
$\mathbb{R}^{n+1}=\{(x,t)\in\mathbb{R}^{n}\times\mathbb{R}\}$, $n \geq 2$, where the 
$L^{\infty}$ coefficient matrix $A$ is
$(n+1)\times(n+1)$, uniformly elliptic, complex and $t$-independent. We show 
that for such operators, boundedness and
invertibility of the corresponding layer potential 
operators on $L^2(\mathbb{R}^{n})=L^2(\partial\mathbb{R}_{+}^{n+1})$, is stable
under complex, $L^{\infty}$ perturbations of the coefficient matrix. Using a variant of the $Tb$ Theorem, we also prove that the layer potentials
are bounded and invertible on $L^2(\mathbb{R}^n)$
whenever $A(x)$ is real and symmetric (and thus, by our stability result, also when $A$ is complex, $\Vert
A-A^0\Vert_{\infty}$ is small enough and $A^0$ is real,  symmetric,
$L^{\infty}$ and elliptic). 
In particular, we establish solvability of the Dirichlet and Neumann (and
Regularity) problems, with $L^2$ (resp. $\dot{L}^2_1)$ data, for small complex 
perturbations of
a real symmetric matrix.  Previously, $L^2$ solvability results for 
complex (or even
real but non-symmetric) coefficients were known to hold only for perturbations of constant matrices
(and then only for the Dirichlet problem), or in the special
case that the coefficients $A_{j,n+1}=0=A_{n+1,j}$, $1\leq j\leq
n$, which corresponds to the Kato square root problem. 
\end{abstract}

\maketitle

\tableofcontents

\section{Introduction, statement of results, history\label{s1}}

In this paper, we consider the solvability of boundary value problems for divergence form complex coefficient equations
$Lu=0$, where \begin{equation*} L=-\dv A\nabla\equiv-\sum_{i,j=1}^{n+1}\frac{\partial}{\partial
x_{i}}\left(A_{i,j} \,\frac{\partial}{\partial x_{j}}\right)\end{equation*}
is defined in $\mathbb{R}^{n+1}=\{(x,t)\in\mathbb{R}^{n}\times\mathbb{R}\}, n\geq 2$ (we use 
the notational convention that
$x_{n+1}=t$), and where $A=A(x)$ is an $(n+1)\times(n+1)$ matrix of complex-valued $L^{\infty}$ coefficients, defined on
$\mathbb{R}^{n}$ (i.e., independent of the $t$ variable) and satisfying the 
uniform ellipticity condition
\begin{equation}
\label{eq1.1} \lambda|\xi|^{2}\leq\Re e\,\langle A(x)\xi,\xi\rangle
\equiv \Re e\sum_{i,j=1}^{n+1}A_{ij}(x)\xi_{j}\bar{\xi_{i}}, \quad
  \Vert A\Vert_{L^{\infty}(\mathbb{R}^{n})}\leq\Lambda,
\end{equation}
 for some $\lambda>0$, $\Lambda<\infty$, and for all $\xi\in\mathbb{C}^{n+1}$, $x\in\mathbb{R}^{n}$. 
The divergence form equation is interpreted in the weak sense, i.e., we say that $Lu=0$
in a domain $\Omega$ if $u\in W^{1,2}_{loc}(\Omega)$ and 
$$\int A \nabla u \cdot \overline{\nabla \Psi} = 0$$ for all complex valued $\Psi \in C_0^\infty(\Omega)$.
 
The boundary value problems that we consider are classical.  To state them, we first recall
the definitions of the non-tangential maximal operators $N_{\ast},\widetilde{N}_{\ast}$. 
Given $x_0\in\mathbb{R}^{n}$, define
the cone $\gamma(x_0)=\{(x,t)\in\mathbb{R}_{+}^{n+1}:|x_0-x|<t\}$. Then
for $U $ defined in $\mathbb{R}_{+}^{n+1}$,
\begin{equation*} N_{\ast} U(x_0)  \equiv\sup_{(x,t)\in\gamma(x_0)}|U(x,t)|,\quad
\widetilde{N}_{\ast} U(x_0)  \equiv\sup_{(x,t)\in \gamma(x_0)}\left(\fint\!\!\fint_{\substack{|x-y|<t\\ |t-s|<t/2}
}|U(y,s)|^{2}dyds\right)^{\frac{1}{2}}.\end{equation*}
Here, and in the sequel, the symbol $\fint$ denotes the mean value, i.e.,
$\fint_E f \equiv |E|^{-1} \int_E f .$
We use the notation 
$u \to f \, n.t.$ to mean that
for $a.e. \, x \in \mathbb{R}^n$, $\lim_{(y,t) \to (x,0)} u(y,t) = f(x),$
where the limit runs over $(y,t)\, \in \gamma(x).$

We shall consider the Dirichlet problem\footnote{Our uniform $L^2$ estimate for solutions of (D2) can be improved to an $L^2$ bound for $N_*u$, given certain $L^p$ 
estimates for the layer potentials.  The fourth named author and M. Mitrea will present 
the $L^p$ theory in a forthcoming publication.  In the present paper, we shall be content with a weak-$L^2$ bound for
$N_*u$.} \begin{equation}
\begin{cases} Lu=0\text{ in }\mathbb{R}_{+}^{n+1}=\{(x,t)\in\mathbb{R}^{n}\times(0,\infty)\}\\ 
\lim_{t\to 0}u(\cdot,t)=f\text{ in }
L^{2}(\mathbb{R}^{n}) \text{ and } n.t.\\ 
\sup_{t>0}\|u(\cdot,t)\|_{L^2(\mathbb{R}^n)}<\infty,
\end{cases}\tag{D2}\label{D2}\end{equation}
 the Neumann problem\footnote{We shall elaborate in section \ref{s4nt} the precise nature
by which the co-normal derivative assumes the prescribed data.} \begin{equation}
\tag{N2}\begin{cases} Lu=0\text{ in }\mathbb{R}_{+}^{n+1}\\
\frac{\partial u}{\partial v}(x,0)\equiv-\sum_{j=1}^{n+1}A_{n+1,j}(x)
\frac{\partial u}{\partial
x_{j}}(x,0)=g(x)\in L^{2}(\mathbb{R}^{n})\\ 
\widetilde{N}_{\ast}(\nabla u)\in L^{2}(\mathbb{R}^{n}),\end{cases}\label{N2}\end{equation}
 and the Regularity problem \begin{equation}
\tag{R2}\begin{cases} Lu=0\text{ in }\mathbb{R}_{+}^{n+1}\\ u(\cdot,t)\to 
f\in\dot{L}_{1}^{2}(\mathbb{R}^{n}) \, n.t.\\
\widetilde{N}_{\ast}(\nabla u)\in L^{2}(\mathbb{R}^{n}).\end{cases}\label{R2}\end{equation}
Our solutions will be unique among the class of solutions satisfying the stated $L^2$ bounds (in the case of (N2) and (R2), this uniqueness will hold modulo constants).  The homogeneous Sobolev space $\dot{L}_{1}^{2}$ is defined as the completion of $C_0^{\infty}$ with respect to 
the seminorm
$\|\nabla F\|_2.$  For $n \geq 3$ this space can be identified (modulo constants) with the space 
$I_1(L^2) \equiv \Delta^{-1/2}(L^2) \subset L^{2^*},$ where $2^* \equiv 2n/(n-2);$ for $n=2$, the identification with
$I_1(L^2)$ is still valid, but in that case
the fractional integral $I_1 f$ must itself be defined modulo constants for $f \in L^2$, and the space embeds
in $BMO$.

We remark that for the class of operators that we consider, solvability of these boundary value problems in the
half-space may readily be generalized to the case of domains given by the region above a Lipschitz graph, and even to
the case of star-like Lipschitz domains. We shall return to this point later. We shall also discuss later the
significance of our assumption that the coefficients are $t$-independent.

In order to state our main results, we shall need to recall a few definitions and facts. We say that $u$ is locally
H\"{o}lder continuous in a domain $\Omega$ if there is a constant $C$
and an exponent $\alpha>0$ such that for any ball
$B=B(X,R)$, of radius $R$, whose concentric double $2B \equiv 
B(X,2R)$ is contained in $\Omega$, we have that \begin{equation}
|u(Y)-u(Z)|\leq
C\left(\frac{|Y-Z|}{R}\right)^\alpha\left(\fint_{2B}|u|^{2}\right)^{\frac{1}{2}},\label{eq1.2}\end{equation}
whenever $Y,Z\in B$. Observe that any $u$ satisfying \eqref{eq1.2} also 
satisfies Moser's ``local boundedness" estimate \cite{M}
\begin{equation}
\sup_{Y\in B}|u(Y)|\leq C \left(\fint_{2B}|u|^{2}\right)^{\frac{1}{2}}.\label{eq1.3}\end{equation}
By the classical De Giorgi-Nash Theorem~\cite{DeG,N}, \eqref{eq1.2} and hence 
also \eqref{eq1.3} hold, with $C$ and $\alpha$ depending only on dimension and
the ellipticity parameters, whenever $u$ is a
solution of $Lu=0$ in $\Omega\subseteq\mathbb{R}^{n+1}$, if {\it in addition} the coefficient matrix $A$ is real
(for this result, it need not be $t$-independent). Moreover, it is shown in
\cite{A} (see also~\cite{AT,HK}), that property \eqref{eq1.2} is 
stable under complex, $L^{\infty}$ perturbations.

We now recall the method of layer potentials. For $L$ as above, 
let $\Gamma,\,\Gamma^*$ denote the fundamental solutions\footnote{See \cite {HK2} for a construction of the fundamental solution.}
for $L$ and $L^*$ respectively, in $\mathbb{R}^{n+1}$, so that $$L_{x,t} \,\Gamma (x,t,y,s) = \delta_{(y,s)},\,\,\,
L^*_{y,s}\, \Gamma^*(y,s,x,t) \equiv L^*_{y,s} \,\overline{\Gamma (x,t,y,s)} = \delta_{(x,t)},$$
where $\delta_X$ denotes the Dirac mass at the point $X$, and
$L^*$ is the hermitian adjoint of $L$.  By the $t$-independence of our 
coefficients, we have that \begin{equation}
\Gamma(x,t,y,s)=\Gamma(x,t-s,y,0).\label{eq1.4}\end{equation}
 We define the single and double layer potential operators, by 
 \begin{equation}
\begin{split}\label{eq1.5}S_{t}f(x) & \equiv\int_{\mathbb{R}^{n}}\Gamma(x,t,y,0)\,f(y)\,dy, \,\,\, t\in \mathbb{R}\\
\mathcal{D}_{t}f(x) & 
\equiv\int_{\mathbb{R}^{n}}\overline{\partial_{\nu^*} \Gamma^*
(y,0,x,t)}\,f(y)\,dy,\,\,\, t \neq 0,\end{split}
\end{equation}
 where $\partial_{\nu^*}$ is the adjoint exterior
 conormal derivative; i.e., if $A^{\ast}$ denotes the hermitian adjoint of $A$, then
\begin{equation*}
\partial_{\nu^*} \Gamma^* (y,0,x,t)
=-\sum^{n+1}_{j=1}A_{n+1,j}^{\ast}(y)\frac{\partial \Gamma^*}{\partial
y_{j}}(y,0,x,t)=-e_{n+1}\cdot A^{\ast}(y)
\nabla_{y,s}\Gamma^*(y,s,x,t) \mid_{s=0}\end{equation*}
 (recall that $y_{n+1}=s$). We define (loosely\footnote{For non-smooth coefficients, 
 some care should be taken to define the ``principal value"
operators on the boundary - see Section \ref{s4nt}.}, for the moment) 
boundary singular integrals
\begin{equation}
\begin{split}\label{eq1.6}Kf(x) & \equiv
  ``p.v."\int_{\mathbb{R}^{n}}\overline{\partial_{\nu^*} \Gamma^* (y,0,x,0)}\,f(y)\,dy\\
\widetilde{K}f(x) & \equiv ``p.v."\int_{\mathbb{R}^{n}}\frac{\partial\Gamma}{\partial\nu}(x,0,y,0)\,f(y)\,dy\end{split}
\end{equation}
 where $\frac{\partial}{\partial\nu}$ denotes the exterior conormal 
 derivative in the $(x,t)$ variables.  Classically,
$\widetilde{K}$ is often denoted $K^{\ast}$, but we avoid this notation here as $\widetilde{K}$ need not be the adjoint of $K$ unless
$L$ is self-adjoint. Rather, for us, $K^*, S^*$ and $\mathcal{D}^*$ will denote the analogues of $K,S$ and $\mathcal{D}$ corresponding to $L^*$ (although sometimes we shall write $K^{L^*}$, etc., when we wish to emphasize the dependence on a particular operator), and we use the notation $ad\!j\,(T)$ to denote the Hermitian adjoint of an operator $T$ acting in $\mathbb{R}^n$.  With these conventions, we have that $\widetilde{K} = ad\!j\, (K^*)$, as the 
reader may verify. We apologize for this departure from tradition, but the context of
complex coefficients seems to require it.

For sufficiently smooth coefficients, the following ``jump relation" formulae
have been established in \cite{MMT}.   We defer to 
Section \ref{s4nt} our discussion of the jump formulae, and the nature of their ``non-tangential" realization, in the non-smooth case.  We have
\begin{eqnarray}
\label{eq1.7}\mathcal{D}_{\pm s} f
\!& \to & \! \left(\mp\frac{1}{2}I+K\right)f \\\label{eq1.7a} 
\left(\nabla S_{t}\right)|_{t=\pm s} f
\!&\to & \!\mp\frac{1}{2}\cdot\frac{f(x)}{A_{n+1,n+1}(x)}e_{n+1}+\mathcal{T}f  ,
\end{eqnarray}
(these convergence statements must be interpreted properly - see Section \ref{s4nt}) where \begin{equation} \mathcal{T}f(x)\equiv ``p.v."\int_{\mathbb{R}^{n}}
 \nabla\Gamma(x,0,y,0)f(y)dy.\label{eq1.8}\end{equation}
Then, as usual\footnote{In the setting of non-smooth coefficients, 
some rather extensive preliminaries
are required in order to apply the layer potential method to obtain solvability;
see Section \ref{s4nt}.}, 
one obtains solvability of \eqref{D2} in the upper (resp. lower) half 
space by establishing 
boundedness on $L^{2}(\mathbb{R}^n)$ of 
$f\to \mathcal{D}_{\pm t} f$, uniformly in $t$,
and invertibility of $-\frac{1}{2}I+K$ (resp. $\frac{1}{2}I+K$). Similarly, solvability of \eqref{N2} and
\eqref{R2} follows from $L^{2}$ boundedness of  $f \rightarrow
\widetilde{N}_{\ast}(\nabla S_{\pm t}f),$ and (for \eqref{N2}) invertibility on
$L^{2}$ of $\pm\frac{1}{2}I+\widetilde{K}$, and (for \eqref{R2}) invertibility 
of the mapping
$S_{0}=S_{t}\!\mid_{t=0}\,:L^{2}(\mathbb{R}^{n})\rightarrow\dot{L}_{1}^{2}(\mathbb{R}^{n})$. We now set some convenient
terminology: we shall say that an operator $L$ for which all of the above hold has 
``Bounded and Invertible Layer
Potentials".  If {\it in addition} we have 
the square function estimate \begin{equation}
\int_{-\infty}^{\infty}\int_{\mathbb{R}^{n}}\left|t\partial_{t}^{2}S_{t}f(x)\right|^{2}\frac{dx dt}{|t|}\leq C\Vert
f\Vert_{2}^{2},\label{eq1.9}\end{equation}
then we shall say that $L$ has ``Good Layer Potentials".  Finally, we shall refer to the
constant in \eqref{eq1.9}, together with all of the constants arising 
in the estimates for the boundedness and invertibility of the layer potentials,
collectively as the ``Layer Potentials Constants" for $L$.

In this paper, we prove the following theorems.  In the sequel we assume always that
our $(n+1) \times (n+1)$ coefficient matrices are $t-$independent, complex,
and satisfy the ellipticity condition \eqref{eq1.1} and the De Giorgi-Nash-Moser estimates
\eqref{eq1.2} and \eqref{eq1.3}.

\begin{theorem}\label{t1.10} Suppose that $L_0=-\dv A^0\nabla $ and $L_1 =-\dv A^1\nabla$ are operators of the 
type described above, and that solutions $u_0,\,w_0$ of $L_0u_0=0,\,
L^*_0 w_0 = 0$ satisfy the De Giorgi-Nash-Moser 
estimates \eqref{eq1.2} and \eqref{eq1.3}. Suppose also that $L_0$ and $L^\ast_0$ have
``Good Layer Potentials". Then $L_1$ and $L^*_1$ have Good Layer Potentials, provided that
\begin{equation*} \|A^0-A^1\|_{L^\infty (\mathbb{R}^n)}\leq \epsilon_0,\end{equation*} where 
$\epsilon _0$ is
sufficiently small depending only on dimension and on the various constants associated to 
$L_0$ and $L^*_0$, specifically: the ellipticity parameters, 
the De Giorgi-Nash-Moser constants \eqref{eq1.2} and \eqref{eq1.3}, 
and the Layer Potential Constants.\end{theorem}
We observe that it is not clear whether the property that $L$ has ``Good Layer Potentials"
is preserved under regularization of the coefficients.  For this reason, we shall be forced to prove 
Theorem \ref{t1.10} without recourse to the usual device of making an {\it a priori} assumption of smooth coefficients.  We also note that we shall use the invertibility of the layer potentials
associated to $L_0$ and $L_0^*$ even to establish the boundedness of
the layer potentials associated to $L_1$ (see Section \ref{s5} below).

\begin{theorem}\label{t1.11} Suppose that $L=-\dv A\nabla$ is an operator of the type defined above, and in addition,
suppose that $A$ is real and symmetric. Then $L$ has Good Layer Potentials, and its Layer Potential Constants depend only on dimension and on the ellipticity parameters in \eqref{eq1.1}.\end{theorem}

We remark that while Theorem~\ref{t1.11} yields in particular the solvability of
(D2), (N2) and (R2) in the case that $A$ is real and symmetric, it is only the 
fact that this solvability is obtainable 
via layer potentials that is new here,
the solvability of \eqref{D2} having been previously obtained by
Jerison and Kenig~\cite{JK1}, and that of \eqref{N2} and \eqref{R2} by Kenig and Pipher~\cite{KP}, without the use of layer potentials.  The essential missing ingredient had been
the boundedness of the layer potentials.

The previous two theorems are our main results.
As corollaries, we obtain

\begin{theorem}\label{t1.12} Suppose that $L_1=-\dv A^1\nabla$ is an operator of the type defined above, and that
$\|A^1-A^0\|_{L^\infty (\mathbb{R}^n)}\leq \epsilon_0$, for some real, symmetric, $t$-independent uniformly elliptic
matrix $A^0\in L^\infty (\mathbb{R}^n)$. Then \eqref{D2}, \eqref{N2} and \eqref{R2} are all solvable for $L_1$, provided
that $\epsilon_0$ is sufficiently small, depending only on dimension and the ellipticity parameters for
$A^0$.  The solution of \eqref{D2} is unique among the class of solutions $u$ for which
$\sup_{t>0}\|u(\cdot,t)\|_{L^2(\mathbb{R}^n)} < \infty$, and the solutions of \eqref{N2} and \eqref{R2} are unique modulo constants among the class of solutions 
for which $\widetilde{N}_*(\nabla u) \in L^2$.\end{theorem}

\begin{theorem}\label{t1.13}  The conclusion of Theorem~\ref{t1.12} 
holds also in the case that $\|A^1-A^0\|_\infty$ is
sufficiently small, where $A^0$ is now a 
constant, elliptic complex matrix.\end{theorem}
The last theorem follows from Theorem~\ref{t1.10}, and the fact that constant coefficient
operators have Good Layer Potentials (see the appendix, Section \ref{sappend}).

We note that by a standard device, Theorems~\ref{t1.10}, \ref{t1.11} and \ref{t1.12} all extend readily to the case
where $\Omega=\{(x,t):t>F(x)\}$, with $F$ Lipschitz. Indeed, by {}``pulling back\char`\"{} under the mapping
$\rho:\mathbb{R}_{+}^{n+1}\rightarrow\Omega$ defined by \begin{equation*}
\rho(x,t)=(x,F(x)+t),\end{equation*}
 we may reduce to the case of the half-space. The pull-back operators
 are of the same type, and, in particular, the coefficients remain
$t$-independent. Moreover, if the original coefficients are real and symmetric, then so are those of the pull-back
operator. In this setting, the parameter $\epsilon_{0}$ will also depend on $\Vert\nabla F\Vert_{\infty}$. In addition,
our results may be further extended to the setting of star-like
Lipschitz domains (which would seem to be the most
general setting in which the notion of {}``radial independence\char`\"{} 
of the coefficients makes sense). The idea is to use a partition of unity
argument, as in \cite{MMT}, to reduce to the case of a Lipschitz graph.
We omit the details.

Let us now briefly review the history of work in this area, which falls broadly
into two categories, depending on whether or not the $t$-independent
coefficient matrix is
self-adjoint. We discuss the former category first, and we mention only the
case of a single equation, although results for certain constant
coefficient self-adjoint
systems in a Lipschitz domain are known, see e.g. \cite{K,K2} for 
further references.  (Moreover, the present setting of complex coefficients
may be viewed in the context of $2\times 2$ systems, and indeed this provides part of
our motivation to consider the complex case).
For Laplace's equation in a Lipschitz domain, the 
solvability of \eqref{D2} was obtained by Dahlberg~\cite{D}, and
that of \eqref{N2} and \eqref{R2} by Jerison and Kenig~\cite{JK2}; solvability of the same problems
via harmonic layer potentials is due to
Verchota~\cite{V}, using the deep result of Coifman, McIntosh and 
Meyer~\cite{CMcM} concerning the $L^{2}$ boundedness
of the Cauchy integral operator on a Lipschitz curve. The results of 
\cite{V} and \cite{CMcM} are subsumed in our
Theorem~\ref{t1.11} via the pull-back mechanism discussed above.
Moreover, as mentioned above, for $A$ real, symmetric and t-independent, 
the solvability
of (D2) was obtained in \cite{JK1}, and that of (N2) and 
(R2) in \cite{KP}, but those authors did not use layer
potentials.  The case of real symmetric coefficients with some smoothness 
has been treated via layer potentials in \cite{MMT}.

In the ``non self-adjoint" setting, previous results had been obtained in 
three special cases.
First, it was known that \eqref{D2} is solvable for 
small, complex perturbations of {\it constant} elliptic matrices. This is
due to Fabes, Jerison and Kenig~\cite{FJK} via the method of 
multilinear expansions.  To our knowledge, (R2) and (N2) had not been treated in this setting.

Second, one has solvability of \eqref{D2}, \eqref{N2} and \eqref{R2} in the special 
case that the matrix $A$ is of the
``block " form
\begin{equation}
\left[\begin{array}{c|c}
 & 0\\ B & \vdots\\
 & 0\\
\hline 0\cdots0 & 1\end{array}\right]\label{eq1.14}\end{equation}
 where $B=B(x)$ is a $n\times n$ matrix. In this case, (D2) is an easy 
consequence of the semigroup theory, while \eqref{R2} amounts to solving the Kato square root problem for
the $n$-dimensional operator \begin{equation*} J=-\dv_{x}B(x)\nabla_{x},\end{equation*}
 and \eqref{N2} amounts to $L^{2}$ boundedness of the Riesz transforms $\nabla J^{-\frac{1}{2}}$ (equivalently, to
solving the Kato problem for the adjoint operator $ad\!j \,(J)$). Moreover, the boundedness of the Riesz transform
$\nabla J^{-\frac{1}{2}}$ can also be interpreted as the statement that the single 
layer potential is bounded from $L^2$ into $\dot{L}^2_1.$
These results were obtained in \cite{CMcM} $(n=1),$
\cite{HMc} $(n=2)$, \cite{AHLT} (when $B$ is a perturbation of a 
real, symmetric matrix), \cite{HLMc} (when the
kernel of the heat semi-group $e^{-tJ}$ has a Gaussian upper bound) 
and \cite{AHLMcT} in general\footnote{We remark that Theorem \ref{t1.10} may be combined with these results for block matrices \eqref{eq1.14} to allow perturbations of the block case, but we do not pursue this point here;
see, however, \cite{AAH}, where this is done without imposing De Giorgi-Nash-Moser bounds,
and where also extensions of Theorems \ref{t1.12} and \ref{t1.13} will be presented,
via the development of a functional calculus for certain Dirac type operators.}.  

Third, Kenig, Koch, Pipher and Toro~\cite{KKPT} have obtained 
solvability of (Dp) (the problems (Dp), (Np) and (Rp) are defined analogously to (D2), (N2) and (R2), 
but with $L^2$ bounds replaced by $L^p$) in the case $n=1$ (that is, in $\mathbb{R}^2_+$),
for $p$ sufficiently large
depending on $L$, in the case that $A(x)$ is real, but non-symmetric. Moreover, they construct a family of examples in
$\mathbb{R}_{+}^{2}$ in which solvability of
(Dp) may be destroyed for any specified $p$ by taking $A(x)$ to be an 
appropriate perturbation of the $2\times 2$ identity matrix.  Very recently, in the same setting of real, non-symmetric 
coefficients in two dimensions (that is, in $\mathbb{R}^2_+$), Kenig and Rule \cite{KR} have obtained
solvability of (Nq) and (Rq), where q is dual to the \cite{KKPT} exponent.  
Their result uses boundedness, but not invertibility, of the layer potentials.

The main purpose, then, of the present paper is to develop, to the extent possible, an $L^{2}$ theory of boundary value problems
for \textit{full} coefficient matrices with complex (including also real, not necessarily symmetric) entries. In fact,
in the setting of $L^{2}$ solvability with $t$-independent coefficients, the counter-example of \cite{KKPT} shows that
our perturbation results are in the nature of best possible.

A word about $t$-independence is in order. It has been observed by 
Caffarelli, Fabes and Kenig \cite{CFK} that some regularity in
the transverse direction is necessary, in order to deduce solvability of \eqref{D2}. More precisely, they show that
given any function $\omega(\tau)$ with $\int_{0}^{1}(\omega(\tau))^{2}d\tau/\tau=+\infty$, there exists a real, symmetric
elliptic matrix $A(x,t)$, whose modulus of continuity in the $t$ direction is controlled by $\omega$, but for which the
corresponding elliptic-harmonic measure and the Lebesque measure on the boundary are mutually singular. On the other
hand, it is shown in \cite{FJK} that (D2) does hold, assuming that the transverse modulus of continuity
$\omega(\tau) \equiv\sup_{x\in\mathbb{R}^{n},\,0<t<\tau} |A(x,t)-A(x,0)|$ satisfies
the square Dini condition $\int_{0}^{1}(\omega(\tau))^{2}d\tau/\tau<\infty$,
provided that $A(x,0)$ is sufficiently close to a constant matrix $A_{const}$. It seems likely that the methods of the present paper would allow us to obtain a
similar result, but with the constant matrix $A_{const}$ replaced by an $L^{\infty}$ matrix $A^{0}(x)$ satisfying the
hypotheses of Theorem~\ref{t1.10} (in particular, real, symmetric). However, we have not pursued this variant here, in
part because we conjecture that somewhat sharper estimates should be true. To explain this point of view, we recall that
a more refined, scale invariant version of the square Dini condition has been introduced by R. Fefferman, Kenig and
Pipher~\cite{FKP}, and Kenig and Pipher~\cite{KP,KP2}, to prove 
perturbation results in which one
assumes (roughly) that $|A^{1}(x,t)-A^{0}(x,t)|^{2}\frac{dxdt}{t}$ is a 
Carleson measure (actually, their condition is slightly
stronger, but in the same spirit). Note that this condition 
requires that $A^{1}=A^{0}$ on the boundary. Our work
provides a complement to \cite{FKP} and \cite{KP,KP2}, in 
that we allow the coefficients to differ at the boundary. At
present, the results of \cite{FKP} and \cite{KP,KP2} apply only to the 
case of real coefficients. It is an interesting open problem to extend the 
theorems of \cite{FKP} and \cite{KP,KP2} 
to the case of complex coefficients, even in the case of small Carleson norm. Given
such an extension, along with our results here, one could specialize to the case $A^{1}(x,t)=A(x,t)$,
$A^{0}(x,t)=A(x,0)$, with $A(x,0)$ close enough to a 
{}``good\char`\"{} (e.g., real, symmetric) matrix, to obtain a
rather complete picture of the situation for $L^{2}$ solvability.

Let us now set some notation that will be used throughout the paper. We shall use $\dv$ and $\nabla$ to denote the full
$n+1$ dimensional divergence and gradient, respectively. At times, we shall need to consider the $n$-dimensional
gradient and divergence, acting only in $x$, and these we denote either by $\nabla_{\|}$ and $\dv_{\|}$,
or by $\nabla_x$ and $\dv_x$; i.e.
\begin{equation*}
\nabla_{\|}=\left(\frac{\partial}{\partial x_{1}},\frac{\partial}{\partial x_{2}},\dots,\frac{\partial}{\partial
x_{n}}\right) = \nabla_x \end{equation*}
and for $\mathbb{R}^n$-valued $\vec{w}$, $\dv_{\|}\vec{w}\equiv\nabla_{\|}\cdot \vec{w}$. 
Similarly, given an $(n+1)\times (n+1)$ matrix $A$, we shall let $A_{\|}$ denote the $n \times n$ sub-matrix with entries $(A_{\|})_{i,j} \equiv A_{i,j}, \,\,\,\, 1 \leq i,j \leq n,$
and we define the corresponding elliptic operator acting in $\mathbb{R}^n$ by
$$L_{\|} \equiv -\dv_x A_{\|} \nabla_x .$$
We shall also use
the notation 
\begin{equation*} D_{j}\equiv \frac{\partial}{\partial x_{j}} = \partial_{x_j}\,,
\quad1\leq j\leq n+1\end{equation*}
 bearing in mind that $x_{n+1}=t$. Points in $\mathbb{R}^{n+1}$ 
 may sometimes be denoted by capital letters, e.g.
$X=(x,t)$, $Y=(y,s)$.  Balls in $\mathbb{R}^{n+1}$ and $\mathbb{R}^n$ will be denoted respectively by 
$B(X,r) \equiv \{Y:|X-Y|<r\}$ and $\Delta_r(x) \equiv \{y: |x-y|<r\}.$
We shall often encounter operators whose 
kernels involve derivatives applied to the second set of
variables in the fundamental solution $\Gamma(x,t,y,s)$. We shall 
indicate this by grouping the operators with
appropriate parentheses, thus: \begin{equation*}\left(S_{t}\nabla\right)f(x) 
\equiv\int_{\mathbb{R}^{n}}\nabla_{y,s}\Gamma(x,t,y,s)\mid_{s=0}f(y)\,dy.\end{equation*}
Hence,  one then has
$$\left(S_t \nabla_\|\right)\cdot \vec{f} = 
- S_t \left(\dv_\| \vec{f}\right), \quad (S_t D_{n+1}) = - \partial_t \,S_t,$$
where in the second identity we have used \eqref{eq1.4}

Given a cube $Q$, we denote the side length of $Q$ by $\ell(Q)$. Furthermore, given a positive number $r$, we let
$rQ$ denote the concentric cube with side length $r\ell(Q)$.  

We shall use $P_t$ to denote a nice approximate identity, acting on 
functions defined on $\mathbb{R}^n$;  i.e. 
$P_tf(x)=\phi_t \ast f,$ where $\phi_t(x)=t^{-n}\phi \left(x/t\right)$, $\phi \in
C^\infty_0 (\{ |x|<1\})$, $0\leq \phi$ and $\int_{\mathbb{R}^n}\phi=1$.

Following~\cite{FJK}, we introduce a convenient norm for dealing with square functions (although we warn the reader that
our measure differs from that used in \cite{FJK}): \begin{equation*}
\Vert|F\Vert|_{\pm} \equiv
\left(\iint_{\mathbb{R}_{\pm}^{n+1}}|F(x,t)|^{2}\frac{dxdt}{|t|}\right)^{\frac{1}{2}},
\quad \Vert|F|\Vert_{all} 
\equiv\left(\iint_{\mathbb{R}^{n+1}}|F(x,t)|^{2}\frac{dxdt}{|t|}\right)^{\frac{1}{2}}.\end{equation*}
For a family of operators $U_{t}$, we write
\begin{equation*}
\Vert|U_{t}|\Vert_{+,op}\equiv\sup_{\Vert f\Vert_{L^{2}(\mathbb{R}^{n})}=1}\Vert|U_{t}f|\Vert_{+},\end{equation*}
and similarly for $|\|\cdot\||_{-,op}$ and $|\|\cdot\||_{all,op}$. Sometimes, we may drop the {}``$+$\char`\"{} sign when it is clear that we are working in the upper
$\frac{1}{2}$-space. As usual, we allow generic constants $C$ to depend upon dimension and ellipticity, and, in the proof of the perturbation result, upon the constants associated to the ``good" operator $L_0$. Specific
constants, still depending on the same parameters, will be denoted $C_{1}$, $C_{2}$, etc..

The paper is organized as follows. In sections~\ref{s2} and \ref{s3offdiag}, we prove some useful technical estimates.  In section \ref{s4nt} we discuss the boundary behavior  and uniqueness of our solutions. 
The next five sections are the heart of the matter, in which we prove Theorem~\ref{t1.10}  
(sections~\ref{s3},
\ref{s4} and \ref{s5}), and Theorem~\ref{t1.11} (sections~\ref{s6} and \ref{s7}).  
Section \ref{sappend} is an appendix, in which we briefly discuss the constant coefficient case.

\medskip
\noindent {\bf Acknowledgements}.  The fourth named author thanks M. Mitrea for
helpful conversations concerning several of the topics treated in this work, including constant coefficient operators, the boundary behavior of layer potentials, and in particular,
for suggesting the approach used here in Lemma \ref{l4.ntjump} to obtain the analogue of the classical jump relation formulae.

\section{Some consequences of De Giorgi-Nash-Moser bounds\label{s2}}

Throughout this section, and throughout the rest of the paper, we suppose always that our differential
operators satisfy our ``standard assumptions": that is, divergence 
form elliptic, with ellipticity parameters
$\lambda$ and $\Lambda$, defined in $\mathbb{R}^{n+1}, n\geq 2$, with complex coefficients that are bounded, measurable and $t-$independent;  moreover, we suppose that solutions of $Lu=0$ satisfy the
De Giorgi-Nash-Moser estimates \eqref{eq1.2} and \eqref{eq1.3}.  We now prove some technical estimates using rather familiar arguments.   In the sequel,
$\Gamma$ will denote the 
fundamental solution of $L$, and we set 
\begin{equation}\label{eq2.0}K_{m,t}(x,y) \equiv (\partial_t)^{m+1} \Gamma(x,t,y,0)\end{equation}  
\begin{lemma}\label{l2.1} Suppose that $L$ and $L^*$ satisfy the
``standard assumptions" as above. Then there exists a constant $C_1$ depending only on
dimension, ellipticity and \eqref{eq1.2} and \eqref{eq1.3}, such that for every integer $m \geq -1$,
for all $t\in \mathbb{R}$, and $x,y\in \mathbb{R}^n$, we have
\begin{equation}\label{eq2.5} \left|K_{m,t} (x,y)\right| \leq
CC^{m^2}_1 (|t|+|x-y|)^{-n-m}\end{equation}
\begin{equation}\label{eq2.6} \left|\left(\mathbb{D}^hK_{m,t}(\cdot,y)\right)(x)\right| + 
\left|\left(\mathbb{D}^hK_{m,t}(x,\cdot)\right)(y)\right|
\leq CC^{m^2}_1 \frac{|h|^\alpha}{(|t|+|x-y|)^{n+m+\alpha}},\end{equation} whenever 
$2|h|\leq |x-y|$ or $|h|<20|t|$, for some $\alpha >0$, where $\left(\mathbb{D}^h f\right)(x) 
\equiv f(x+h) - f(x).$
\end{lemma}

\begin{proof}[Sketch of proof] The case $m=-1$ of \eqref{eq2.5} follows from its 
parabolic analogue in \cite{AT}, Section 1.4; alternatively, the reader may consult \cite{HK2}
for a direct proof in the elliptic case. The case $m = 0$ may be
treated by applying \eqref{eq1.3} to the solution $u(x,t)=\partial_t \Gamma (x,t,y,0)$  in the ball $B((x,t),R/2)$, with $R=\sqrt{|t|^2+|x-y|^2}$, and then using Caccioppoli's inequality to reduce to the case
$m=-1$. The case $m>0$ is obtained
by iterating the previous argument, and \eqref{eq2.6} follows from \eqref{eq1.2} and \eqref{eq2.5}.\end{proof}

We remark that, by taking more care with the Caccioppoli argument, 
using a ball of appropriately chosen radius
$c_m R$ rather than $R/2$, one may obtain the natural growth bound $m! C_1^m$ in \eqref{eq2.5} and
\eqref{eq2.6}.  We leave the details to the interested reader.

\begin{lemma}\label{l2.7} Suppose that $L,L^*$ satisfy the standard assumptions.  Then,
there exists a constant $C_2$, and for each $\rho>1$ a constant $C_\rho$, such that  for every cube $Q\subseteq \mathbb{R}^n$, for
all $x\in Q$, and for all integers $k\geq 1$ and $m\geq -1$, we have
\begin{enumerate}\item[(i)]$\int_{2^{k+1}Q\backslash 2 ^kQ}\left| (2^k\ell
(Q))^m\left(\partial_t\right)^{m+1}\nabla_y \Gamma (x,t,y,0)\right|^2 dy 
\leq CC_2^{m^2} (2^k\ell(Q))^{-n-2}, \,\,\,\,\, \forall t\in \mathbb{R}$
\item[(ii)] $\int_{2Q}\left| \ell
(Q)^m\left(\partial_t\right)^{m+1}\nabla_y \Gamma (x,t,y,0)\right|^2 dy 
\leq C_\rho^{m^2+1} \ell(Q)^{-n-2}, \,\,\,\, \frac{\ell (Q)}{\rho}<|t|< \rho\,\ell (Q).$\end{enumerate}\end{lemma}

\begin{proof} We first suppose that $A \in C^{\infty}$; we shall remove this restriction at the end of the proof.  Of course, our quantitative bounds will not depend on smoothness.  Let us consider estimate (i) first.  We shall actually prove that for $C_2$ large enough we have
\begin{equation}\label{eq2.8} \sum^\infty_{m=0}C_2 ^{-m^2} \left\| (2^k\ell(Q))^m\left(
\partial_t \right)^{m+1} \nabla_y \Gamma (x,t,\cdot,0 )\right\|^2 _{L^2(2^{k+1}Q\backslash 2^kQ)}
\leq C(2^k\ell (Q))^{-n-2}.\end{equation} Fix $x\in Q$. Let $\varphi_k\in C^\infty_0$, $\varphi_k\equiv 1$
on $2^{k+1}Q\backslash 2^kQ$, $\supp
\varphi_k\subset \frac{3}{2}2^{k+1}Q\backslash \frac{3}{2}2^{k-1}Q$, with
$$\|\nabla \varphi_k\|_\infty \leq C(2^k \ell(Q))^{-1}.$$  We observe that
\begin{equation*}\begin{split}I_m&\equiv \int \left|\left( \partial_t\right)^{m+1} \nabla_y\Gamma
(x,t,y,0)\right|^2 \varphi^2_k (y)dy\\ &\leq C\, \Re e \int A^*_{\|} \,\nabla_y \left( \partial_t\right)^{m+1}
\Gamma (x,t,y,0)\cdot
\overline{\nabla_y \left(\partial_t\right)^{m+1} \Gamma (x,t,y,0)}\varphi^2_k (y)\,dy\\
\intertext{(where $A^*_{\|}$ is the adjoint of the $n\times n$ matrix $A_{\|}$ 
defined by $(A_{\|})_{ij}=A_{ij}$, $1\leq i,j\leq n$)}\end{split}\end{equation*} 
\begin{equation*}\begin{split}&=C\, \Re e \int (L^*_{\|})_y
\left( \partial_t\right)^{m+1}\Gamma (x,t,y,0)\,\overline{\left(
\partial_t\right)^{m+1}\Gamma (x,t,y,0)}\,\varphi^2_k (y)\,dy\\ &\qquad- C\,\Re e \int A^*_{\|} \,\nabla_y \left(
\partial_t\right)^{m+1}\Gamma (x,t,y,0)\,\overline{\left(
\partial_t\right)^{m+1}\Gamma (x,t,y,0)} \cdot \nabla_y\varphi^2_k(y)\,dy\\ &=I'_m+I''_m,\end{split}\end{equation*}
where $L^*_{\|} \equiv -\dv_{x}A^*_{\|} \nabla_{x}$. For each integer $m\geq -1$, define
\begin{equation*}a_m=a_m(x)\equiv \|(2^k \ell (Q))^m\left(D_{n+1}\right)^{m+1} 
\nabla_y\Gamma (x,t,\cdot,0 )\varphi_k\|_2 = \left(2^k \ell (Q) \right)^m I_m^{1/2}.\end{equation*} 
Since $\Gamma (x,t,\cdot ,\cdot )$ is a solution of
$L^\ast$ away from $x,t$, we have that
\begin{equation*}(L^*_{\|})_y \,\Gamma (x,t,y,0)=\sum^n_{i=1}D_iA^\ast _{i,n+1}D_{n+1}
\Gamma +\sum^{n+1}_{j=1}
A^\ast_{n+1,j}\cdot D_{n+1}D_j\Gamma ,\end{equation*} where in the second term we have used $t$-independence. We designate
the respective contribution of these two terms to $I'_m$ by $I'_{m,1}$
and $I'_{m,2}$. Now,
\begin{equation*}\begin{split} |I'_{m,2}|&\leq C\int |\nabla_{y ,s}\left( D_{n+1}\right)^{m+2} 
\Gamma \,| \,\,|(D_{n+1})^{m+1} \Gamma \,|\,\varphi^2_k\\ &\leq C \left( \| \nabla_y\left(
D_{n+1}\right)^{m+2} \Gamma \,\varphi_k \|_2+\|
\left( D_{n+1}\right) ^{m+3}\Gamma \,\varphi_k\|_2\right) \|\left( D_{n+1}\right)^{m+1}\Gamma \,\varphi_k\|_2\\ &\leq CC^{m^2}_1 \left((2^k \ell (Q))^{-(m+1)}a_{m+1}+C_1 ^{(m+2)^2}
(2^k \ell (Q))^{-(m+2)-\frac{n}{2}}\right)(2^k\ell(Q))^{-m-\frac{n}{2}}\\
\intertext{(where we have used \eqref{eq2.5})} &\leq CC^{m^2}_1 \left(a_{m+1}(2 ^k\ell
(Q))^{-2m-1-\frac{n}{2}}+C_1^{(m+2)^2} (2 ^k\ell (Q))^{-2m-2-n}\right)\\ &\leq C\delta a^2_{m+1}(2^k \ell
(Q))^{-2m}+CC_1^{m^2+(m+2)^2}\left(\delta^{-1}+1\right) (2^k\ell (Q))^{-2m-2-n},\end{split}\end{equation*}
where $\delta >0$ is at our disposal. Also, after integrating by parts
\begin{equation*}\begin{split} I'_{m,1}&=-C\Re e \sum^n_{i=1} \int A^\ast_{i,n+1}\left(  \partial_t\right)^{m+2}\!\Gamma (x,t,y,0) \,\left(\partial_t\right)^{m+1}\overline{D_i\Gamma (x,t,y,0)}\,\varphi^2_k (y)dy\\
&-C\Re e \sum^n_{i=1} \int A^\ast _{i,n+1}\left( \partial_t\right)^{m+2}\!\Gamma (x,t,y,0) \, 
\left(
\partial_t\right)^{m+1}\overline{\Gamma (x,t,y,0)}\, D_i\varphi^2_kdy.\end{split}\end{equation*} By Cauchy's inequality,
\eqref{eq2.5} and the bound for $\| \nabla \varphi_k\|_\infty$, we obtain
\begin{equation*}|I'_{m,1}|\leq C\delta I_m +CC_1^{2 (m+1)^2}\left( \delta^{-1}+1\right)
(2^kQ)^{-2m-n-2}.\end{equation*}
Similarly, \begin{equation*}|I''_m|\leq C\delta I_m+CC_1^{2m^2}\delta^{-1}(2^k\ell (Q))^{-2m-n-2}.\end{equation*}
Collecting our estimates for $I'_{m,1}$, $I'_{m,2}$, and $I''_m$, we obtain  for $\delta$ small enough that
\begin{equation*}\left(2^k\ell (Q)\right)^{2m}I_m=a_m^2\leq C\delta a^2_{m+1}+CC_1^{2 (m+2)^2}\delta^{-1}(2^k\ell
(Q))^{-n-2}.\end{equation*} Thus,
\begin{equation*}\sum^\infty_{m=-1}C_2^{-m^2} a^2_m\leq \sum^\infty_{m=-1} C^{-m^2}_2 C\delta a^2_{m+1} +
\sum^\infty_{m=-1}C_2^{-m^2}CC^{2 (m+2)^2}_1 \delta^{-1} (2^k \ell (Q))^{-n-2}.\end{equation*} We now choose
$\delta=\delta_m=\frac{1}{2 C}C^{-2m-1}_2$, so that the right side of the last inequality equals
\begin{equation*}\frac{1}{2}\sum^\infty_{m=-1}C_2^{-(m+1)^2}a^2_{m+1}+2C\sum^\infty_{m=-1}C_2^{-m^2+2m+1}C^{
2(m+2)^2}_1 (2^k\ell (Q))^{-n-2}.\end{equation*}
Choosing now $C_2=C^3_1$, we obtain \eqref{eq2.8}, under the a priori assumption that $$\sum^\infty_{m=-1}C^{-m^2}_2
a^2_m<\infty.$$ The latter holds if $A(x)\in C^\infty$, for in that case 
$\left(\partial_t\right)^{m+1}\nabla_y\Gamma (x,t,y)$ satisfies point-wise bounds 
analogous to \eqref{eq2.5}, possibly depending
on the regularization of the coefficients. The constants in \eqref{eq2.8} and in the conclusion of Lemma~\ref{l2.7} are independent of this regularization.

The proof of estimate (ii) is similar, except that we replace the cut-off function $\varphi_k$
by $\varphi \in C_0^\infty (3Q)$, with $\varphi \equiv 1$ on $2Q$.  We omit the details.

To finish the proof of the lemma, it remains to remove the a priori assumption of smoothness of the coefficients. To this end, fix a cube $Q$, and let $g \in C^\infty_0 (Q), \, \vec{f} \in C^\infty_0 (R_k(Q),\mathbb{C}^n),$
where $R_0(Q)\equiv 2Q,$ and $R_k(Q) \equiv 2^{k+1}Q \setminus 2^kQ, \, k \in \mathbb{N}$. 
It is enough to prove the estimate
$$|\langle g, (D_{n+1})^{m+1} S_t\,(\dv_\| \vec{f}) \rangle| \leq CC_2^{m^2/2} (2^k \ell (Q) )^{-\frac{n}{2} - m - 1} \|g\|_1 \|\vec{f}\|_2,$$
with $t>0,$ and, when $k=0$, $\rho^{-1} \ell (Q) \leq t \leq \rho \ell (Q),$ with the constants depending upon $\rho$ in the latter situation.  The case $t<0$ may be handled by an identical argument, which we omit.  We define $$A_\varepsilon \equiv P_\varepsilon A \equiv \phi_\varepsilon \ast A,$$
where $\phi_\varepsilon (x) \equiv \varepsilon^{-n} \phi (x/\varepsilon ),$ and
$ \phi \in C_0^\infty ( \{ |x|<1\})$ is non-negative and even, with $\int_{\mathbb{R}^n} \phi = 1$.  
Then $A_\varepsilon \to A $ a.e..
Set $$L_\varepsilon \equiv - \dv A_\varepsilon \nabla,$$
and let $\Gamma_\varepsilon$ denote the corresponding fundamental solution.  We note that
\begin{equation*} L_\varepsilon^{-1} - L^{-1}=  L_\varepsilon^{-1} L \,L^{-1} - L_\varepsilon^{-1}
L_\varepsilon L^{-1}= L_\varepsilon^{-1} div ( A_\varepsilon -A) \nabla L^{-1}.
\end{equation*}
We choose a non-negative even cut-off function $\varphi \in C_0^\infty (-1,1)$, with $\int_{\mathbb{R}} \varphi = 1.$
Fix $t>0$ (or $t \in (\rho^{-1} \ell (Q), \rho \ell (Q) )$ if $k=0$).  For $\delta >0$, set 
$\varphi_\delta(s)\equiv \delta^{-1} \varphi(s/\delta)$, and define  
\begin{equation*}\vec{f}_\delta (y,s)
\equiv \vec{f}(y)\, \varphi_\delta(s),\,\quad
g_{t,\delta} (x,\tau)  \equiv g(x) \,\varphi_\delta(t-\tau)
\end{equation*}

Now, fix $\varepsilon > 0$ and suppose that $0<\delta<t/8$.  Then for
$|t-\tau|< \delta$, we have 
\begin{eqnarray*}\left(D_{n+1}\right)^{m+1} L_\varepsilon^{-1} \dv_\| \vec{f}_\delta(x,\tau) &=&
\iint (\partial_\tau)^{m+1} \Gamma_\varepsilon (x,\tau,y,s) \dv_\| \vec{f}(y) \,\varphi_\delta(s) \, dy ds\\
&=& \int \varphi_\delta(s) \left(D_{n+1}\right)^{m+1} 
\left(S_{\tau -s}^{L_\varepsilon} \dv_\| \vec{f}\right) (x) \,ds,
\end{eqnarray*}
where $S^{L_\varepsilon}_t$ denotes the single layer potential operator associated to $L_\varepsilon$.
Thus, 
\begin{eqnarray*} \left|\langle g_{t,\delta}, (D_{n+1})^{m+1} 
L_\varepsilon^{-1} \dv_\| \vec{f}_\delta \rangle \right| \!\!&=& \!\!
\left|\iint \! \varphi_\delta (\tau)\varphi_\delta (s) \,\langle g,(D_{n+1})^{m+1}\!
\left(S^{L_\varepsilon}_{t-(\tau +s)} \nabla_\|\right) \cdot \vec{f} \,\rangle \,ds d\tau \right|\\ \!\!&\leq &\!\! \,CC_2^{m^2/2} \, \|g\|_1 \|\vec{f}\|_2 \,(2^k \ell (Q) )^{-\frac{n}{2}-m-1}
\end{eqnarray*}
by the a priori bound obtained for smooth coefficients, since $|\tau + s | < 2\delta \leq t/4$
and $\|\varphi\|_1 = 1$.
Moreover, \begin{eqnarray*}|\langle g_{t,\delta} , (D_{n+1})^{m+1} \!\left(L_\varepsilon^{-1}-L^{-1}\right) \dv_\| \vec{f}_\delta \rangle|\!\!&=&\!\!\!\langle (D_{n+1})^{m+1}g_{t,\delta} ,  L_\varepsilon^{-1} \dv (A_\varepsilon -A) \nabla L^{-1}  \dv_\| \vec{f}_\delta \rangle| \\ \!\!&=&\! \!\!\langle \nabla \!
\left(L_\varepsilon^*\right)^{-1}\!(D_{n+1})^{m+1}g_{t,\delta} ,    
(A_\varepsilon -A) \nabla L^{-1} \! \dv_\| \vec{f}_\delta \rangle|,
\end{eqnarray*}
which converges to $0$ as $\varepsilon \to 0$, for each fixed $\delta >0$, by dominated convergence, since
$$\nabla \left(L_\varepsilon^*\right)^{-1}(D_{n+1})^{m+1}g_{t,\delta}, \,\, \nabla L^{-1}  \dv_\| \vec{f}_\delta 
\,\in L^2(\mathbb{R}^{n+1}).$$
(For the first term, the case $m=-1$ uses that $C_0^\infty \subset L^{2_*} \hookrightarrow L^2_{-1},$ where
$2_* = 2(n+1)/(n+3)$ is the lower Sobolev exponent in $n+1 \geq 3$ dimensions.)
Thus, $$|\langle g_{t,\delta} , (D_{n+1})^{m+1} L^{-1} \dv_\| \vec{f}_\delta \rangle|\leq C
C_2^{m^2/2} \, \|g\|_1 \|\vec{f}\|_2 \,(2^k \ell (Q) )^{-\frac{n}{2}-m-1}.$$
The conclusion of the lemma now follows from the observation that 
\begin{eqnarray*}\langle g_{t,\delta} , (D_{n+1})^{m+1} L^{-1} \dv_\| \vec{f}_\delta \rangle &=& \!\!\iint 
\varphi_\delta (\tau) \varphi_\delta (s)\,\langle g , (D_{n+1})^{m+1} S_{t-(\tau+s)} \dv_\| 
\vec{f} \rangle \,ds d\tau \\ & \to & \langle g , (D_{n+1})^{m+1} S_t \dv_\| \vec{f} \rangle,
\end{eqnarray*}
as $\delta \to 0$, since
$h(t) \equiv \langle g , (D_{n+1})^{m+1} S_t \dv_\| \vec{f} \rangle$
is continuous (even $C^\infty$) in $(0,\infty)$. 
\end{proof}

As a Corollary of the previous two Lemmata we deduce

\begin{lemma}\label{l2.9} Suppose that $L,L^*$ satisfy the standard assumptions, 
and let ${\bf f}: \mathbb{R}^n \to \mathbb{C}^{n+1}.$ Then for every cube $Q$ and for all integers $k\geq
1$ and $m \geq -1 ,$ we have 
\begin{enumerate}\item[(i)]$\|\partial^{m+1}_t (S_t\nabla )\cdot({\bf f}
1_{2^{k+1}Q\setminus 2 ^kQ})\|^2_{L^2(Q)} 
\leq CC_2^{m^2} 2^{-nk} (2 ^k\ell (Q))^{-2m-2} \|{\bf f}\|^2_{L^2(2^{k+1}Q\backslash 2^k
Q)}, \,\,\,\, t\in \mathbb{R}$
\item[(ii)] $\|\partial^{m+1}_t (S_t\nabla )\cdot({\bf f}
1_{2Q})\|^2_{L^2(Q)} 
\leq C_\rho^{m^2+1}  \ell (Q)^{-2m-2} \| {\bf f}\|^2_{L^2(2
Q)}, \,\,\,\, \frac{\ell (Q)}{\rho}<|t|< \rho\,\ell (Q).$\end{enumerate}
\end{lemma}

\begin{proof} We consider estimate $(i)$.  Let $x\in Q$. Then
\begin{equation*}\begin{split} &|\partial^{m+1}_t (S_t\nabla )\cdot({\bf f}1_{2^{k+1}Q\backslash 2^kQ}) (x)|^2=
\left| \int_{2^{k+1}Q\backslash 2^kQ}\partial^{m+1}_t \nabla _{y,s}\Gamma (x,t,y,s)\mid_{s=0}\cdot \,{\bf f}(y)dy\right|^2\\ &\quad
\leq \|\partial^{m+1}_t \nabla_{y,s}\Gamma (x,t,y,s)\mid_{s=0}\|^2_{L^2(2^{k+1}Q\backslash 2^k Q)}
\|{\bf f}\|^2_{L^2(2^{k+1}Q\backslash 2^kQ)}\\ &\quad \leq CC_2^{m^2} \left(2^k\ell
(Q)\right)^{-n-2m-2}\|{\bf f}\|^2_{L^2(2^{k+1}Q\backslash
2^kQ)},\end{split}\end{equation*} where in the last step we have used  Lemma~\ref{l2.7}$(i)$ and
\eqref{eq2.5}. The bound $(i)$ now follows from an integration over $Q$.
The proof of $(ii)$ is similar, and is omitted.\end{proof}

\begin{lemma}\label{l2.10} Suppose that $L,L^*$ satisfy the standard assumptions, 
and let ${\bf f}: \mathbb{R}^n \to \mathbb{C}^{n+1},\,f:\mathbb{R}^n \to \mathbb{C}.$  
Then for every $t\in \mathbb{R}$,
and for every integer $m \geq 0$, we have 
\begin{enumerate}\item[(i)] $\quad \qquad \| t^{m+1}\partial^{m+1}_t (S_t\nabla )\cdot {\bf f}\|_{L^2(\mathbb{R}^n)}\leq 
C C_2^{m^2/2}\,\|
{\bf f}\|_2$\item[(ii)]$\quad \qquad \| t^{m+1}\partial^{m+1}_t \nabla S_t  f\|_{L^2(\mathbb{R}^n)}\leq 
C C_2^{m^2/2}\,\|
f\|_2.$\end{enumerate}\end{lemma}

\begin{proof} Fix $t\in \mathbb{R}$ and $m \geq 0$. It is enough to prove $(i)$, since $(ii)$ follows by duality and the fact that $ad\!j \, S_t = S^*_{-t}$, where $S^*_t$ is the single layer potential operator associated to $L^*$.  We may further suppose that $t\neq 0$, since otherwise the left hand side of the inequality vanishes.  Set $\theta_{t}=t^{m+1}\partial^{m+1}_t (S_t\nabla)$. We write
\begin{equation*}\|\theta _{t}\, {\bf f}\|_{L^2} = \left( \sum_Q  \int_Q 
|\theta_{t}\,{\bf f}|^2\right)^{1/2} 
 =\left( \sum_Q  \fint_Q  \int_Q |\theta_{t}\,{\bf f}|^2 \right)^{1/2},
\end{equation*} where the sum runs over the dyadic grid of cubes with $\ell (Q)\approx |t|$.
With $Q$ fixed, we decompose ${\bf f}$ into ${\bf f}1_{2Q}$ plus a sum of dyadic ``annular" pieces 
$({\bf f}1_{2^{k+1}Q\backslash 2^kQ}).$  The bound $(i)$ now follows from Lemma~\ref{l2.9}.
We omit the details.
\end{proof}

The next lemma says that $$L = L_\| - \sum_{j=1}^{n+1} A_{n+1,j} \,D_{n+1} D_j - 
\sum_{i=1}^n D_i A_{i,n+1} D_{n+1}$$ in an appropriate weak sense on each ``horizontal" cross-section.

\begin{lemma}\label{l2.identity}  Let $L$ satisfy the standard assumptions of this paper.
Suppose that $Lu = g$ in the strip $a<t<b$,
where $g\in C_0^\infty(\mathbb{R}^{n+1})$.
Suppose also that $\nabla u, \nabla \partial_t u \in L^2 
(\mathbb{R}^n)$, uniformly in $t\in (a,b)$, with norms depending continuously on $t \in (a,b)$.
Then for every $F \in L^2 (\mathbb{R}^n) \cap \dot{L}^2_1(\mathbb{R}^n),$
and for all $ t \in (a,b),$ we have that
\begin{equation}\label{eqdiffidentity}\begin{split} & 
\int_{\mathbb{R}^n} A_\|(x) \nabla_x \,u(x,t) \,\nabla_x F(x) \,dx \, 
=\, \sum_{j=1}^{n+1} \int_{\mathbb{R}^n} A_{n+1,j}\,(x)\, \partial_{x_j} \partial_t \,u(x,t) \,F(x)\,dx
\\ & \qquad \qquad- \sum_{i=1}^n \int_{\mathbb{R}^n}A_{i,n+1}(x)\,\partial_t u(x,t) \,\partial_{x_i} F(x) \,dx 
 \,+\,\int_{\mathbb{R}^n} g(x,t) \,F(x) dx.\end{split}\end{equation}
\end{lemma}

\begin{proof}
Let $t\in (a,b)$, and let $\eta < \min (t-a,b-t).$  Set $\varphi_\eta (s) = \eta^{-1} \varphi \left(
s/\eta\right),$ where $\varphi \in C_0^\infty (\frac{-1}{2},\frac{1}{2}), \, 0 \leq \varphi, \, \int \varphi = 1.$  Define $$F_{t,\eta} (x,s) \equiv F(x) \varphi_\eta (t-s).$$
Then by the definition of weak solutions, and $t$-independence, we have
\begin{equation*}\begin{split} & \iint A_\|(x) \nabla_x \,u(x,s) \,\nabla_x F_{t,\eta} (x,s) \,dx ds \, 
=\, \sum_{j=1}^{n+1} \iint A_{n+1,j}\,(x)\, \partial_{x_j} \partial_t \,u(x,s) \,F_{t,\eta}(x,s)\,dx ds
\\ & \qquad \qquad- \sum_{i=1}^n \iint A_{i,n+1}(x)\,\partial_s u(x,s) \,\partial_{x_i} F_{t,\eta}(x,s) \,dx ds
 \,+\,\iint g(x,s) \,F_{t,\eta} (x,s) dx ds.\end{split}\end{equation*} 
By our hypotheses, the functions of $t$
defined by the four integrals in 
\eqref{eqdiffidentity}, are all continuous in $(a,b).$  
The conclusion of the lemma then follows if we let $\eta \to 0.$
\end{proof}

We may now prove a ``2-sided" version of Lemma \ref{l2.10}.
\begin{lemma}\label{l2.2sidegrad}  Suppose that $L,L^*$ satisfy the standard assumptions, 
and let ${\bf f}: \mathbb{R}^n \to \mathbb{C}^{n+1}.$  Then for every $t \in \mathbb{R}$, and for every integer
$m \geq 0$, we have 
$$\|t^{m+2} \nabla_\| \partial_t^{m+1} \left(S_t \nabla \right) \cdot {\bf f} \|_2 \leq C_m \|{\bf f}\|_2.$$
\end{lemma}
\begin{proof}  Fix $t \in \mathbb{R}$.  We may suppose that $t \neq 0$, since otherwise the left hand side vanishes.  By Lemma \ref{l2.10} $(ii)$ and $t-$independence, we may replace 
$\left(S_t \nabla \right) \cdot {\bf f}$ by $\left(S_t \nabla_\| \right) \cdot \vec{f}
=  - S_t \dv_\| \vec{f},$
where $\vec{f} \in C_0^\infty (\mathbb{R}^n, \mathbb{C}^n).$  It then follows from 
Lemma \ref{l2.10} $(ii)$ that
\begin{equation}\label{eq2.apriori} \beta_m(t) \equiv 
\|t^{m+2} \nabla_\| \partial_t^{m+1} \left(S_t \nabla_\| \right) \cdot \vec{f} \|^2_2
\leq C t^2 C_2^{m^2} \|\dv_\| \vec{f}\|_2^2.
\end{equation}
This last bound will not appear in our final quantitative estimates. Rather, the point is
that the left hand side is {\it a priori} finite with some (non-optimal) quantitative control.

By ellipticity, Lemma \ref{l2.identity} and Lemma \ref{l2.10} $(i)$, we have that
\begin{eqnarray*} \beta_m(t) &  \leq & C  t^{2m+4}\,\Re e\,\langle A_\| \nabla_\| \partial_t^{m+1} S_t \dv_\|
\vec{f},\nabla_\| \partial_t^{m+1} S_t \dv_\| \vec{f} \,\rangle \\
&=& C\Re e\,\sum_{j=1}^{n+1}\langle A_{n+1,j}\, t^{m+3} D_j \partial_t^{m+2} S_t \dv_\|
\vec{f},t^{m+1} \partial_t^{m+1} S_t \dv_\| \vec{f} \,\rangle \\ && \qquad \quad - \,C\Re e\,
\sum_{i=1}^{n}\langle A_{i,n+1}\, t^{m+2}  \partial_t^{m+2} S_t \dv_\|
\vec{f},t^{m+2} D_i\partial_t^{m+1} S_t \dv_\| \vec{f} \,\rangle \\
& \leq & C \delta^{-1} C_2^{m^2} \|\vec{f}\|_2^2 + C \delta C_2^{(m+2)^2} \|\vec{f}\|_2^2 + C \delta \beta_{m+1}(t) + C \delta^{-1} C_2^{(m+1)^2} \|\vec{f}\|_2 + C \delta \beta_m(t),\end{eqnarray*}
where $\delta$ is at our disposal.  Choosing $\delta$ small enough, we may hide the last term, 
so that $$\beta_m(t) \leq C \delta^{-1} C_2^{(m+2)^2}\|\vec{f}\|_2^2 + C \delta \beta_{m+1}(t).$$
Thus, taking $\delta = \delta_m = \delta_0 C_2^{-2m}$, with $\delta_0$ small, we have
\begin{eqnarray*} \sum_{m=0}^\infty C_2^{-3m} C_2^{-(m+2)^2} \beta_m(t) &\leq & C \sum_{m=0}^\infty C_2^{-3m}
\left(\delta_m^{-1} \|\vec{f}\|^2_2 +C_2^{-(m+2)^2} \delta_m \,\beta_{m+1}(t)\right) \\
&\leq & C \sum_{m=0}^\infty 
\left(\delta_0^{-1} C_2^{-m} \|\vec{f}\|^2_2 
+ \delta_0 C_2^{-3(m+1)}C_2^{-(m+3)^2} \beta_{m+1}(t)\right) \\
&\leq & C \|\vec{f}\|_2^2 + \frac{1}{2}\sum_{m=1}^\infty C_2^{-3m} C_2^{-(m+2)^2} \beta_m(t),
\end{eqnarray*}
by choice of $\delta_0$ small enough.  By \eqref{eq2.apriori}, the series converges, so the last term may be hidden on the left side of the inequality.  In particular, we conclude that
$$\beta_m(t) \leq C C_2^{(m+2)^2 + 3m} \, \|\vec{f}\|_2^2.$$
\end{proof}

\begin{lemma}\label{l2.holder} Suppose that $L,L^*$ satisfy the standard assumptions.
Fix a cube $Q \subset \mathbb{R}^n$, and suppose that $y,y' \in Q$.  For $(x,t) \in \mathbb{R}^{n+1}$, set
$$u(x,t) \equiv \Gamma(x,t,y,0) - \Gamma(x,t,y',0).$$  If $\alpha$ is the H\"{o}lder exponent in 
\eqref{eq2.6}, then for every integer $k\geq 4,$ we have
\begin{equation}\label{eq2.averageholder}\int_{2^{k+1}Q\setminus 2^kQ} |\nabla u(x,t)|^2 dx
\leq C 2^{-k\alpha} \left(2^k \ell (Q) \right)^{-n}.
\end{equation} 
\end{lemma}
\begin{proof} By \eqref{eq2.6}, it is enough to prove \eqref{eq2.averageholder} with $\nabla_x$
in place of $\nabla$.  Let $\varphi_k \in C^\infty_0(3\cdot 2^k Q \setminus 3\cdot2^{k-2} Q)$,
with $\varphi_k \equiv 1$ on $2^{k+1} Q \setminus 2^k Q$ and $ \|\nabla_x \varphi_k\|_\infty 
\leq C (2^k \ell (Q))^{-1}$.  Then the left hand side of \eqref{eq2.averageholder} is bounded by
an acceptable term involving a $t$ derivative, plus
\begin{multline*}\int |\nabla_x u(x,t)|^2 \left(\varphi_k(x)\right)^2 dx \leq   C \Re e\int A_\| \nabla_xu \cdot 
\overline{\nabla_x u } \varphi_k^2\\ =  \,C \Re e\int A_\| \nabla_xu \cdot 
\nabla_x\left(\overline{ u } \varphi^2_k\right)\, - \,C\Re e
\int A_\| \nabla_xu \cdot \nabla_x\left(\varphi_k^2\right)
\overline{ u }\, \equiv I + II.
\end{multline*}
By Lemma \ref{l2.1}, for $y,y' \in Q$ and $ x \in (2^{k-1} Q)^c$, we have
\begin{equation}\label{eq2.size}|u(x,t)| \leq C 2^{-k\alpha} \left(2^k \ell (Q) \right)^{1-n}.\end{equation}
and also
\begin{equation}\label{eq2.size2}|\partial_tu(x,t)| \leq C 2^{-k\alpha} \left(2^k \ell (Q) \right)^{-n}.
\end{equation}
Using the first of these,  we 
obtain 
\begin{multline*} |II|  \leq  C 2^{-k\alpha} \left(2^k \ell (Q) \right)^{1-n}
\|\nabla_x \varphi_k\|_\infty \int|\nabla_x u|\,\varphi_k\\
\leq C 2^{-k\alpha} \left(2^k \ell (Q) \right)^{-n/2}
 \left(\int|\nabla_x u|^2\varphi^2_k\right)^{1/2} \leq \,C \varepsilon^{-1} 2^{-2k\alpha} 
\left(2^k \ell (Q) \right)^{-n} + \varepsilon \int|\nabla_x u|^2\varphi^2_k,
\end{multline*}
where $\varepsilon$ is at our disposal.
Moreover, by Lemma \ref{l2.identity},
\begin{multline*} I =- C \,\Re e \sum_{i=1}^n \left\{ \int A_{i,n+1} \partial_tu \,\overline{D_i u} 
\varphi_k^2 + \int A_{i,n+1} \partial_t u \,D_i\left(\varphi_k^2\right) \overline{ u}\right\}\\\!\!\!\!\!\!\!\!
+  C\,\Re e \sum_{j=1}^{n+1} \int A_{n+1,j}\, D_j \partial_t u \, 
\overline{u} \,\varphi_k^2
\equiv  I_1 + I_2 + I_3. 
\end{multline*}
Now, $I_1$ satisfies exactly the same bound as term $II$, and by essentially the same argument, except that
we use \eqref{eq2.size2} in place of \eqref{eq2.size}.
Moreover, using \eqref{eq2.size}, \eqref{eq2.size2}, and the properties of $\varphi_k$, we see that
$$|I_2| \leq C 2^{-2k\alpha} \left(2^k \ell (Q) \right)^{-n}. $$
To handle $I_3$, we note first that the case $m=0$ of Lemma \ref{l2.7}$(i)$ (with the roles of $x$ and $y$
reversed), applied separately for $y$ and $y'$, implies that
$$\int|\partial_t \nabla_x u(x,t)|^2 \varphi_k^2 dx \leq C \left(2^k \ell (Q) \right)^{-n-2}.$$
Thus, using also \eqref{eq2.size}, we have
$$|I_3| \leq C 2^{-k\alpha} \left(2^k \ell (Q) \right)^{-n}. $$
Collecting these estimates, choosing $\varepsilon$ sufficiently small,
and hiding the small term on the left hand side of the inequality,
we obtain the desired bound.
\end{proof}

In the sequel, we shall find it useful to
consider approximations of the single layer potential.  The bounds in the following lemma
will not be used quantitatively, but will serve rather 
to justify certain formal manipulations.
For $\eta > 0,$ set
\begin{equation}\label{eq2.def} S_t^\eta \equiv \int_{\mathbb{R}} \varphi_\eta(t-s) \, S_s \, ds,
\end{equation}
where $\varphi_\eta \equiv \tilde{\varphi}_\eta \ast \tilde{\varphi}_\eta$, $\tilde{\varphi}_\eta
\in C_0^\infty ( -\eta/2,\eta/2)$ is non-negative and even, with $\int \tilde{\varphi}_\eta = 1$ and
$\tilde{\varphi}_\eta (t) \equiv \eta^{-1} \tilde{\varphi} (t/\eta).$

\begin{lemma}\label{l2.approx}
Suppose that $L,L^*$ satisfy the standard assumptions, and let $S_t$ denote the single layer potential operator associated to $L$.  Then for each $\eta>0$ and for every $f \in L^2(\mathbb{R}^n)$
with compact support, we have
\begin{enumerate}\item[(i)] $\quad \| \partial_t S_t^\eta f\|_{L^2(\mathbb{R}^n)}\leq 
C_{\beta,\eta} \|f\|_{L^{2n/(n+2\beta)}(\mathbb{R}^n)}\, , \,\,\, 0 < \beta < 1.$
\item[(ii)]$\quad \| \nabla_x S_t^\eta f\|_{L^2(\mathbb{R}^n)}\leq 
C_{\eta} \|f\|_{L^{(2n+2)/(n+3)}(\mathbb{R}^n)}$
\item[(iii)] $\quad \|| t \partial_t^2 S_t^\eta f\||\leq 
C_{\beta,\eta} \|f\|_{L^{2n/(n+2\beta)}(\mathbb{R}^n)}\, , \,\,\, 0 < \beta < 1.$
\item[(iv)]$\quad  \|\nabla\left(S_t^\eta - S_t \right) f \|_{L^2(\mathbb{R}^n)} \leq C \frac{\eta}{|t|}\|f\|_2,\,\,\,\,\eta < |t|/2.$
\item[(v)] $\quad \lim_{\eta \to 0} \int_\varepsilon^\infty \! \int_{\mathbb{R}^n} 
|t\nabla \partial_t\left(S_t^\eta - S_t \right) f|^2 \,\frac{dx\, dt}{t} = 0, \,\,\,\, 
0 < \varepsilon < 1.$
\item[(vi)] $\quad \text{For each cube } Q \subset \mathbb{R}^n, \,\,
\|\partial_t S_t^\eta \|_{L^2(Q) \to L^2(\mathbb{R}^n)} \leq C_{\eta,\ell(Q)}$.
\end{enumerate}\end{lemma}
\begin{proof}
$(i).$  We observe that $$\partial_t S_t^\eta f (x) = \int_{\mathbb{R}^n} k_t(x,y) f(y) dy,$$
where $k_t(x,y) \equiv \partial_t\left(\varphi_\eta \ast \Gamma(x,\cdot,y,0)\right)(t).$
Thus, by Lemma \ref{l2.1} 
\begin{equation*}|k_t(x,y)| \leq C \min \left(|x-y|^{-n}, \eta^{-1} |x-y|^{1-n} \right)
\,\leq \,C \eta^{-\beta} |x-y|^{\beta - n}, \,\,\,\, 0 < \beta < 1.
\end{equation*}
Estimate $(i)$ now follows by the fractional integral theorem.

\smallskip
\noindent $(ii).$  We first note that
\begin{eqnarray*}S_t^\eta f(x) 
&=& \iint \!\! \int \Gamma (x,t-s-\sigma ,y,0) f(y) dy \tilde{\varphi}_\eta (s) 
\tilde{\varphi}_\eta (\sigma)ds d\sigma \\
&=& \int \left(L^{-1} f_\eta \right)(x,t-\sigma) \tilde{\varphi}_\eta (\sigma) d\sigma,
\end{eqnarray*}
where $f_\eta(y,s) \equiv f(y) \tilde{\varphi}_\eta (s).$  
Let $\vec{g} \in C_0^\infty ( \mathbb{R}^n, \mathbb{C}^n),$
with $\|\vec{g}\|_2 = 1,$ and set $\vec{g}_\eta (x,\sigma ) \equiv \vec{g}(x) \tilde{\varphi}_\eta (\sigma).$  Then
\begin{multline*}
|\langle \vec{g}, \nabla_x S_t^\eta f \rangle | 
=\left| \iint \dv_x \vec{g}_\eta (x,\sigma ) \overline{\left(L^{-1} f_\eta \right)}(x,t-\sigma) dx d\sigma 
\right| \\
\leq \|\vec{g}_\eta\|_{L^2(\mathbb{R}^{n+1})}\|\nabla L^{-1} f_\eta\|_{L^2(\mathbb{R}^{n+1})}
\, \leq \, C \eta^{-1/2} \|f_\eta\|_{L^{2_\ast}(\mathbb{R}^{n+1})}\equiv  C \eta^{-1/2}
\|\varphi_\eta\|_{L^{2_\ast}(\mathbb{R})}\|f\|_{L^{2_\ast}(\mathbb{R}^{n})},
\end{multline*}
where $2_\ast = (2n+2)/(n+3),$ since $L^{2_\ast}(\mathbb{R}^{n+1})\hookrightarrow \dot{L}^2_{-1}(\mathbb{R}^{n+1}) \equiv \left(\dot{L}^2_{1}(\mathbb{R}^{n+1})\right)^*,$
and $\nabla L^{-1} \dv : L^2(\mathbb{R}^{n+1}) \to L^2(\mathbb{R}^{n+1}).$  Estimate $(ii)$ now follows.

\smallskip
\noindent
$(iii).$  We proceed as for estimate $(i)$, and write
$$t\partial_t^2 S_t^\eta f (x) = \int_{\mathbb{R}^n} h_t(x,y) f(y) dy,$$
where $h_t(x,y) \equiv t\partial_t^2\left(\varphi_\eta \ast \Gamma(x,\cdot,y,0)\right)(t),$
so that, by Lemma \ref{l2.1},
\begin{equation*}|h_t(x,y)| \leq  C t\min \left(|x-y|^{-n-1}, \eta^{-2} |x-y|^{1-n} \right)
\,\leq \,C t\, \eta^{-1-\beta} |x-y|^{\beta - n}, \,\,\,\, 0 < \beta < 1.
\end{equation*}
Moreover, if $t > 2\eta,$ we have the sharper estimate
$$|h_t(x,y)| \leq C \frac{t}{\left(t+|x-y|\right)^{n+1}} \leq C t^{-\beta} |x-y|^{\beta-n}, 
\,\,\,\,0<\beta <1.$$  Thus,
$$\||t\partial_t^2 S_t^\eta f \||^2 \leq C \left(\int_0^{2\eta}\eta^{-2-2\beta} t dt + \int_{2\eta}^\infty t^{-1-2\beta} dt\right) \|f\|^2_{L^{2n/(n+2\beta)}(\mathbb{R}^n)},$$
and $(iii)$ follows.

\smallskip
\noindent
$(iv).$  We suppose that $\eta <|t|/2$.  Then $$\|\nabla\left(S_t^\eta - S_t \right) f \|_{L^2(\mathbb{R}^n)}
\leq \varphi_\eta \ast \|\nabla\left(S_{(\cdot)} - S_t \right) f \|_{L^2(\mathbb{R}^n)}.$$
But for $|s-t|<\eta<|t|/2,$ we have by the mean value theorem and
Lemma \ref{l2.10}$(ii)$ that
$$\|\nabla\left(S_s - S_t \right) f \|_{L^2(\mathbb{R}^n)} \leq \frac{\eta}{|t|} \sup_{|\tau-t|<|t|/2} 
\|\tau\nabla\partial_\tau S_\tau f \|_{L^2(\mathbb{R}^n)}\leq C\frac{\eta}{|t|} \|f\|_2.$$

\smallskip
\noindent
$(v).$  We take $\eta<\varepsilon /2$,
and write
\begin{eqnarray}\nonumber \int_\varepsilon^\infty \! \int_{\mathbb{R}^n} 
|t\nabla \partial_t\left(S_t^\eta - S_t \right) f|^2 \,\frac{dx\, dt}{t}\!\!&=& \!\!
\int_\varepsilon^\infty \! \int_{\mathbb{R}^n} 
|\varphi_\eta \ast t\nabla D_{n+1}\left(S_{(\cdot)} - S_t \right) f|^2 \,\frac{dx\, dt}{t}\\
\!\!& \leq &  \!\!\int_\varepsilon^\infty \varphi_\eta \ast \| 
t\nabla D_{n+1}\left(S_{(\cdot)} - S_t \right) f\|_2^2 \,\frac{dt}{t}
\label{eq2.last}
\end{eqnarray}
We claim that the last expression converges to $0$, as $\eta \to 0$.
Indeed, for $|s-t|<\eta<t/2,$ we have that
$$
\|t\nabla D_{n+1}\left(S_s - S_t \right) f \|_{L^2(\mathbb{R}^n)} \leq  \eta \sup_{|\tau-t|<t/2} 
\|\tau\nabla\partial^2_\tau S_\tau f \|_{L^2(\mathbb{R}^n)} \leq  C\frac{\eta}{t} \|f\|_2 $$
by Lemma \ref{l2.10}$(ii)$.  Thus, for $\eta < \varepsilon/2$, 
\eqref{eq2.last} is bounded by $C \eta^2 \varepsilon^{-2} \|f\|_2^2$, and the claim follows.

\smallskip
\noindent $(vi)$.  Estimate $(vi)$ follows from $(i)$ and H\"{o}lder's inequality.
\end{proof}

\section{Some consequences of ``off-diagonal" decay estimates \label{s3offdiag}}

Here, we prove some estimates that hold in general for operators satisfying the conclusions of Lemmas~\ref{l2.9} and
\ref{l2.10}. For the sake of notational convenience, we observe that part (i) of 
the former conclusion can be reformulated as
\begin{equation}
\Vert\theta_{t}(f1_{2^{k+1}Q\backslash 2^{k}Q})\Vert_{L^{2}(Q)}^{2}\leq
C_{m}2^{-nk}\left(\frac{|t|}{2^{k}\ell(Q)}\right)^{2m+2}\Vert
f\Vert^2_{L^{2}(2^{k+1}Q\backslash 2^{k}Q)}\label{eq2.11}\end{equation}
 where $\theta_{t}=t^{m+1}\partial_{t}^{m+1}(S_{t}\nabla).$ We now consider generic operators $\theta_{t}$ which
satisfy~\eqref{eq2.11} for some integer $m\geq 0$.  The next lemma is essentially due to 
Fefferman and Stein~\cite{FS}. We
omit the well known proof.

\begin{lemma}\label{l2.12} Suppose that $\{\theta_{t}\}_{t\in \mathbb{R}}$ is a family of
operators which satisfies \eqref{eq2.11} for some integer $m \geq 0$ and
in every cube $Q$, whenever $|t|\leq C\ell (Q)$. If $\| |
\theta_{t}|\|_{op}\leq C$, then
\begin{equation*}|\theta_{t}b(x)|^2\frac{dxdt}{|t|}\end{equation*} is a Carleson measure for every $b\in
L^\infty$.\end{lemma}

\begin{lemma}\label{l2.13} Suppose that 
$\{\theta_{t}\}_{t\in \mathbb{R}}$ is a
family of operators satisfying \eqref{eq2.11} for some integer $m \geq 0$, as well as the bound
\begin{equation*}\sup_{t\in \mathbb{R}}\| \theta_{t}f\|_{L^2(\mathbb{R}^n)}\leq C\|f\|_2.\end{equation*} Suppose
that $\{ \Lambda_t\}_{t\in \mathbb{R}}$ is a family of operators satisfying the bounds
\begin{equation*} \sup_{t\in \mathbb{R}}\| \Lambda_t f\|_2\leq C\| f\|_2,\qquad
 \|\Lambda_t f\|_{L^2(E)}\leq C\exp  \left\{ \frac{-\dist (E,E')}{C|t|}\right\} \| f\|_{L^2(E')}
 \end{equation*} whenever
(in the latter estimate) support $f\subseteq E'$. Then $\theta_{t}\Lambda_t$ 
also satisfies \eqref{eq2.11}, whenever $|t|\leq
C\ell(Q)$.\end{lemma}

\begin{proof} We may suppose that $k\geq 4$, otherwise, subdivide $Q$ dyadically to reduce to this case. Given $Q$, set
$\tilde{Q} \equiv 2^{k-2}Q$. Then
\begin{equation}\label{eq2.14} \theta_{t}\Lambda_t=\theta_{t} 1_{\tilde{Q}}\Lambda_t
+\theta_{t}1_{\mathbb{R}^n\backslash \tilde{Q}}\Lambda_t .\end{equation} For the first term, 
we have the bound
\begin{equation*}\begin{split} \| \theta_{t} 1_{\tilde{Q}} \Lambda_t \,(f1_{2^{k+1}Q\backslash
2^kQ})\|_{L^2(Q)} & \leq \, \| \theta_{t}\|_{2\to 2} \| \Lambda_t \,(f1_{2^{k+1}Q\backslash 2^kQ})\|
_{L^2(\tilde{Q})}\\ &\leq \, \| \theta_{t}\|_{2\to 2}\exp 
\left\{ \frac{-2^k\ell (Q)}{C|t|}\right\} \| f\|_{L^2
(2^{k+1}Q\backslash 2^kQ)}\end{split}\end{equation*} which in particular yields \eqref{eq2.11} for this
term, if $|t|\leq C\ell(Q)$. Next, we consider the second term in
\eqref{eq2.14}, which equals
\begin{equation*}\sum_{j\geq k-2} \theta_{t}
1_{2^{j+1}Q\backslash 2^jQ}\Lambda_t.\end{equation*} The
desired bound follows for this term by applying \eqref{eq2.11} for each $j$ fixed, and summing the resulting geometric series.\end{proof}

\begin{lemma}\label{l2.15} {\bf (i)}. Suppose that $\{ R_t\}_{t\in \mathbb{R}}$ is a 
family of operators satisfying  \eqref{eq2.11}, for some $m\geq 1$, and for 
all $|t|\leq C\ell (Q)$, and suppose also that $\sup_{t}
\|R_t\|_{2\to2}\leq C$, and that $R_t1=0$ for all $t\in \mathbb{R}$ (our hypotheses allow $R_t1$ to be defined as an
element of $L^2_{\loc }$). Then for $h\in \dot{L}^2_1 (\mathbb{R}^n)$,
\begin{equation}\label{eq2.15a}\int_{\mathbb{R}^n}|R_th|^2\leq Ct^2\int_{\mathbb{R}^n}|\nabla_x h|^2.\end{equation}
{\bf (ii)}. If, in addition, $\|R_t \dv_x\|_{2 \to 2} \leq C/|t|,$ then also
\begin{equation}\label{eq2.15b} \||R_t f\|| \leq C \|f\|_2.\end{equation}\end{lemma}

\begin{proof} We suppose that $t>0$, and show that \eqref{eq2.15a} implies \eqref{eq2.15b}.  
The latter follows from 
\begin{equation}\label{eq3.orthog}\|R_t \left(s^2 \Delta \,e^{s^2 \Delta}\right) \|_{2\to 2} 
\leq C \, \min \left(\frac{s}{t},\frac{t}{s}\right),\end{equation}
by a standard orthogonality
argument.  
In turn, \eqref{eq3.orthog} is easy to prove:  
the case $t<s$ is just \eqref{eq2.15a}, and the case
$s<t$ follows by hypothesis 
from the factorization $\Delta = \dv_x \nabla_x$.

We now turn to the proof of \eqref{eq2.15a}.  Let $\mathbb{D}(t)$ denote the grid of dyadic cubes with $\ell (Q)\leq |t|<2 \ell (Q)$.   For convenience of notation we set $m_Qh\equiv \fint_Q h$.  Then 
\begin{multline*} \left( \int_{\mathbb{R}^n}|R_th|^2\right)^{\frac{1}{2}} =\left(\sum_{Q\in
\mathbb{D}(t)}\int_Q |R_th|^2\right)^{\frac{1}{2}}=\left( \sum_{Q\in \mathbb{D}(t)}\int_Q |R_t
(h-m_{2Q}h)|^2\right)^{\frac{1}{2}}\\
\leq\left(\sum_{Q\in \mathbb{D}(t)}\int_Q |R_t[(h-m_{2 Q}h) 1_{2 Q}]|^2\right)^{\frac{1}{2}} 
+\left(\sum_{Q\in \mathbb{D}(t)}\int_Q |R_t[(h-m_{2 Q}h)1_{(2
Q)^c}]|^2\right)^{\frac{1}{2}}\,\equiv I + II.\end{multline*} Since
$R_t:L^2\to L^2$, we have by Poincar\'e's inequality that
\begin{equation*}I\leq C\left( \sum_{Q\in \mathbb{D}(t)}\int_{2 Q} |h-m_{2
Q}h|^2\right)^{\frac{1}{2}} \leq C|t|\left(
\sum_{Q\in \mathbb{D}(t)}\int_{2 Q}|\nabla_x h|^2\right)^{\frac{1}{2}}\leq C|t|\,\| \nabla_xh\|_2.\end{equation*}Moreover, we are given that $R_t$ satisfies \eqref{eq2.11}.   Thus,
\begin{multline*}II\leq \sum^\infty_{k=1} \left( \sum_{Q\in \mathbb{D}(t)} \int_Q
|R_t[(h-m_{2 Q}h)1_{2^{k+1}Q\backslash 2 ^kQ}]|^2\right)^{\frac{1}{2}}\\\leq 
C\sum^\infty_{k=1} \left(\sum_{Q\in \mathbb{D}(t)} \! 2^{-k(n+4)} \!\int_{2^{k+1}Q}\!|h-m_{2
Q}h|^2\right)^{\frac{1}{2}}\\\leq C\sum^\infty_{k=1}\sum^k_{j=1} 
\left( \sum_{Q\in \mathbb{D}(t)}2^{-4k}2^{-jn}
\int_{2^{j+1}Q}|h-m_{2^{j+1}Q}h|^2\right)^{\frac{1}{2}},\end{multline*}
where in the last step we have used that
\begin{equation*}h-m_{2 Q} h=h-m_{2^{k+1}Q}h + m_{2^{k+1}Q}h-m_{2^kQ}h+\dots -\dots +
m_{4Q}h-m_{2Q}h.\end{equation*}   
By Poincar\'{e}'s inequality, since $j\leq k$ we
obtain in turn the bound
\begin{multline*}C|t|\sum^\infty_{k=1} 2^{-k}\sum^k_{j=1}\left( \sum_{Q\in \mathbb{D}(t)}
2^{-jn}\int_{2^{j+1}Q} |\nabla_x h|^2\right)^{\frac{1}{2}}\leq
C|t|\sum^\infty_{k=1}2^{-k}\sum^k_{j=1}\left( \sum_{Q\in \mathbb{D}(t)}\int_Q
\fint_{2^{j+1}Q}|\nabla_x h|^2\right)^{\frac{1}{2}}\\  \leq
C|t|\sum^\infty_{k=1}2^{-k}\sum^k_{j=1}\left( \int_{\mathbb{R}^n}\fint_{|x-y|\leq C2^jt}|\nabla_x
h(x)|^2dxdy\right)^{\frac{1}{2}}  = C|t|\,\|\nabla_x h\|_2 .\end{multline*} 
\end{proof}

\begin{lemma}\label{l2.17} Given $\{R_t\}_{t\in \mathbb{R}_+}$ as in part {\bf (i)} of
the previous lemma, we have that
\begin{equation*}\| | t^{-1} R_tF|\|\leq C\| \nabla_x F\|_{L^2(\mathbb{R}^n)},\end{equation*} provided that
$\left|\frac{1}{t}R_t\Phi (x)\right|^2\frac{dxdt}{|t|}$ is a Carleson measure, where $\Phi (x)\equiv x$.\end{lemma}

\begin{proof} We may assume that $F \in C_0^\infty $, and that $t>0$.  Let $\mathbb{D}_j$ denote the dyadic grid of cubes of side length $2^{-j}$. Then
\begin{equation}\begin{split}\label{eq2.18} \| | t^{-1}R_t F| \|^2&= \sum^\infty_{j=-\infty}
\sum_{Q\in \mathbb{D}_{-j}}\int^{2^{j+1}}_{2^j}\int_Q | t^{-1} R_tF(y)|^2 dy\frac{dt}{t}\\
&=\sum^\infty_{j=-\infty} \sum_{Q\in \mathbb{D}_{-j}}\int^{2^{j+1}}_{2^j} \int_Q \fint_Q |
t^{-1}R_t F(y)|^2 dydx\frac{dt}{t}.\end{split}\end{equation} 
We now use an idea taken from \cite{J} and \cite[pp. 32-33]{Ch2}. 
For $(x,t)$ fixed, set $$G_{x,t}(z) \equiv F(z) - F(x) - (z-x)\cdot P_t\left(\nabla_\|F\right)(x),$$
where as usual $P_t$ is an approximate identity.
Since $R_t 1=0$, we have, for any fixed $x$,
\begin{equation*}\frac{1}{t}R_t F(y)=\frac{1}{t}R_t\left(G_{x,t}\right)(y) + 
\frac{1}{t} R_t\Phi (y)\cdot P_t \left(\nabla_{\|}F\right)(x) \equiv I+II.\end{equation*}
The contribution of $II$ to \eqref{eq2.18} is bounded by
\begin{multline*}  \sum^\infty_{j=-\infty} \sum_{Q\in \mathbb{D}_{-j}}\int^{2^{j+1}}_{2
^j}\int_Q |P_t\left(\nabla_{\|}F\right)(x)|^2\fint_Q \left|\frac{1}{t}R_t\Phi (y)\right|^2dydx\frac{dt}{t}\\  \leq C
\int^\infty_0 \int_{\mathbb{R}^n} |P_t\left(\nabla_{\|}F\right)(x)|^2\left\{ \fint_{B(x,Ct)} \left|
\frac{1}{t}R_t\Phi (y)\right|^2 dy\right\} dx\frac{dt}{t} \leq C\| \nabla_{\|}F\|_{L^2(\mathbb{R}^n)}^2\|
\mu\|_{\mathcal{C}},\end{multline*} 
by Carleson's Lemma, where
\begin{multline*} 
\| \mu\|_{\mathcal{C}}\equiv \sup_Q \int^{\ell (Q)}_0 \fint_Q \left\{ \fint_{B(x,Ct)} \left|
\frac{1}{t} R_t \Phi (y)\right|^2 dy\right\} dx\frac{dt}{t}\\ \leq C\sup_Q \int^{\ell (Q)}_0 \fint_{CQ} \left|
\frac{1}{t}R_t\Phi (y)\right|^2 \fint_{|x-y|\leq Ct}dx dy\frac{dt}{t}\leq C\sup_Q \int^{\ell (Q)}_0 \fint_Q
\left|\frac{1}{t} R_t\Phi\right|^2
\frac{dxdt}{t}.\end{multline*}

Next we consider the contribution of $I$ to \eqref{eq2.18}. For $Q\in \mathbb{D}_{-j}$, and $x\in Q$, we have 
\begin{equation*}I=
R_t\left(t^{-1}G_{x,t}1_{2Q}\right)(y) + \sum_{k=1}^\infty R_t\left(t^{-1}G_{x,t}1_{2^{k+1}Q\setminus
2^kQ}\right)(y) 
\equiv I_0 +\sum^\infty_{k=1}I_k.\end{equation*} 
Since $R_t:L^2\to L^2$, we
obtain the bound
\begin{equation*}\begin{split}\| |I_0|\|^2&\leq C\sum^\infty_{j=-\infty} \sum_{Q\in \mathbb{D}_{-j}}
\int_{2^j}^{2^{j+1}}\int_Q
\fint_{2 Q} \frac{|G_{x,t}(y)|^2}{t^2} dydx\frac{dt}{t}\\ &\leq C\int^\infty_0
\int_{\mathbb{R}^n} (\beta (x,t))^2\frac{dxdt}{t}\,\leq \,C\| \nabla_{\|}F\|_{L^2
(\mathbb{R}^n)},\end{split}\end{equation*} where
$(\beta (x,t))^2=\fint_{|x-y|<Ct}t^{-2}|G_{x,t}(y)|^2
dy,$ and where the last step 
is a well known consequence of Plancherel's Theorem, see, e.g.
\cite[pp. 32-33]{Ch2} or \cite[pp. 249-250]{H}. 
Furthermore, since $R_t$ satisfies \eqref{eq2.11} for some $m\geq 1$, whenever $t\approx \ell (Q)$, we have that
\begin{multline*} C^{-1} \sum^\infty_{k=1} \| | I_k|\| \leq
 \sum^\infty_{k=1}\left( \sum^\infty_{j=-\infty}
\sum_{Q\in \mathbb{D}_j} \int_{2^j}^{2^{j+1}}\!\!\!\int_Q \frac{1}{t^n 2^{k(n+4)}} \int_{|x-y|\leq
C2^kt} \!\! \frac{ |G_{x,t}(y)|^2}{t^2} \frac{dydxdt}{t}\right)^{\frac{1}{2}}\\
=\sum^\infty_{k=1}2^{-k}\left( \int^\infty_0\int_{\mathbb{R}^n} \fint_{|x-y|\leq C 2^k t}
\frac{|G_{x,t}(y)|^2}{(2^kt)^2} dydx\frac{dt}{t}\right)^{\frac{1}{2}} 
\equiv \sum^\infty_{k=1} 2^{-k}\| |\beta_k |\|,\end{multline*} where,
after making the change of variable $t\to t/ 2^k$, 
\begin{equation*}\beta_k (x,t)=\left(\fint_{|x-y |\leq Ct} \frac{|F(y)-F(x)-(y-x)\cdot P_{2^{-k}t}\left(\nabla_\|
F\right)(x)|^2}{t^2} dy\right)^{1/2}.\end{equation*}
We now claim that $\| |\beta _k|\|\leq C \sqrt{k} \|
\nabla_{\|}F\|_2,$ from which the conclusion of the lemma trivially follows. 
By Plancherel's Theorem, the definiton of $P_t$ and the change of variable $x-y=h$ we have
\begin{equation*}\| | \beta_k |\|^2=\int^\infty_0 \fint_{|h|<Ct} 
\int_{\mathbb{R}^n}\frac{|e^{i\xi \cdot h}-1-(ih\cdot \xi)\,
\hat{\phi} (2^{-k}t\xi )|^2}{t^2|\xi |^2} |\xi|^2 |\hat{F}(\xi )|^2 d\xi dh\frac{dt}{t},\end{equation*} 
where $\phi \in
C^\infty_0 \{ |x|<1\}$ and $\int \phi \equiv 1$. By the change of variable $h\to th$, we have
\begin{equation*}\| | \beta_k|\|^2 = \int^\infty_0 \fint_{|h|<C} 
\int_{\mathbb{R}^n}\frac{|e^{i\xi \cdot ht}-1-(iht\cdot
\xi)\,
\hat{\phi } (2^{-k}t\xi )|^2}{t^2|\xi |^2} |\xi|^2|\hat{F}(\xi )|^2\frac{d\xi dh dt}{t}.\end{equation*} Since $\hat{\phi
}\in \mathcal{S}$ and $\hat{\phi} (0)=1$, we have that 
\begin{equation*}\frac{|e^{it\xi \cdot h}-1-(ih\cdot t\xi)\,\hat{\phi} (2^{-k}t\xi )|}{t|\xi |} \leq C \min \left(
t|\xi |,1,
\frac{2^k}{t|\xi |}\right).\end{equation*}  Indeed, if $t|\xi | \leq 1,$ then
\begin{equation*}\begin{split}\frac{|e^{it\xi \cdot h}-1-(ih\cdot t\xi)\,\hat{\phi} (2^{-k}t\xi )|}{t|\xi |}&\leq \frac{|e^{it\xi \cdot h}-1-ih\cdot t\xi|}{t|\xi |}\, + \,\frac{|ih\cdot t\xi\,\left(1-\hat{\phi} (2^{-k}t\xi )\right)|}{t|\xi |}\\&\leq C( t|\xi| + 2^{-k} t|\xi|)\leq
Ct|\xi|.
\end{split}\end{equation*}
On the other hand, if $t|\xi| > 1,$ then
$$\frac{|e^{it\xi \cdot h}-1|}{t|\xi |}\leq \frac{2}{t|\xi|},$$
and $$\frac{|(ih\cdot t\xi)\,\hat{\phi} (2^{-k}t\xi )|}{t|\xi |}
\leq C |\hat{\varphi} (2^{-k}t\xi )| \leq \frac{C}{1 + 2^{-k}t|\xi|}\leq C \min\left(1,\frac{2^k}{t|\xi|}\right).$$
We then obtain
the bound $\| | \beta_k|\|^2  \leq C k \|
\nabla_{\|}F\|^2_2$ as claimed. \end{proof}

\begin{lemma}\label{l2.19} Suppose that $\theta_t$ satisfies \eqref{eq2.11} for some $m\geq 0$, whenever $0<t\leq C\ell
(Q)$ and that $\|\theta_t\|_{2 \to 2}\leq C$. Let $b\in L^\infty (\mathbb{R}^n)$, and let 
$\mathcal{A}_t$ denote a
self-adjoint averaging operator whose kernel $\varphi_t(x,y)$ satisfies 
$|\varphi_t(x,y)|
\leq Ct^{-n}  \, 1_{\{|x-y|<Ct\}}, \, \,\varphi_t \geq 0, \, \,\int \varphi_t(x,y) dy =1. $  Then 
\begin{equation*} \sup_{t\in \mathbb{R}_+}\| (\theta_t b)\mathcal{A}_t f\|_2\leq C\| b\|_\infty \| f\|_2.\end{equation*}\end{lemma}

\begin{proof} Since we do not assume that $\theta_t:L^\infty \to L^\infty$, this requires a bit of an argument. Observe that 
\begin{equation*}\| (\theta_t b)\mathcal{A}_t f\|^2_2\,\leq
 \,\| f\|_2\| \mathcal{A}_t(|\theta_t b|^2\mathcal{A}_t f)\|_2\,\leq\, \| f\|^2_2
\| \mathcal{K}_t (x,\cdot )\| _{L^1(\mathbb{R}^n)},\end{equation*} where 
$\mathcal{K}_t(x,y)$ is the kernel 
of the self-adjoint operator $f \to \mathcal{A}_t(|\theta_t b|^2 \mathcal{A}_t f)$, i.e., 
$$\mathcal{K}_t(x,y) = \int_{\mathbb{R}^n} \varphi_t(x,z) |\theta_t b (z)|^2 \varphi_t(z,y) dz.$$
Consequently,
$$\| \mathcal{K}_t (x,\cdot )\| _{L^1} = \int_{\mathbb{R}^n} \varphi_t(x,z) |\theta_t b (z)|^2  dz \leq Ct^{-n}
\int_{|x-z|<Ct} |\theta_t b (z)|^2  dz.$$
Thus, by \eqref{eq2.11} and the fact that $\theta_t$ is bounded on $L^2$ uniformly in $t$, we have that
$$\| \mathcal{K}_t (x,\cdot )\|^{1/2} _{L^1} \leq 
C \left\{\left(\fint_{Q(x,4Ct)} |b|^2\right)^{1/2} + \sum_{k=2}^{\infty}
2^{-k}\left(\fint_{Q(x,2^{k+1}Ct)\setminus Q(x,2^{k}Ct)} |b|^2 \right)^{1/2}\right\} \leq C \|b\|_{\infty} ,$$
where $Q(x,Rt)$ is the cube centered at $x$ with side
length $Rt$. This proves the lemma.\end{proof}

\begin{lemma}\label{l2.20} Suppose that
\begin{equation*}\Omega _t=\int^t_0 \left( \frac{s}{t}\right)^\delta 
W_{t,s}\,\theta_s \frac{ds}{s},\end{equation*}
for some $\delta >0$, where $\sup_{t,s}\| W_{t,s}\|_{2\to 2}\leq C$. Then
\begin{equation*}\| |\Omega _t| \|_{op}\leq C\| |\theta_s |\|_{op}.\end{equation*}\end{lemma}

\begin{proof} This is a standard Schur type argument. Indeed, if $\| | G(x,t)|\| \leq 1$, then
\begin{equation*}\begin{split} & \left| \int^\infty_0 \int_{\mathbb{R}^n} G(x,t)\Omega_t f(x)dx\frac{dt}{t}\right| = \left| \int^\infty_0 \int^\infty_0 \int_{\mathbb{R}^n}1_{\{s<t\}} 
\left( \frac{s}{t}\right)^\delta G(x,t)W_{t,s}
\theta_s f(x)dx\frac{dt}{t}\frac{ds}{s}\right|\\ &\quad \leq \left( \int^\infty_0 \! \int_{\mathbb{R}^n} |G(x,t)|^2
\int^t_0\left( \frac{s}{t}\right)^\delta
\frac{ds}{s}\frac{dxdt}{t}\right)^{\frac{1}{2}} \left( C\int^\infty_0 \! \int_{\mathbb{R}^n} |\theta_s f(x)|^2\int^\infty_s
\left(\frac{s}{t} \right)^\delta \frac{dt}{t} \frac{dx ds}{s} \right)^{\frac{1}{2}}\\ 
&\quad \leq C\| | \theta_s
f|\|.\end{split}\end{equation*}\end{proof}

\section{Traces, jump
relations, and uniqueness \label{s4nt}}
We begin by proving a useful technical lemma.
\begin{lemma}\label{l4.nt0} Let $L,L^*$ satisfy the standard assumptions.  Suppose that $Lu=0$ and 
that $\widetilde{N}_* (\nabla u) \in
L^2 (\mathbb{R}^n)$. Then
\begin{equation}\label{eq4.nt0}\sup_{t>0}\|\nabla u(\cdot,t)\|_2 \leq 
C \|\widetilde{N}_*(\nabla u)\|_2. \end{equation}
\end{lemma}
\begin{proof} The desired bound for $\partial_t u$ follows readily from
$t$-independence and \eqref{eq1.3}.  Thus, we need only consider
$\nabla_x u$. Let $\vec{\psi} \in C_0^\infty (\mathbb{R}^n,\mathbb{C}^n)$, with $\|\vec{\psi}\|_2 =1$.
For $t_0 > 0$ fixed, it will then be enough to establish the bound
$$\left|\int_{\mathbb{R}^n} u(\cdot,t_0) \dv_x \vec{\psi} \,\right|\leq C 
\| \widetilde{N}_* (\nabla u)\|_2.$$
To this end, we write
\begin{eqnarray*}\int_{\mathbb{R}^n} u(\cdot,t_0) \dv_x \vec{\psi}  & = & 
\int_{\mathbb{R}^n} \left(u(x,t_0)- \fint_{t_0/2}^{t_0} u(x,t) dt\right)\dv_x \vec{\psi}(x) dx \\ &&
\qquad  \quad+ \int_{\mathbb{R}^n} \fint_{t_0/2}^{t_0} u(x,t) \dv_x \vec{\psi}(x) \,dt\, dx 
\,\equiv \,I + II. 
\end{eqnarray*}
We first observe that $$|II| =  \left|\int_{\mathbb{R}^n} \fint_{t_0/2}^{t_0} \left(\fint_{|x-y|<t} dy\right)\,
\nabla_x u(x,t) \, \vec{\psi}(x) \,dt\, dx \right| \leq  C  \| \widetilde{N}_* (\nabla u)\|_2,$$
by Cauchy-Schwarz and Fubini's Theorem.  Moreover,
\begin{multline*}|I| = \left|\int_{\mathbb{R}^n} \fint_{t_0/2}^{t_0} 
\int_t^{t_0} \partial_s u(x,s) \,ds \dv_x \vec{\psi}(x) \,dt\, dx \right|\\
= \left| \fint_{t_0/2}^{t_0} 
\int_t^{t_0} \int_{\mathbb{R}^n} \nabla_x \partial_s u(x,s) \, \vec{\psi}(x) \, dx\, ds\, dt\right|
\leq  C  \,t_0\left( \fint_{t_0/2}^{t_0} 
\int_{\mathbb{R}^n} |\nabla \partial_s u(x,s)|^2 dx\, ds \right)^{1/2}\\\leq C 
\left(\fint_{t_0/2}^{t_0} 
\int_{\mathbb{R}^n} | \partial_s u(x,s)|^2 dx\, ds\right)^{1/2},\end{multline*}
where in the last step we have split $\mathbb{R}^n$ into cubes of side length
$\approx t_0$ and used Caccioppoli's inequality.  The conclusion of the lemma follows
since the bound already holds for $\partial_s u$.
\end{proof}

We now discuss some trace results.  The following lemma is the
analogue of Theorem 3.1 of \cite{KP}.  We recall that
$u \to f \, n.t.$ means that
$\lim_{(y,t) \to (x,0)} u(y,t) = f(x),$
for $a.e. \, x \in \mathbb{R}^n$,
where the limit runs over $(y,t)\, \in \gamma(x).$
As usual, 
$P_\varepsilon$ will denote a self-adjoint approximate identity acting in
$\mathbb{R}^n$. 
We shall denote by $W^{1,2}_c$ the subspace of compactly supported elements
of the usual Sobolev space $W^{1,2}$.
\begin{lemma} \label{l4.ntconverge} Suppose that $L,L^*$ satisfy the standard assumptions. 
If $Lu=0$ in $\mathbb{R}^{n+1}_+$ and $\widetilde{N}_* (\nabla u) \in
L^2 (\mathbb{R}^n)$, then there exists $f\in \dot{L}^2_1 (\mathbb{R}^n)$
such that
\begin{enumerate}\item[(i)]  $\|\nabla_\| f\|_2 \leq C \|\widetilde{N}_* (\nabla u)\|_2$, and
$u\to f \,n.t.$, with 
$|u(y,t) - f(x)| \leq C t \widetilde{N}_*(\nabla u)(x)$ whenever
$(y,t)\in \gamma(x).$
\item[(ii)] $ \nabla_\|u(\cdot,t) \to \nabla_\| \,f$ weakly in $L^2(\mathbb{R}^n)$
as $t \to 0$.
\end{enumerate}
If $Lu=0$ in $\mathbb{R}^{n}\times (0,\rho)$, where $0<\rho\leq\infty$, and $\sup_{0<t<\rho}\|\nabla u(\cdot,t)\|_{L^2(\mathbb{R}^n)} <\infty$,
then there exists $g \in L^2(\mathbb{R}^n)$ such that $g = \partial u/\partial\nu$
in the variational sense, i.e.,
\begin{enumerate}
\item[(iii)] $\iint_{\mathbb{R}^{n+1}_+} A\nabla u\cdot \nabla \Phi \,dx\, dt = \int_{\mathbb{R}^n} g \,\Phi 
\, dx, \,\,\,\,\,\,\forall \, \Phi \in W^{1,2}_c(\mathbb{R}^{n}\times(-\rho,\rho))$.
\item[(iv)] $\vec{N} \cdot A \nabla u(\cdot,t) \to g$ weakly in $L^2(\mathbb{R}^n)$
as $t \to 0$.
\end{enumerate}
(Here, $\vec{N} \equiv -e_{n+1}$ is the unit outer normal to $\mathbb{R}^{n+1}_+$).
\end{lemma}
Of course, the analogous results hold for the lower half space.
\begin{proof} The existence of $f\in \dot{L}^2_1 (\mathbb{R}^n)$ satisfying $(i)$ may be obtained by 
following {\it mutatis mutandi} the corresponding argument in \cite{KP} pp. 461-462.

\smallskip
\noindent $(ii)$.  We first 
establish convergence in the sense of distributions.
Let $\vec{\psi} \in C_0^\infty (\mathbb{R}^n,\mathbb{C}^n)$.  Then by $(i)$,
$$\left|\int_{\mathbb{R}^n} \big(\nabla_\|u(\cdot,t) - \nabla_\|f\big)\, \vec{\psi}\right|   = \left|\int_{\mathbb{R}^n} \big(u(\cdot,t) - f\big)\,\dv_\| \vec{\psi}\right|
\leq C t \,  \|\widetilde{N}_*(\nabla u)\|_2 \|\dv_\| \vec{\psi}\|_2\to 0.$$
By the density of $C^\infty_0$ in $L^2$, the weak convergence in $L^2$
then follows readily from \eqref{eq4.nt0}.

\smallskip
\noindent $(iii)$.  We follow \cite{KP}, with some modifications owing to the
unboundedness of our domain.  We treat only the case $\rho = \infty$, and leave it to the reader to check the details in the case of finite $\rho$.  Fix $0<R<\infty $ and set 
$B_R = B(0,R) \equiv \{X \in \mathbb{R}^{n+1}:  |X|<R \} 
, \, B_R^{\pm} \equiv B_R \cap \mathbb{R}^{n+1}_{\pm}$ and $\Delta_R = B_R \cap \{ t = 0\}.$
Define a linear functional
on $W^{1,2}_0 (B_R)$ (the closure of $C_0^\infty$ in 
$W^{1,2} (B_R)$) by
$$\Lambda_R(\Psi) \equiv \iint_{B_R^+} A \nabla u \cdot \overline{\nabla \Psi}, \,\,\,
\Psi \in W^{1,2}_0 (B_R).$$
Clearly, $\|\Lambda_R\| \leq C R^{1/2}\, \sup_{t>0}\|\nabla u(\cdot,t)\|_2.$
By trace theory, $tr \left( W^{1,2}_0(B_R) \right)\subset H^{1/2}_0 (\Delta_R)$,
defined as the closure in $H^{1/2}(\mathbb{R}^n)$ of $C_0^\infty (\Delta_R)$.
Here, $\|f\|_{H^{s}(\mathbb{R}^n)} \equiv \|f\|_{L^2(\mathbb{R}^n)}
+\| \,|\xi |^s \hat{f}\|_{L^2(\mathbb{R}^n)}$, for $0\leq s \leq 1$.
On the other hand, suppose that $\psi \in H^{1/2}_0(\Delta_R).$  We extend 
$\psi$ to $ \psi_{ext} \in W^{1,2}_0(B_R)$ by solving the problems 
\begin{equation}
\begin{cases} \sum_{i=1}^{n+1} \partial^2_{x_i} \psi_{ext}^{\pm}=0\text{ in }B_R^{\pm}\\ 
\psi_{ext}^{\pm}|_{\Delta_R} = \psi,\,\,\, 
\psi_{ext}^{\pm}|_{\partial B_R^{\pm}\cap \mathbb{R}^{n+1}_{\pm}}=0 
\end{cases}\tag{D+,D-}\label{D+,D-}\end{equation}
We set $\psi_{ext} \equiv \psi_{ext}^+ 1_{B_R^+} + \psi_{ext}^- 1_{B_R^-},$
and by standard theory of harmonic functions we have 
$$\|\nabla\psi_{ext}\|_{L^{2}(B_R)} \leq C \|\psi\|_{H^{1/2}(\Delta_R)}.$$
Thus, we may define a bounded linear functional on $H^{1/2}_0 (\Delta_R)$ by
$\Xi_R (\psi) \equiv \Lambda_R (\psi_{ext})$.  Since $\Lambda_R(\Upsilon)
= 0$ whenever $\Upsilon \in W^{1,2}_0 (B^+_R)$, then $\Xi_R (\psi) = \Lambda_R (\Psi)$
for {\it every} extension $\Psi \in W^{1,2}_0 (B_R)$ with $tr(\Psi) = \psi.$
Thus, there exists  a unique $g_R \in H^{-1/2} (\Delta_R)$ with
$$\iint_{B_R^+} A \nabla u \cdot \overline{\nabla \Psi} = \langle g_R, tr(\Psi)\rangle,
\,\,\,\forall\,\Psi \in W^{1,2}_0 (B_R).$$ 
Now suppose that $R_1 < R_2,$ and construct $g_{R_k}$ corresponding to
$B_k \equiv B(0,R_k), k = 1,2$.  Then, since $W^{1,2}_0 (B_1) \subset W^{1,2}_0 (B_2)$
(if we extend elements in the former space to be $0$ outside of $B_1$),
we have that $g_{R_1} = g_{R_2}$ in $H^{-1/2}(\Delta_{R_1}).$
Thus, $\langle g_{R_1},\psi \rangle = \langle g_{R_2},\psi \rangle,$
whenever $\psi \in H^{1/2}_c (\mathbb{R}^n)$, and $B_1,B_2$ contain
the support of $\psi.$  It follows 
that $g \equiv \lim_{R\to 0} g_R$ exists in the 
sense of distributions, and that
\begin{equation}\label{eq4.ntgauss}\iint_{\mathbb{R}^{n+1}_+} A \nabla u \cdot \overline{\nabla \Psi} = \langle g, tr(\Psi)\rangle,\,\,\,\forall\,\Psi \in W^{1,2}_c (\mathbb{R}^{n+1}).\end{equation} 
To complete the proof of $(iii)$, it remains only to establish that $g \in L^2$.
The bound
$$\|g\|_2 \leq C \sup_{t>0}\|\nabla u(\cdot,t)\|_2$$
will be an immediate consequence of $(iv)$, to which we now turn our attention.
 
 \smallskip
 \noindent $(iv)$.   Again we present only the case $\rho=\infty$.
 Since $\sup_{t>0}\|\nabla u(\cdot,t)\|_2 <\infty$, 
 it is enough to verify the weak convergence for
 test functions in $C_0^\infty$.  
 Let $\Psi \in C_0^\infty (\mathbb{R}^{n+1}), \, \psi \equiv \Psi|_{\{t=0\}}$.
 By \eqref{eq4.ntgauss}, 
 it is enough to show that $$\int_{\mathbb{R}^n} \vec{N}\cdot A\nabla u(\cdot,t)\psi
 \to \iint_{\mathbb{R}^{n+1}_+} A \nabla u \cdot \nabla \Psi ,$$ as $t\to 0$.
 Integrating by parts, we see that for each $\varepsilon > 0$,
 \begin{equation}\label{eq4.nt5}\int_{\mathbb{R}^n} \vec{N}\cdot 
 P_\varepsilon(A\nabla u(\cdot,t))\psi
 = \iint_{\mathbb{R}^{n+1}_+}P_\varepsilon\left( A \nabla u(\cdot,t+s)\right)(x) 
 \cdot \nabla \Psi (x,s) dx ds ,\end{equation}
 since $Lu=0$ and our coefficients are $t$-independent.   By dominated convergence,
 we may pass to the limit as $\varepsilon \to 0$ in \eqref{eq4.nt5} to obtain
  \begin{equation}\label{eq4.nt6}\int_{\mathbb{R}^n} \vec{N}\cdot A\nabla u(\cdot,t)\psi
 = \iint_{\mathbb{R}^{n+1}_+} A(x) \nabla u(x,t+s) 
 \cdot \nabla \Psi (x,s) dx ds ,\end{equation}
 It therefore suffices to show that 
 $$\iint_{\mathbb{R}^{n+1}_+} A(x) \big(\nabla u(x,t+s)-\nabla u(x,s)\big)
 \cdot \nabla \Psi (x,s) dx ds = O\left(\sqrt{t}\right) ,\,\,
 \text{ as } t \to 0.$$  To this end, let $R$ denote the
 radius of a ball centered at the origin which contains the support of $\Psi$.  We split the integral into
 $\int_0^{2t} \int_{\{|x|<R\}}  + \int_{2t}^R \int_{\{|x|<R\}}.$
 Since $\sup_{t>0}\|\nabla u(\cdot,t)\|_2 <\infty$, 
 the first of these contributes at most $O(t)$,
 while the second is dominated by 
 $$C \|\nabla \Psi \|_2\, t \left (\int_{t}^R \|\nabla \partial_s u(\cdot,s)\|_{L^2(\mathbb{R}^n)}^2 ds
 \right)^{1/2} \leq C_{\Phi} t \left(\int_{t/2}^\infty \frac{ds}{s^{2}}\right)^{1/2 }
 \sup_{t>0}\|\nabla u(\cdot,t)\|_2  , $$ 
 where in the last step we have used Caccioppoli's inequality in Whitney cubes in the
 $1/2$-space.   The desired conclusion follows.
  \end{proof}

Next we discuss the boundedness of non-tangential maximal functions
of layer potentials.  We recall that
$S_t^\eta$ is defined in \eqref{eq2.def},
and that $P_t$ denotes a smooth approximate identity acting in $\mathbb{R}^n$. 
In the sequel, given an operator $T$, we shall use the notation
\begin{equation}\label{eq4.ntlocaloperatornorm}
\|T\|_{op,Q} \equiv \|T\|_{L^2(Q)\to L^2(\mathbb{R}^n)}
\equiv \sup \frac{\|Tf\|_{L^2(\mathbb{R}^n)}}{\|f\|_{L^2(Q)}},
\end{equation}
where the supremum runs over all $f$ supported in $Q$ with
$\|f\|_2 >0$.
\begin{lemma}\label{l4.nt1}  Let $L,L^*$ satisfy the standard assumptions.
Then for $1<p<\infty$, we have
\begin{enumerate}\item[(i)] $\|N_*(\partial_tS_tf)\|_p \leq C_p \left(\sup_{t>0} \|\partial_t S_t \|_{p\to p}
+1 \right)\|f\|_p.$
\item[(ii)]$
\|\widetilde{N}_*(\nabla S_tf)\|_p \leq C_p \left(\sup_{t>0} \|\nabla_x S_t f \|_{p} +
\|N_*(\partial_tS_tf)\|_p \right).$
\item[(iii)] $\|N_*\left(P_t (\nabla S_t f)\right)\|_p \leq 
C_p \left(\sup_{t>0} \|\nabla_x S_t f \|_{p} +
\|N_*(\partial_tS_tf)\|_p \right).$
\item[(iv)]$ \sup_{t_0\geq0}\|N_*\left(P_t\partial_tS^\eta_{t+t_0}f\right)\|_2
\leq C \left(\sup_{t>0} \|\partial_t S^\eta_t \|_{op,Q}
+1 \right)\|f\|_2, \,\,\eta>0,\,\,\supp f \subset Q.$
\item[(v)]$\|N_*\left((S_t\nabla )\cdot 
{\bf f}\right)\|_{L^{2,\infty}} \leq C \left(\sup_{t>0} \| (S_t\nabla ) \|_{2\to 2}
+1 \right)\|{\bf f}\|_2.$
\item[(vi)]$\|N_*\left(\mathcal{D}_tf\right)\|_{L^{2,\infty}} \leq C \left(\sup_{t>0} \| (S_t\nabla ) \|_{2\to 2}
+1 \right)\|f\|_2.$\end{enumerate}
where $L^{2,\infty}$ denotes the usual weak-$L^2$ space.
\end{lemma}
\begin{proof}
By Lemma \ref{l2.1}, the kernel $K_t(x,y) \equiv \partial_t \Gamma(x,t,y,0)$
is a standard Calder\'{o}n-Zygmund kernel with bounds independent of $t$.  We may then prove $(i)$ 
by a familiar argument involving Cotlar's inequality for maximal singular integrals.  We omit the details (but see the proof of $(iv)$ below, which is similar).  
Estimate $(ii)$ may be obtained by following 
the argument in \cite{KP}, p. 494 (again we omit the details) and $(vi)$ follows from $(v)$.  It remains to prove $(iii),\,(iv)$ and$(v)$.

\smallskip
\noindent $(iii)$.  The proof is similar to that of estimate $(ii)$, and we follow \cite{KP}.
Fix $x_0 \in \mathbb{R}^n,$ and suppose that $|x-x_0|<t$.
It is enough to replace $\nabla$ by $\nabla_\|$.
We have \begin{eqnarray*}P_t\left(\nabla_\| S_t f\right)(x) &=& \nabla_x P_t (S_t f )(x)
\equiv  t^{-1} \vec{Q}_t (S_t f)(x) \\
& = & t^{-1} \vec{Q} _t\left(\int_0^t \partial_s S_s f ds + 
S_0 f - \fint_{\Delta_{2t}(x_0)} S_0f \right) (x) 
\end{eqnarray*}
where we have used that $t \nabla_x P_t \equiv \vec{Q}_t$ annihilates constants.
But $$\left|\vec{Q} _t\left(t^{-1}\int_0^t \partial_s S_s f ds\right)(x)\right|\leq 
C M\left(N_*(\partial_s S_s f)\right)(x_0),$$
and, by Poincar\'{e}'s inequality, $$\left|t^{-1} \vec{Q} _t\left( 
S_0 f - \fint_{\Delta_{2t}(x_0)} S_0f \right) (x)\right| \leq CM(\nabla_\| S_0 f)(x_0). $$

\smallskip
\noindent $(iv)$.  We suppose that $\eta << \ell(Q)$, and that $Q$ is centered at 0,
as it is only this case that we shall encounter in the sequel.  We shall deduce $(iv)$ as a consequence of the following refinement
of Cotlar's inequality for maximal singular integrals.  Let $T$ be a singular integral operator associated to a standard Calder\'{o}n-Zygmund kernel $K(x,y)$. As usual, we define truncated singular integrals
$$T_\varepsilon f (x) \equiv \int_{|x-y|>\varepsilon} K(x,y) f(y) dy,$$
and we define a maximal singular integral
$$T_{*}^R f \equiv \sup_{0<\varepsilon <R} |T_{\varepsilon} f|.$$
We claim that the following holds for all $f$ supported in a cube $Q$:
\begin{equation}\label{eq4.ntcotlar}
T_{*}^{\ell(Q)} f(x) \leq C \left(C_K + \|T\|_{op,Q} \right)Mf(x) + CM(Tf)(x),
\end{equation}
where $C_K$ depends on the Calder\'{o}n-Zygmund kernel conditions.  Momentarily 
taking this claim for granted, we proceed to prove $(iv)$.

Let $K_{t}^\eta(x,y)$ denote the 
kernel of $\partial_t S_{t}^\eta$ (see \eqref{eq2.def}), i.e.,
$$K_{t}^\eta(x,y) \equiv \partial_t\left(\varphi_\eta\ast \Gamma(x,\cdot,y,0)\right)(t).$$
Then by Lemma \ref{l2.1} we have for all $t\geq 0,$ uniformly in $t_0\geq 0$,
\begin{eqnarray}|K_{t+t_0}^\eta(x,y)|\leq C\left(\frac{1_{|x-y|+t>40\eta}}{(t+|x-y|)^n} + 
\frac{1_{|x-y|+t<40\eta} }{\eta |x-y|^{n-1}}\right)\qquad\quad\label{eq4.ntkernel1}\\
\label{eq4.ntkernel3}
\qquad|K_{t+t_0}^\eta(x+h,y) -K_{t+t_0}^\eta(x,y)|\leq 
C \frac{|h|^\alpha}{(t + |x-y|)^{n+\alpha}},\,\,\,\,|x-y|+t>10\eta
\end{eqnarray}
where the last bound holds whenever $|x-y| > 2|h|$ or $2t > |h|$.  
Of course, we also have a similar estimate
concerning H\"{o}lder continuity in the $y$ variable.  In particular, $K_{t+t_0}^\eta(x,y)$ is a standard
Calder\'{o}n-Zygmund kernel, uniformly in $t,\,t_0$ and $\eta$.

We begin by showing that for each fixed $x_0 \in \mathbb{R}^n$ and $t_0 \geq 0$,
\begin{equation}\label{eq4.ntvertical} N_*\left(P_t \partial_t S_{t+t_0}^\eta f\right)(x_0) \leq \sup_{t>0}
|\partial_t S_t^\eta f(x_0)| + CM(Mf)(x_0).
\end{equation}
To see this, let $|x-x_0|<t$, and note that
\begin{equation*}|P_t(\partial_t S_{t+t_0}^\eta f)(x) -\partial_t S_{t+t_0}^\eta f(x_0)| \leq 
C t^{-n} \int_{|x_0-z|<2t} \int_{\mathbb{R}^n} |K^\eta_{t+t_0}(z,y) 
- K^\eta_{t+t_0}(x_0,y)||f(y)| dy dz,
\end{equation*}
for which, in the case $t>10\eta$, we obtain immediately the bound $CMf(x_0)$
by applying  \eqref{eq4.ntkernel3}.  In the case $t\leq 10 \eta$, we split the inner integral into
$$\int_{|x_0-y| >10\eta}+\int_{|x_0-y| \leq 10\eta}\leq CMf(x_0) + C\left(Mf(z)+Mf(x_0)\right),$$
where we have applied \eqref{eq4.ntkernel3} to bound the first term,
and \eqref{eq4.ntkernel1} to handle the second. The estimate \eqref{eq4.ntvertical} 
now follows readily.

Next, we observe that for $f$ supported in a cube $Q$ centered at $0$, with $\ell(Q) >> \eta$,
\begin{equation}\label{eq4.ntvertical2}
\sup_{t>0}|\partial_t S_t^\eta f(x)|\leq\sup_{0<t<\ell(Q)}|\partial_t S_t^\eta f(x)| +CMf(x). 
\end{equation}
Indeed, suppose that $t\geq \ell(Q)>>\eta$. Then
$$|\partial_t S_t^\eta f (x)| \leq \int |K_t^\eta(x,y)f(y)|dy \leq CMf(x),$$
by \eqref{eq4.ntkernel1}, since for $y\in Q$, we have $|x-y|\approx |x|$, if $|x|> Ct$,
and $|x-y|<Ct$, if $|x|< Ct.$

Combining \eqref{eq4.ntvertical} and \eqref{eq4.ntvertical2}, we see that it is enough to treat
$\sup_{0<t<\ell(Q)}|\partial_t S_t^\eta f(x)|$.  To this end,
fix $x_0$ and $t\in (0,\ell(Q))$, and set $\rho \equiv \max (t,2\eta)$.
Then \begin{eqnarray*}
\partial_t S_t^\eta f(x_0) &=& \int_{|x_0-y|>5\rho} \left(K_t^\eta(x_0,y) -K_0^\eta(x_0,y)\right)f(y)dy
\\&&\qquad+\int_{|x_0-y|\leq5\rho} K_t^\eta(x_0,y) f(y)dy -
\int_{5\rho>|x_0-y|>\rho} K_0^\eta(x_0,y)f(y)dy \\
&&\qquad\qquad \qquad + \int_{|x_0-y|>\rho} K_0^\eta(x_0,y)f(y)dy \equiv I + II + III +IV.
\end{eqnarray*}
Then $|I|+|II|+ |III| \leq CMf(x_0)$, by Lemma \ref{l2.1} and by \eqref{eq4.ntkernel1}.
Also,
\begin{eqnarray*}|IV| &\leq & \sup_{0<\varepsilon<\ell(Q)} 
\left|\int_{|x_0-y|>\varepsilon}K_0^\eta(x_0,y) f(y) dy \right|.\\
\end{eqnarray*}
Thus, taking $T$ in \eqref{eq4.ntcotlar} to be the singular integral operator with kernel
$K_0^\eta(x,y)$, we obtain $(iv)$, modulo the proof of \eqref{eq4.ntcotlar}.

We now turn to the proof of \eqref{eq4.ntcotlar}.  The argument is a variant of the standard one.
Suppose that $f$ is supported in a cube $Q$, and fix $\varepsilon \in (0,\ell(Q))$
and $x_0 \in \mathbb{R}^n$.  Set $\Delta \equiv \Delta_{\varepsilon/2}(x_0),\,
2\Delta \equiv \Delta_\varepsilon (x_0).$  Let $f_1 \equiv f1_{2\Delta},\, f_2 \equiv f - f_1.$
Then for $x \in \Delta$, we have
\begin{eqnarray*}|T_\varepsilon f(x_0)| =|T f_2 (x_0)| &=&|Tf_2(x_0) -Tf_2(x) +Tf(x) - Tf_1(x)|\\
&\leq & C_KMf(x_0) + |Tf(x)| + |Tf_1(x)|.\end{eqnarray*}
Let $r\in (0,1)$, and take an $L^r$ average of this last inequality over $\Delta$.
Note that $f_1 = 0$ unless $2\Delta \subset 5Q$, since  
$\text{diam} (2\Delta) \leq 2\ell(Q).$
We therefore obtain 
\begin{eqnarray*}|T_\varepsilon f(x_0)| &\leq& C_KMf(x_0) + M(|Tf|^r)^{1/r}(x_0) + 
\left(\fint_\Delta |Tf_1|^r\right)^{1/r}\\
&\leq & C\left(C_K + \|T\|_{L^1(Q) \to L^{1,\infty}(5Q)}\right)Mf(x_0) + M(Tf)(x_0),\end{eqnarray*}
where we have used Kolmogorov's weak-$L^1$ criterion, and $L^{1,\infty}$ is the usual weak-$L^1$ space.  But by a localized version of the Calder\'{o}n-Zygmund Theorem,
$$\|T\|_{L^1(Q) \to L^{1,\infty}(5Q)} \leq C\left(C_K + \|T\|_{L^2(Q)\to L^2(5Q)}\right)\leq
C\left(C_K + \|T\|_{L^2(Q)\to L^2(\mathbb{R}^n)}\right),$$
and \eqref{eq4.ntcotlar} follows.

\smallskip
\noindent $(v)$.  By $(i)$ and $t$-independence, we may replace $\nabla$ by
$\nabla_x$.  The desired estimate is an immediate consequence of the following pointwise bound. 
For convenience of notation set ${\bf K} \equiv \sup_{t>0}\|\left(S_t\nabla_\|\right)\|_{2\to 2}.$
Let $\vec{f} \in C_0^\infty (\mathbb{R}^n,\mathbb{C}^n).$ 
We shall prove\footnote{The bound for the last term in \eqref{eq4.nt2} may be improved to
$\big(M(|\vec{f}\,|^q)\big)^{1/q}(x)$, for some $q<2$ depending on dimension and ellipticity,
as the fourth named author will show in a forthcoming paper with M. Mitrea.}
\begin{equation}\label{eq4.nt2}
N_* \big( (S_t\nabla_x)\cdot\vec{f}\,\big)(x) \leq C \left(M\big( (S_t|_{t=0}\nabla_x)\cdot\vec{f}\,\big)(x) +
({\bf K} + 1 )\big(M(|\vec{f}\,|^2)\big)^{1/2}(x)\right)
\end{equation}
To this end, we fix $(x_0,t_0) \in \mathbb{R}^{n+1}$ and suppose that $|x_0 - x| < 2t, |t_0-s|<2t$
and that $k\geq 4$.  We claim that 
\begin{equation}\label{eq4.nt3}
\int_{2^kt \leq |x_0-y|<2^{k+1}t} |\nabla_y\big(\Gamma(x,s,y,0) - 
\Gamma(x_0,t_0,y,0)\big)|^2dy\leq C 2^{-k\alpha}(2^kt)^{-n}.
\end{equation}
Indeed, the special case $s=t_0$ is essentially a reformulation of 
Lemma \ref{l2.holder}, but with the roles of $x$ and $y$ reversed.  In general, we write
\begin{equation*}
\Gamma(x,s,y,0) - \Gamma(x_0,t_0,y,0) = \big\{\Gamma(x,s,y,0) - 
\Gamma(x_0,s,y,0)\big\} + \big\{\Gamma(x_0,s,y,0) - 
\Gamma(x_0,t_0,y,0)\big\}.
\end{equation*}
The first expression in brackets is the case $s=t_0$, while the horizontal gradient of the second
equals $$\int_{t_0}^s \nabla_y \partial_\tau \Gamma(x_0,\tau,y,0) d\tau.$$
We may handle the contribution of the latter term via Lemma \ref{l2.7}. This proves the claim.

We set $u(\cdot,t) \equiv (S_t \nabla_\|) \cdot\vec{f},$ and we split $u=u_0 + \sum_{k=4}^\infty u_k\equiv
u_0 + \tilde{u},$
where $$u_0 \equiv   (S_t \nabla_\|) \cdot\vec{f}_0,\,\,\,u_k \equiv  (S_t \nabla_\|) \cdot\vec{f}_k,
\,\,\,\tilde{u} \equiv \sum_{k=4}^\infty u_k,$$
and $\vec{f}_0 \equiv \vec{f} \, 1_{\{|x_0 - \cdot|< 16t\}},\, \vec{f}_k = \vec{f}\, 1_{\mathcal{R}_k}, $ and
$\mathcal{R}_k \equiv \{y: 2^kt \leq |x_0-y|<2^{k+1}t\}.$ 
By \eqref{eq4.nt3}, for $s\in [-2t,2t]$ and $|x_0-x|<2t,$ we have that
$$|u_k(x,s)-u_k(x_0,0)|  \leq  C 2^{-k\alpha/2} 
\left(\fint_{\mathcal{R}_{k}} |\vec{f}|^2 \right)^{1/2} \leq C2^{-k\alpha/2} 
\left(M (|\vec{f}|^2)\right)^{1/2}(x_0).$$
Summing in $k$, we obtain
\begin{equation}\label{eq4.nt7}
|\tilde{u}(x,s)-\tilde{u}(x_0,0)| \leq C \left(M (|\vec{f}|^2)\right)^{1/2}(x_0).
\end{equation}
Moreover, since $Lu_0 = 0$, by \eqref{eq1.3} it follows that
\begin{eqnarray*}|u_0(x,t)| \, \leq \, C \left(\fint\!\fint_{B\left((x,t),t/2\right)} |u_0|^2 \right)^{1/2}   &\leq &
Ct^{-n/2} \sup_{\tau > 0} \|\left( S_\tau \nabla_\|\right)\cdot\vec{f}_0 \|_2 \\& \leq & 
C{\bf K}\left(M (|\vec{f}|^2)\right)^{1/2}(x_0).
\end{eqnarray*}
Taking $s=t$ in
\eqref{eq4.nt7},  we therefore need only establish the bound
\begin{equation}\label{eq4.nt8}
|\tilde{u}(x_0,0)|\leq C({\bf K} +1)\left(M (|\vec{f}|^2)\right)^{1/2}(x_0) +
CM\left(u(\cdot,0)\right)(x_0)
\end{equation}
The proof of \eqref{eq4.nt8} is based on that of the well known Cotlar inequality for maximal singular integrals.   Set $\Delta_0 = \{|x-x_0|<t\},$ and let $x \in \Delta_0.$  We write 
\begin{eqnarray*}|\tilde{u}(x_0,0)| & \leq &|\tilde{u}(x,0) -\tilde{u}(x_0,0)| +|\tilde{u}(x,0)| \\
& \leq & |\tilde{u}(x,0) -\tilde{u}(x_0,0)| +|u_0(x,0)|+|u(x,0)|\\
& \leq & C\left(M (|\vec{f}|^2)\right)^{1/2}(x_0) + |u_0(x,0)|+|u(x,0)|,
\end{eqnarray*}
where in the last step we have used \eqref{eq4.nt7} with $s = 0$.  Averaging
over $\Delta_0$, we obtain
$$|\tilde{u}(x_0,0)| \leq C\left(M (|\vec{f}|^2)\right)^{1/2}(x_0) +
\left(\fint_{\Delta_0} |u_0(x,0)|^2 dx\right)^{1/2}+M\left(u(\cdot,0)\right)(x_0).$$
Since the $L^2$ average of $u_0$ is
bounded by $C {\bf K}\left(M (|\vec{f}|^2)\right)^{1/2}(x_0),$ we obtain \eqref{eq4.nt8}.
\end{proof}

We are now ready to discuss the jump relations  and traces of the layer potentials.
We recall that $S^*_t, \mathcal{D}_t^*$  denote 
the single and double layer potentials associated to $L^*$.
\begin{lemma}\label{l4.ntjump} Suppose that 
$L,L^*$ satisfy the standard assumptions, and  
that the single layer potentials $S_t, S_t^*$ satisfy  
\begin{equation}\label{eq4.ntbounded}
\sup_{t\neq 0}\|\nabla S_t \|_{2\to 2} +
\sup_{t\neq 0}\|\nabla S^*_t \|_{2\to 2}  < \infty.
\end{equation} 
Then there exist $L^2$ bounded operators 
$K, \widetilde{K}, \mathcal{T}$ with the following properties:  for all
$f\in L^2(\mathbb{R}^n)$, we have
\begin{enumerate}
\item[(i)] $\left(\pm \frac{1}{2}I + \widetilde{K}\right)f = \partial_\nu u^\pm $
\end{enumerate} where $u^\pm \equiv S_t f, \,t \in \mathbb{R}_\pm,$
and $\partial_\nu$ denotes the conormal derivative $-e_{n+1} \cdot A \nabla$, interpreted in the weak sense of Lemma \ref{l4.ntconverge} $(iii)$ and $(iv)$.
\begin{enumerate}  
\item[(ii)]$\mathcal{D}_{\pm s}f \to \left(\mp\frac{1}{2}I + K \right)f \,\,\, $ weakly in $L^2$
\item[(iii)]  $\left(\nabla S_{t}\right)|_{t = \pm s} f 
\to \left(\mp\frac{1}{2A_{n+1,n+1}}e_{n+1} + \mathcal{T}\right)f\,\,
\, $ weakly in $L^2$.
\end{enumerate} \end{lemma}

\begin{proof} It is enough to prove $(i)$.  Indeed, if we 
define $$K \equiv ad\!j\left(\widetilde{K^*}\right),$$
then $(ii)$ follows from $(i)$ and the observation that  $\mathcal{D}_s = 
ad\!j\left(\vec{N}\cdot A^*\nabla S^*_{t}\right)|_{t = -s}.$ 
To obtain $(iii)$, we first use \eqref{eq4.ntbounded}, Lemma \ref{l4.nt1}, 
Lemma \ref{l4.ntconverge} and the formula
\begin{equation}\label{eq4.ntDSidentity} -A_{n+1,n+1} \partial_t 
S_t =\vec{N}\cdot A \nabla S_t +\sum_{j=1}^n A_{n+1,j} D_j S_t \end{equation}
to deduce that $\partial_t S_t f$ converges weakly in $L^2$, as $t \to 0$.  
Thus, we may define $$\mathcal{T} f \equiv tr\left(\nabla S_t f \right).$$
Then $(iii)$ follows from \eqref{eq4.ntDSidentity} since
$\nabla_\| S_t f$ does not jump across the boundary.

To prove $(i)$, we
apply Lemma \ref{l4.ntconverge} $(iii)$ in both $\mathbb{R}^{n+1}_{\pm}$,
to obtain $g^{\pm} \in L^2(\mathbb{R}^n)$, with $g^{\pm} =
\partial_\nu u^{\pm}$ in the weak sense.
We now define\footnote{We are indebted to M. Mitrea for suggesting this approach.}
 $\widetilde{K}$ by 
\begin{equation}\label{eq4.ntjump}
\left(\frac{1}{2} I + \widetilde{K}\right)f \equiv g^+,\,\,\,\left(-\frac{1}{2} I + \widetilde{K}\right)f \equiv g^-,
\end{equation}
and to show that this operator is well defined, we need only verify that
$g^+ - g^- = f$.  It is enough to prove that
\begin{equation}\label{eq4.ntconormal}
\iint_{\mathbb{R}^{n+1}_{+}} A \nabla u^{+} \cdot \nabla \Psi dx dt + 
\iint_{\mathbb{R}^{n+1}_{-}} A \nabla u^{-} \cdot \nabla \Psi dx dt= \int_{\mathbb{R}^n}
f \Psi dx, \end{equation}  for all $\Psi \in C_0^\infty(\mathbb{R}^{n+1}).$
To this end, set 
$u_\eta^{\pm} \equiv S_t^\eta f$, where $S_t^\eta$ is defined in \eqref{eq2.def},
so that $$u_\eta^{\pm} =\iint_{\mathbb{R}^{n+1}} \Gamma(x,t,y,s) f_\eta (y,s) dy ds,
\,\,\,t\in \mathbb{R}_{\pm}$$ where  $f_\eta (y,s) \equiv f(y) \varphi _\eta (s)$ and
$\varphi_\eta$ is the kernel of a smooth approximate identity acting in $1$ dimension.  
Let $U_\eta \equiv 
u_\eta^+ 1_{\mathbb{R}^{n+1}_+} + u_\eta^- 1_{\mathbb{R}^{n+1}_-}$.
Since $L\Gamma = \delta$, we have that
\begin{eqnarray*}\iint_{\mathbb{R}^{n+1}_+} A\nabla u_\eta^+ \cdot \nabla \Psi +
\iint_{\mathbb{R}^{n+1}_-} A\nabla u_\eta^- \cdot \nabla \Psi &=& 
\iint_{\mathbb{R}^{n+1}} A\nabla U_\eta \cdot \nabla \Psi\\
&=&\iint_{\mathbb{R}^{n+1}}  f_\eta  \Psi \to\int_{\mathbb{R}^n} f \Psi ,
\end{eqnarray*}
as $\eta \to 0.$  On the other hand, fixing $\varepsilon$ momentarily, we have that
$$\iint_{\mathbb{R}^{n+1}_+} A\nabla (u_\eta^+ - u^+) \cdot \nabla \Psi \,=\,
\int_\varepsilon^\infty \!\int_{\mathbb{R}^n}\,\, +\,\,\int_0^\varepsilon \!\int_{\mathbb{R}^n} 
\equiv I_\varepsilon + II_\varepsilon . $$
Fix a number $R$ greater than the diameter of supp$(\Psi)$.  Then 
$$|I_\varepsilon| \leq C_\Psi \int_\varepsilon^R \sup_{\varepsilon < t < R} \|\nabla (S_t^\eta
-S_t)f \|_{L^2(\mathbb{R}^n)} \to 0$$
as $\eta \to 0$, by Lemma \ref{l2.approx}.  Moreover,
$$\sup_{\eta >0}|II_\varepsilon| \leq C_\Psi\, \varepsilon \sup_{t \neq 0} \|\nabla S_t f \|_2
\leq C_\Psi \,\varepsilon \|f\|_2,$$
where we have used that $\sup_{\eta > 0} \|\nabla S_t^\eta f \|_2 \leq \sup_t \|\nabla S_t f\|_2$,
by construction of $S_t^\eta$ \eqref{eq2.def}.  The analogous convergence result for the lower 
half-space concludes the proof of $(i)$.
\end{proof}

We turn now to the issues of non-tangential  and strong $L^2$ convergence for $\mathcal{D}_t$.

\begin{lemma}\label{l4.ntstrongconvergence}  Suppose that $L,L^*$ satisfy the standard assumptions,
that the single layer potentials $S_t, S_t^*$ satisfy \eqref{eq4.ntbounded}, and that 
$S^*_0\equiv S^*_t|_{t=0}:L^2(\mathbb{R}^n) \to \dot{L}^2_1(\mathbb{R}^n)$ is bijective. 
Then for every $f \in L^2(\mathbb{R}^n)$, we have the following:
$$\mathcal{D}_{\pm t} f \to \left(\mp\frac{1}{2}I + K \right)f \,\,\,\, n.t. \text{ and in } L^2.$$
\end{lemma}

We first require a special case of the Gauss-Green formula.
\begin{lemma}\label{l4.ntgauss-green}  Let $L,L^*$ satisfy the standard assumptions, and
suppose that $Lu = 0,\, L^* w = 0$ in $\mathbb{R}^{n+1}_+$
with \begin{equation}\label{eq4.ntuniform}\sup_{t>0}\left(\|\nabla u (\cdot,t)\|_2 
+ \|\nabla w(\cdot,t)\|_2\right) < \infty,\end{equation} and $ 
\partial_\nu u\, \overline{w}(\cdot,0),\,\overline{\partial_{\nu^*} w}\, u(\cdot,0) 
\in L^1(\mathbb {R}^n)$\footnote{Here, $\partial_\nu$ and $\partial_{\nu^*}$ are the exterior conormal derivatives, corresponding 
to the matrices $A$ and $A^*$ respectively, which exist in the weak sense of
Lemma \ref{l4.ntconverge}.}. 
Suppose also that there exist $R_0,\beta>0$ such that for all $R > R_0$, we have
\begin{equation}\label{eq4.ntdecay}\iint_{\mathbb{R}^{n+1}_+\cap\left(B(0,2R)\setminus B(0,R)\right)}
|\nabla u| |\nabla w| + |\nabla u|R^{-1} |w| + |\nabla w| R^{-1} |u|  = O\left(R^{-\beta}\right).
\end{equation} Then
$$\int_{\mathbb{R}^n}\partial_\nu u \,\overline{w} =\int_{\mathbb{R}^n}u \,\overline{\partial_{\nu^*}w}.$$
\end{lemma} 
Of course, the analogous result holds in $\mathbb{R}^{n+1}_-$. 

\begin{proof} By the symmetry of our hypotheses, it is enough to show that 
\begin{equation}\label{eq4.ntgauss-green}
\iint_{\mathbb{R}^{n+1}_+} A \nabla u \cdot \overline{\nabla w} = 
\int_{\mathbb{R}^n} \partial_\nu u \,\overline{w}.\end{equation}
To this end, for $R_0<R<\infty$, let $\Theta_R(X) \equiv \Theta (X/R)$, where
$\Theta \in C_0^\infty (B(0,2))$ and $\Theta \equiv 1$ 
in $B(0,1).$ 
We set $w_R \equiv w \Theta_R$.  Then by Lemma \ref{l4.ntconverge}, we have that
$$\iint_{\mathbb{R}^{n+1}_+} A \nabla u \cdot \overline{\nabla w_R} = 
\int_{\mathbb{R}^n} \partial_\nu u \,\overline{w_R}.$$  A simple limiting argument
completes the proof.
\end{proof}

\begin{corollary}\label{cor4.ntgauss-green}Let $L,L^*$ satisfy the standard assumptions,
and suppose that the respective single layer potentials $S_t, S^*_t$ satisfy \eqref{eq4.ntbounded}.
Further suppose that $u(\cdot,\tau) = S_\tau \psi$ in $\mathbb{R}^{n+1}_-,$ where $\psi \in
C_0^\infty(\mathbb{R}^n)$.  Then setting $u_0 \equiv u(\cdot,0)$, we have
\begin{equation}\label{eq4.ntdoublelayeridentity}
\mathcal{D}_t u_0 =  S_t (\partial_\nu u).
\end{equation}
\end{corollary}

\begin{proof} It is enough to show that for all $\varphi \in C_0^\infty (\mathbb{R}^n)$,
we have $$\int_{\mathbb{R}^n}\mathcal{D}_t u_0\,\overline{\varphi} = 
\int_{\mathbb{R}^n}S_t (\partial_\nu u) \,\overline{\varphi}. $$ 
Note that $ad\!j(\mathcal{D}_t) = \vec{N} \cdot A^* \left(\nabla S_\tau^*\right)|_{\tau = -t},$
and that $ad\!j(S_t) = S^*_{-t}.$
Set $u^*(\cdot, \tau) \equiv S_\tau^* \varphi,$, so that $L^*u^* = 0$ in
$\mathbb{R}^{n+1}\setminus\{\tau = 0\}.$  It suffices to verify the hypotheses of
Lemma \ref{l4.ntgauss-green}, in the lower half-space, for $u, w$, with $w(\cdot,s )\equiv
u^*(\cdot, s -t), \, s\leq0.$  Estimate \eqref{eq4.ntuniform} is
immediate by \eqref{eq4.ntbounded}.  By Lemma \ref{l2.1}, we have 
\begin{equation}\label{eq4.ntasymptotic}|u(X)| +|w(X)|= O(|X|^{-n+1}) \,\,\,\,\text{ as } |X|\to \infty.
\end{equation}
Also, $Lu = 0, L^*w=0$ in $\mathbb{R}^{n+1}\setminus B(0,R_0),$ if $R_0$
is chosen large enough, since $\varphi, \psi$ have compact support.  Thus, by Caccioppoli,
$$\iint_{\mathbb{R}^{n+1}_-\cap(B(0,2R)\setminus B(0,R))} |\nabla u|^2 \leq C
\iint_{\mathbb{R}^{n+1}_-\cap(B(0,3R)\setminus B(0,R/2))} \left(\frac{|u|}{R}\right)^2 =
O\left(R^{-n+1}\right),$$
for $R>4R_0,$ and similarly for $w.$  Estimate \eqref{eq4.ntdecay} follows.
Finally, the boundary integrability of $\partial_\nu u \,\overline{w}$ and
$\overline{\partial_{\nu^*} w} \, u$ follows readily from
Cauchy-Schwarz, the fact that $n \geq 2$, and two observations:  first, that
by Lemma \ref{l2.9} and duality, we have
$$\int_{\Delta_{2R}(0)\setminus \Delta_R(0)} |\partial_\nu u|^2 +|\partial_{\nu^*} w|^2
= O(R^{-n});$$  second, that  \eqref{eq4.ntasymptotic} implies that
$$\int_{\Delta_{2R}(0)\setminus \Delta_R(0)}|u|^2+|w|^2 =O\left(R^{2-n}\right).$$ 
We leave the remaining details to the reader.
\end{proof}

\begin{proof}[Proof of Lemma \ref{l4.ntstrongconvergence}]
Since we have already obtained the limits
$(\mp\frac{1}{2}I + K)f$ in the weak sense (Lemma \ref{l4.ntjump}), it is enough here merely to establish existence of $n.t.$ and strong $L^2$  limits,
without concern for their precise values.  
We give the proof only in the case of the upper half-space,
as the proof in the other case is the same.  

We begin with the matter of non-tangential convergence.
Observe that $ad\!j\,(S_t \nabla) =
\left(\nabla S_{\tau}^*\right)|_{\tau = -t}$, so by
\eqref{eq4.ntbounded} and Lemma \ref{l4.nt1}$(vi)$, it is enough to establish
$n.t.$ convergence for $f$ in a dense class in $L^2$.  We claim now that 
$\{S_0 \dv_\| \vec{g}: \vec{g} \in C_0^\infty (\mathbb{R}^n,\mathbb{C}^n)\}$ is dense in
$L^2$.  Indeed, by hypothesis and duality, $S_0:\dot{L}^2_{-1} \to L^2$ is bijective.
Thus, $L^2 = \{S_0 \dv_\| \vec{g}: \vec{g} \in L^2\}.$  The density of $C_0^\infty$ in
$L^2$ establishes the claim.  

We now set $f = u_0 =S_0 (\dv_\|\vec{g}),$ with $\vec{g} \in C_0^\infty,$ and  
let $u(\cdot, \tau)=S_\tau (\dv_\| \vec{g}), \,\tau<0.$  We may then apply 
Corollary \ref{cor4.ntgauss-green} to obtain that $\mathcal{D}_t f = S_t(\partial_\nu u). $
Moreover, \eqref{eq4.ntbounded}, Lemma \ref{l4.nt1}, and Lemma \ref{l4.ntconverge} imply that
$\partial_\nu u \in L^2$ and hence also that $S_t (\partial_\nu u)$ converges $n.t.$,
from which fact the non-tangential part of $(ii)$ now follows.  

We turn now to the issue of strong convergence in $L^2$.  By
\eqref{eq4.ntbounded}, we have in particular that $L^2$ bounds hold,
uniformly in $t>0$, for $\mathcal{D}_t$.  Thus, it is once again enough to establish convergence in a dense class.  To this end, choose $u_0,u$ as above.  It suffices to show that
$\mathcal{D}_t u_0$ is  Cauchy convergent in $L^2$, as $t \to 0.$  Suppose that
$0<t'<t \to 0$, and observe that, by Corollary \ref{cor4.ntgauss-green}, 
\eqref{eq4.ntbounded} and our 
previous observation that $\partial_\nu u \in L^2,$
\begin{equation*} \|\mathcal{D}_t u_0 - \mathcal{D}_{t'} u_0\|_2= \|\int_{t'}^t\partial_s S_s (\partial_\nu
u)\,ds\|_2
\leq  (t-t')\|\partial_s  S_s \partial_\nu u\|_2 \to 0. 
\end{equation*}
\end{proof}

\begin{lemma}\label{l4.ntunique} ({\bf Uniqueness}). 
Suppose that $L,L^*$ satisfy the standard assumptions,
and that we have existence of solutions to (D2) and (R2).
Then those solutions are unique, in the following sense:
\begin{enumerate}\item[(i)]  If $u$ solves (D2), with $u(\cdot,t)\to 0$ in $L^2$, as $t \to 0$,
then $u \equiv 0.$
\item[(ii)]  If $u$ solves (R2), and $u \to 0 \,\, n.t.,$ then $u \equiv 0.$\footnote{Our data in the problem
(R2) belongs to $\dot{L}^2_1,$ whose elements are defined modulo constants;  thus, 
uniqueness in this context must be interpreted correspondingly.  We assume here
that we have chosen a particular realization of the data 
equal to $0 \, a.e.$ on the boundary.}
\end{enumerate}  If, in addition, $L$ and $L^*$ have ``Good Layer Potentials",
then the solution to (N2) is unique, in the sense that:
\begin{enumerate}
\item[(iii)]  If $u$ solves (N2), with $\partial u/\partial\nu = 0$ in the sense of Lemma \ref{l4.ntconverge}
$(iii)$ and $(iv)$, then $u \equiv 0$ (modulo constants).
\end{enumerate}
\end{lemma}
\begin{proof}  Consider first uniqueness in (D2).
We begin by constructing Green's function.  By Lemma \ref{l2.7} with $m=-1$,
for each fixed $(x,t) \in \mathbb{R}^{n+1}$, we have $\Gamma (x,t,\cdot ,0) \in \dot{L}^2_1,$
with \begin{equation} \label{eq4.ntfundamental}
\|\nabla_\|\Gamma(x,t,\cdot ,0)\|_{L^2(\mathbb{R}^n)} \leq C t^{-n/2}.
\end{equation}
Thus, by (R2), there exists $w=w_{x,t}$ solving
\begin{equation}
\tag{R2}\begin{cases} Lw=0\text{ in }\mathbb{R}_{+}^{n+1}\\ w(\cdot,s)\to 
\Gamma (x,t, \cdot,0) \, n.t.\\
\|\widetilde{N}_{\ast}(\nabla w)\|_{L^{2}(\mathbb{R}^{n})} \leq C t^{-n/2}.
\end{cases}\label{R2green}\end{equation}
Set $$G(x,t,y,s) \equiv \Gamma (x,t,y,s) - w_{x,t}(y,s),$$ and note that 
\begin{equation}\label{eq4.nt12}\sup_{s: |s-t|>t/8} \|\nabla G(x,t,\cdot ,s) \|_{L^2(\mathbb{R}^n)} \leq C t^{-n/2}.\end{equation}
Let $\theta \in C_0^\infty (\mathbb{R}^{n+1}_+)$, with $\theta \equiv 1$ in a neighborhood of
$(x,t)$.  Then, since $Lu=0$, we have
\begin{multline*}u(x,t) = (u \theta)(x,t) = \iint \overline{A^* \nabla_{y,s} G(x,t,y,s)
}\cdot \nabla (u \theta) dy ds\\ =-\iint \overline{G} \,\nabla\theta \cdot A \nabla u  + \iint 
\overline{\nabla G} \cdot A\nabla \theta \, u  \equiv  I + II.
\end{multline*}
We  now choose  
$\phi \in C_0^\infty (-2,2), \,\phi \equiv 1$ in $(-1,1)$, with $0 \leq \phi \leq 1,$
and set $\theta (y,s) \equiv [1-\phi (s/\varepsilon)]\,\phi (s/(100R))\,  \phi (|x-y|/R),$
with $\varepsilon < t/8, R > 8t.$ 
With this choice of $\theta$, the domains of integration in $I$ and
$II$ are contained in a union $\Omega_1 \cup \Omega_2 \cup \Omega_3$,
where 
\begin{enumerate}\item $\Omega_1 \subset \Delta_{2R}(x) \times  \{ \varepsilon < s < 2\varepsilon\} ,$
with $\|\nabla \theta\|_{L^\infty (\Omega_1)} \leq C \varepsilon^{-1}.$
\item $\Omega_2 \subset \Delta_{2R}(x) \times \{100R<s<200R\},$ with 
 $\|\nabla \theta\|_{L^\infty (\Omega_2)} \leq C R^{-1}.$
\item $\Omega_3 \subset \left(\Delta_{2R}(x) \setminus \Delta_R(x)\right) \times \{0 < s < 200R\}.$
with  $\|\nabla \theta\|_{L^\infty (\Omega_3)} \leq C R^{-1}.$
\end{enumerate}
We treat term $I$ first.
We recall from \cite{HK2} that
\begin{equation}\label{eq4.nt13} 
\|\nabla_{(\cdot)} G(X,\cdot)\|_{L^2(\mathbb{R}^{n+1}_+ \setminus
B(X,r))} \leq C r^{(1-n)/2},\,\,\,\forall r>0,\,X \in \mathbb{R}^{n+1}_+
\end{equation}
and that \begin{equation}\label{eq4.nt14}
|G(X,Y)|\leq C|X-Y|^{1-n},\end{equation}
whenever $|X-Y|\leq \min (\delta(X),\delta(Y))$,
where $\delta(X)$ denotes the distance to the boundary of the half-space
(i.e., the $t$-coordinate).
Thus, in particular we obtain that
\begin{equation}\label{eq4.nt15}R^{-1} \|G(x,t,\cdot)\|_{L^2(\Omega_2 \cup \Omega_3)}\leq C
R^{(1-n)/2}, 
\end{equation}
where in proving the bound on $\Omega_3$ 
we have used that $G$ vanishes on the boundary, to reduce matters to 
\eqref{eq4.nt13}.
We then have that \begin{equation}\label{eq4.nt16}
\frac{1}{C}|I| \leq \varepsilon^{-1} \iint_{\Omega_1} |G| |\nabla u|
+ R^{(1-n)/2}\left(\iint_{\Omega_2 \cup \Omega_3}|\nabla u|^2\right)^{1/2}\equiv I_1 + I_2.
\end{equation}
Since $u$ vanishes on $\{t=0\}$, we may apply Caccioppoli's inequality
in $\Omega_2\cup\Omega_3$ to obtain that
$I_2 \leq C R^{-n/2} \sup_{s>0}\|u(\cdot,s)\|_2 \to 0$
as $R \to\infty$. 

To treat $I_1$, we first note that for $(y,s) \in \Omega_1$,
\begin{equation}\label{eq4.nt17}
|G(x,t,y,s)|\leq C \varepsilon\left((|x-y|+t)^{-n} + \widetilde{N}_*(\nabla w_{x,t})(y)\right) ,
\end{equation} by Lemma \ref{l2.1}, Lemma \ref{l4.ntconverge} and construction of $G$. 
Consequently,
\begin{equation}\label{eq4.nt18}
\left(\varepsilon^{-1}\iint_{L^2(\Omega_1)}|G(x,t,y,s)|^2dyds\right)^{1/2} \leq C\varepsilon \,t^{-n/2},
\end{equation}
Thus, using Caccioppoli to estimate the $L^2$ norm of $\nabla u$ in $\Omega_1$, we obtain
that $$I_1\leq C t^{-n/2}\sup_{s<3\varepsilon} \|u(\cdot,s)\|_2 \to 0$$ as $\varepsilon \to 0,$
since $u(\cdot,s) \to 0 $ in $L^2$.

We now consider term $II$.  By Cauchy-Schwarz and then Caccioppoli's inequality,
\begin{eqnarray}\nonumber |II|& \leq &\varepsilon^{-1}\iint_{\Omega_1}
|\nabla G| \,|u| + R^{-1} \iint_{\Omega_2 \cap \Omega_3} |\nabla G|\,|u|\equiv II_1 + II_2\\\label{eq4.nt19}
&\leq & C \varepsilon^{-3/2} \|G(x,t,\cdot,\cdot)\|_{L^2(\Omega_1)}\sup_{s<2\varepsilon} \|u(\cdot,s)\|_2 
\\\nonumber  &&\qquad \quad  +R^{-3/2}\|G(x,t,\cdot,\cdot)\|_{L^2(\Omega_2 \cup \Omega_3)} \sup_{s>0}\|u(\cdot,s)\|_2.
\end{eqnarray}
By \eqref{eq4.nt18}, the term $II_1$ may be handled exactly like $I_1$, and by
\eqref{eq4.nt15}, $II_2$ yields the same bound as $I_2$.  The proof of uniqueness
in (D2) is now complete.

\smallskip
\noindent {\it Uniqueness in (R2).}  Suppose now that
$\widetilde{N}_*(\nabla u) \in L^2$, and that $u \to 0 \, n.t.$.  Choosing 
$\theta$ as above, we split $u(x,t) = (u \theta)(x,t)$ into the
same terms $I + II$, which we dominate again by
$I_1 + I_2$ and $II_1 + II_2$ as in \eqref{eq4.nt16} and \eqref{eq4.nt19},
respectively.  We now claim that 
$$I_1 +II_1 \leq C \varepsilon t^{-n/2} \|\widetilde{N}_*(\nabla u) \|_2 \to 0$$ 
as $\varepsilon \to 0$.  For $I_1$, this follows from
Cauchy-Schwarz and \eqref{eq4.nt18}.  To handle $II_1$,
we first note that, by Lemma \ref{l4.ntconverge}$(i)$,  $|u(y,s)| \leq C \varepsilon 
\widetilde{N}_*(\nabla u)(y)$ in $\Omega_1$, since $u(\cdot,0) = 0\,a.e.$. The claim then follows
from Cauchy-Schwarz and Caccioppoli (applied to $\nabla G$).  

Rewriting the last expression in \eqref{eq4.nt16}, we see that
$$I_2 = R^{(2-n)/2} \left( R^{-1} \iint_{\Omega_2 \cup \Omega_3} |\nabla u|^2
\right)^{1/2} \leq CR^{(2-n)/2}\|\widetilde{N}_* (\nabla u)\|_{L^2(\Delta_{2R}(x)\setminus \Delta_R(x))},
$$ by construction of $\Omega_2 \cup \Omega_3$. 
Moreover, Lemma \ref{l4.ntconverge}$(i)$ implies that
$|u|/R \leq C 
\widetilde{N}_* (\nabla u)$  in $\Omega_3$ and
$|u|/R \leq C \inf_{\Delta_{2R}(x)\setminus \Delta_R(x)}
\widetilde{N}_* (\nabla u)$ in $\Omega_2$.  Thus, by \eqref{eq4.nt13},
\begin{eqnarray*}II_2 & \leq &CR^{1/2}\left(\iint_{\Omega_2 \cup \Omega_3}
|\nabla_{y,s} G(x,t,y,s)|^2 dy ds\right)^{1/2}\left(R^{-1}\iint_{\Omega_2\cup\Omega_3}
\frac{|u|}{R}\right)^{1/2}\\&\leq &CR^{(2-n)/2}\,
\|\widetilde{N}_* (\nabla u)\|_{L^2(\Delta_{2R}(x)\setminus \Delta_R(x))}.\end{eqnarray*}
Since $n\geq 2$,  we obtain dominated convergence to $0$.

\smallskip
\noindent {\it Uniqueness in (N2).}  Suppose that $\widetilde{N}_* (\nabla u) \in L^2$,
and that $\partial u/\partial_\nu = 0$, where the latter is interpreted in the sense of Lemma \ref{l4.ntconverge}$(iii)$ and $(iv)$.  By Lemma \ref{l4.ntconverge}$(i)$, we have that $u \to u_0 \, n.t.$, for some $u_0
\in \dot{L}^2_1(\mathbb{R}^n).$  By uniqueness in (R2), $$u(\cdot,t) = S_t(S_0^{-1} u_0),$$
where $S_0 \equiv S_t|_{t=0}.$  Thus,  by Lemma \ref{l4.ntjump},
$$0 = \frac{\partial u}{\partial \nu} =\left(\frac{1}{2}I + \widetilde{K}\right)(S_0^{-1}u_0).$$
But by hypothesis, $\frac{1}{2}I + \widetilde{K}:  L^2 \to L^2$ and $S_0: L^2 \to \dot{L}^2_1$
are bijective, so that $u_0 =0$ in the sense of $\dot{L}^2_1$, i.e., $u_0 \equiv constant\, a.e.$.
By uniqueness in (R2), $u\equiv constant.$

\end{proof}

As a corollary of uniqueness, we shall obtain the following ``Fatou Theorem".
\begin{corollary}\label{4.ntfatou} Let $L,L^*$ satisfy the standard assumptions, and
have ``Good Layer Potentials".    Suppose also that $Lu=0$, and that
\begin{equation}\label{eq4.ntuniform2} \sup_{t>0} \|u(\cdot,t)\|_2 <\infty.
\end{equation}
Then $u(\cdot,t)$ converges $n.t.$ and in $L^2$ as $t \to 0$.
\end{corollary}

\begin{proof}By Lemma \ref{l4.ntstrongconvergence}, it is enough to show that
$u(\cdot,t) = \mathcal{D}_t h$ for some $h \in L^2(\mathbb{R}^n).$  We follow the argument in
\cite{St2}, pp. 199-200, substituting $\mathcal{D}_t$ for the classical Poisson kernel.
For each $\varepsilon > 0$, set $f_\varepsilon \equiv u(\cdot,\varepsilon).$  Let
$u_\varepsilon$ be the layer potential solution with data $f_\varepsilon$;  i.e.,
$$u_\varepsilon (x,t) \equiv \mathcal{D}_t\left[\left(-\frac{1}{2}I + K \right)^{-1} f_\varepsilon\right](x).$$
We claim that $u_\varepsilon(x,t) = u(x,t+\varepsilon).$
\begin{proof}[Proof of Claim]  Set $U_\varepsilon \equiv  u(x,t + \varepsilon) - u_\varepsilon(x,t).$
We observe that
\begin{enumerate}\item $LU_\varepsilon = 0$ in $\mathbb{R}^{n+1}_+$ (by $t$-independence of coefficients).
\item \eqref{eq4.ntuniform2} holds for $U_\varepsilon$, uniformly in
$\varepsilon > 0$ 
\item $U_\varepsilon (\cdot,0) = 0$ and $U_\varepsilon(\cdot,t) \to 0 \, \,\,n.t.$ and in $L^2$.
\end{enumerate}
(Item (3) relies on interior continuity \eqref{eq1.2} and smoothness in $t$, along with Lemma \ref{l4.ntstrongconvergence}). 
The claim now follows by Lemma \ref{l4.ntunique}.
\end{proof}

We return now to the proof of the Corollary.  By \eqref{eq4.ntuniform2},
$\sup_\varepsilon \|f_\varepsilon\|_2 < \infty.$  Hence, there exists a subsequence
$f_{\varepsilon_k}$ converging in the weak$^*$ topology to some $f \in L^2$.
For arbitrary $g \in L^2$, set $g_1 \equiv ad\!j\left( -\frac{1}{2} I+ K \right)^{-1}
ad\!j (\mathcal{D}_t) g$, and observe that
\begin{eqnarray*}\int_{\mathbb{R}^n} \mathcal{D}_t \left(-\frac{1}{2} I+K\right)^{-1} f \,\overline{g}
& = & \int_{\mathbb{R}^n} f\, \overline{g_1} = \lim_{k \to \infty} \int_{\mathbb{R}^n} f_{\varepsilon_k}
\overline{g_1}\\&=& \lim_{k \to \infty} \int_{\mathbb{R}^n}\mathcal{D}_t \left(-\frac{1}{2} I + K\right)^{-1}
\!f_{\varepsilon_k} \,\overline{g} \\&=& 
\lim_{k \to \infty} \int_{\mathbb{R}^n} u(\cdot,t+\varepsilon_k)\,\overline{g}
=\int_{\mathbb{R}^n} u(\cdot,t)\, \overline{g}.
\end{eqnarray*}

Since $g$ was arbitrary, the desired conclusion follows.
\end{proof}

We conclude this section with a discussion of $n.t.$ convergence of gradients.
\begin{lemma}\label{l4.ntconvergegradient}Suppose that $L,L^*$ satisfy the standard assumptions, and have ``Good Layer Potentials". 
Then for all $f \in L^2$, we have $$P_s\left(\left(\nabla S_{t}\right)|_{t = \pm s}\right) f 
\to \left(\mp\frac{1}{2A_{n+1,n+1}}e_{n+1} + \mathcal{T}\right)f\,\,
\,\,n.t. \text{  and in } L^2.$$  
\end{lemma}

\begin{proof}We treat only the case of the upper half space, as
the proof in the other case is the same.  Since the weak limit has already been established (Lemma
\ref{l4.ntjump}) for
$\nabla S_{t}$, it is a routine matter to verify that the strong and $n.t.$ limits for
$P_t (\nabla S_t)$ will take the same value, once the existence of those limits has been established.
It is to this last point that we therefore turn our attention.  By Lemma \ref{l4.nt1}
and the dominated convergence theorem, it is enough to establish
$n.t.$ convergence.

The non-tangential 
convergence of $\partial_t S_t$ follows immediately from the ``Fatou Theorem" just proved; 
a simple real variable argument yields the same 
conclusion for $P_t \partial_t S_t.$  We may therefore replace
$\nabla$ by $\nabla_\|$.  On the other hand, we shall still need
to consider the boundary trace of $\partial_t S_t f$, which for the duration of this proof 
we denote by $Vf$.    
Fix now $x_0 \in \mathbb{R}^n$.  For $|x-x_0|<t$, we write
\begin{eqnarray*}P_t\left(\nabla_\| S_t f\right)(x)&=& \nabla_xP_t\left(\int_0^t\partial_s S_s f ds \right)(x)
+P_t\left(\nabla_\|S_0 f\right)(x)\\ &\equiv & \vec{Q}_t \left(\frac{1}{t}\int_0^t \partial_s S_s f ds\right)(x)
+P_t\left(\nabla_\|S_0 f\right)(x) \equiv I + II,
\end{eqnarray*}
where $\vec{Q}_t 1 = 0$.  By standard facts for approximate identities, $II \to \nabla_\| S_0 f \,n.t.$.
Also, $$I =  \vec{Q}_t \left(\frac{1}{t}\int_0^t \big(\partial_s S_s f -Vf \big)\,ds\right)(x)+
\vec{Q}_t\left(Vf - Vf(x_0)\right)(x) \equiv I_1 + I_2.$$
It is straightforward to verify that $I_2 \to 0$ as $t \to 0$, if
$x_0$ is a Lebesgue point for the $L^2$ function $Vf$. 
The term $I_1$ is more problematic.  We first observe that by
Lemma \ref{l4.ntconverge},
\begin{equation}\label{eq4.nterror1} \left|\vec{Q}_t \left(\frac{1}{t}\int_0^t 
\big(S_s f -S_0f \big)\,ds\right)(x)\right|
\leq CtM\left(\widetilde{N}_*(\nabla S_tf)\right)(x_0)\to 0\end{equation}
for $a.e. \,x_0$.  Thus also for $\vec{f} \in C_0^\infty(\mathbb{R}^n)$, we have
\begin{equation}\label{eq4.nterror2} \left|\vec{Q}_t \left(\frac{1}{t}\int_0^t \left((S_s\nabla_\|)\cdot\vec{f} -(S_0\nabla_\|)\cdot\vec{f} 
\right)\,ds\right)(x)\right|
\to 0\,\,n.t..\end{equation}  
By Lemma \ref{l4.nt1}$(v)$, the density of 
$C_0^\infty$ in $L^2$, and the fact that $\vec{Q}_t$ is dominated by the Hardy-Littlewood maximal operator which
is bounded from $L^{2,\infty}$ to itself, the latter convergence continues to hold
for $\vec{f} \in L^2$.  Moreover, if $u_0$ belongs to the dense class
$\{S_0 \dv_\| \vec{g}: \vec{g} \in C_0^\infty\} $, by 
Corollary \ref{cor4.ntgauss-green} and \eqref{eq4.nterror1},
we have that
\begin{equation}\label{eq4.nterror3}\left|\vec{Q}_t \left(\frac{1}{t}\int_0^t \left(\mathcal{D}_t u_0 -
tr(\mathcal{D}_t u_0)\right)\,ds\right)(x)\right| \to 0\,\,n.t.,\end{equation}
and again this fact remains true for $u_0$ in $L^2$, by Lemma \ref{l4.nt1}$(vi)$
and our previous observation concerning the action of the maximal operator on weak $L^2$.
Combining \eqref{eq4.nterror2} and \eqref{eq4.nterror3} with the adjoint version of the identity
\eqref{eq4.ntDSidentity}, we obtain convergence to $0$ for the term $I_1$ since every
$f\in L^2$ can be written in the form $f=A_{n+1,n+1}^* h, \, h \in L^2.$  
\end{proof}

\section{Proof of Theorem~\ref{t1.10}: preliminary arguments\label{s3}}

As noted above, the De Giorgi-Nash estimate \eqref{eq1.2} is stable under 
$L^{\infty}$ perturbation of the coefficients. 
Thus, for $\epsilon_{0}$ sufficiently small, solutions of  $L_{1}u=0, \, L_1^*w=0$
satisfy \eqref{eq1.2} and \eqref{eq1.3}. In particular,  the 
results of Section \ref{s2} apply to the fundamental solutions and layer potentials
$\Gamma_0,\,S^0_{t}$ and $\Gamma_1,\,S^1_{t}$ corresponding to $L_0$ and $L_1$, respectively.

We claim that the conclusion of Theorem~\ref{t1.10} will follow, once we have proved
\begin{equation} \||t\nabla \partial_tS_t^1\||_{op}\,+\,
\sup_{t>0}\Vert\nabla S^1_{t}\Vert_{2\rightarrow 2}  \leq C\label{eq3.1}\end{equation}
(recall that $\nabla \equiv \nabla_{x,t}$).
Indeed, by the symmetry of our hypotheses, similar bounds will then hold in the lower half space, and 
for $S_t^{L_1^*}$.  
Now, by $t$-independence, $-(S^1_{t}D_{n+1})=D_{n+1}S^1_{t}$.  Moreover,
if $\mathcal{J}_t(x,y)$ denotes the kernel of
$(S^1_{t}\nabla_{\|})$, and $\Gamma_1^{\ast}$ is the 
fundamental solution for the adjoint operator $L_{1}^{\ast}$,
then the kernel of $ad\!j(S^1_{t}\nabla_{\|})$ is 
\begin{equation*}
\overline{\mathcal{J}_t(y,x)}  =\nabla_{x}\overline{\Gamma_1(y,t,x,0)}
=\nabla_{x}\Gamma_1^{\ast}(x,0,y,t)=\nabla_{x}\Gamma_1^{\ast}(x,-t,y,0).
\end{equation*}
Consequently,
$ad\!j(S^1_{t}\nabla_{\|})=\nabla_\|S^{L^*_1}_{-t}$, so that $L^2$ boundedness of
$(S^1_t \nabla)$ (and hence of $\mathcal{D}_t^1$) 
follows from that of $\nabla S_{-t}^{L_1^*}$.
Thus, by  Lemma \ref{l4.ntjump}, we also obtain $L^{2}$ bounds for 
$K^{1},\widetilde{K}^{1}$  and $\mathcal{T}^{1}$.  Appropriate non-tangential control follows from
Lemma \ref{l4.nt1}.  Moreover, since we have allowed complex
coefficients, analytic perturbation
theory implies that \begin{equation*}
\Vert K^{0}-K^{1}\Vert_{2\rightarrow2}+\Vert\widetilde{K}^{0}-\widetilde{K}^{1}\Vert_{2\rightarrow2}+\Vert
\,\mathcal{T}^{0}-\mathcal{T}^{1}\Vert_{2\rightarrow2}\leq C\Vert A^0-A^1\Vert_{\infty}.\end{equation*}
The method of continuity then yields the invertibility of $\pm\frac{1}{2}I+K^{1}:L^{2}\rightarrow L^{2}$,
$\pm\frac{1}{2}I+\widetilde{K}^{1}:L^{2}\rightarrow L^{2}$ and $S^1_{0}\equiv S^1_{t}|_{t=0}:L^{2}\rightarrow \dot{L}_{1}^{2}$.
It therefore suffices to prove \eqref{eq3.1}.

\begin{lemma}\label{l3.3} Suppose that $L,L^*$ satisfy the standard assumptions. 
For $f\in C_0^\infty$,  $\eta > 0$, and $t_0 \geq 0$, we have
\begin{eqnarray}\label{eq3.4}  \| \nabla_{\|}S_{t_0}f\|_{2} &\leq& 
C\left(\|N_*\big(P_t \partial_t S_{t+t_0}f\big)\|_{2}+\| |
t\nabla \partial _tS_t f| \| + \|f\|_2\right)\\
\label{eq3.4a}  \| \nabla_{\|}S^\eta_{t_0}f\|_{2} &\leq& 
C\left(\|N_*\left(P_t \partial_t S^\eta_{t+t_0}f\right)\|_{2}+\| |
t\nabla \partial _tS^\eta_t f| \| + \|f\|_2\right)\\
\label{eq3.5}\| |t\nabla \partial_t S_t f|\|&\leq& C\| | t\partial^2_t S_t f|\|+C\|f\|_2\\
\label{eq3.5a}
\| |t\nabla \partial_t S^\eta_t f|\|&\leq& C\| | t\partial^2_t S^\eta_t f|\|+C\|f\|_2.
\end{eqnarray} The analogous bounds hold also in the lower half-space.\end{lemma}

Before proving the lemma, 
let us use it to reduce the proof of Theorem \ref{t1.10} to two main estimates,
whose proofs we shall give in the next two sections.
We claim that it suffices to prove that for all $f\in C_{0}^{\infty}$,  and $\eta \in (0,10^{-10})$,
we have \begin{equation}
\Vert|t\partial_{t}^{2}S^{1,\eta}_{t}f|\Vert _{all} \leq C\epsilon_0\left( 
\||t\nabla \partial_tS_t^{1,\eta}f\||_{all}
+\Vert N^{db}_*\left(P_t\partial_{t}S^{1,\eta}_{t}f
\right)\Vert_{2}+\sup_{t\neq0}
\|\nabla S_t^{1,\eta}f\|_2\right) +C\Vert
f\Vert_{2}\label{eq3.6}\end{equation}
\begin{equation}
\sup_{t\neq 0} \Vert \partial_{t}S^{1,\eta}_{t}f\Vert_{2} \leq 
C \epsilon_0\left(\||t\nabla \partial_t S_t^{1,\eta}f \||_{all} + 
\|N_*^{db} \left(P_t\partial_t S_t^{1,\eta}f\right)\|_2 + \sup_{t\neq0}
\|\nabla S_t^{1,\eta}f\|_2\right)+C\Vert
f\Vert_{2} 
,\label{eq3.7}\end{equation}
where $N_*^{db}$ denotes the non-tangential maximal operator 
with respect to the double cone
$\gamma^{db}(x)\equiv\gamma^+(x) \cup \gamma^-(x) \equiv
\{(y,t) \in \mathbb{R}^{n+1}: |x-y|<|t|\}.$
Indeed, for $\epsilon_0$ sufficiently small, Lemma \ref{l2.approx} $(iii)$ and \eqref{eq3.5a}
allow us to hide the small triple bar norm in \eqref{eq3.6}, so that
\begin{equation}\Vert|t\nabla\partial_{t}S^{1,\eta}_{t}f|\Vert_{all}  \leq C\epsilon_0
\left(\Vert N^{db}_*\left(P_t\partial_{t}S^{1,\eta}_{t}f
\right)\Vert_{2}+ \sup_{t\neq0}
\|\nabla S_t^{1,\eta}f\|_2\right) +C\Vert f\Vert_{2}.\label{eq3.9a}\end{equation}
Using \eqref{eq3.4a}, \eqref{eq3.9a} and hiding the small gradient term via 
Lemma \ref{l2.approx} $(i,ii)$,
we obtain
\begin{equation}\label{eq3.9b}\sup_{t\neq 0}\|\nabla S_t^{1,\eta}f\|_2 \leq 
C\left(\sup_{t_0\geq 0}\|N_*^{db}\left(P_t\partial_t 
S_{t\pm t_0}^{1,\eta}f\right)\|_2 + \sup_{t\neq 0}\|\partial_tS^{1,\eta}_tf\|_2 +\|f\|_2\right),\end{equation}
where the notation $N_*^{db}\left(P_t\partial_t S_{t\pm t_0}^{1,\eta}f\right)$ is interpreted to mean
$t+t_0$ in the upper cone $\gamma^+$, and $t-t_0$ in the lower cone $\gamma^-.$
Feeding the latter estimate back into \eqref{eq3.9a}, we obtain 
\begin{equation}\Vert|t\nabla\partial_{t}S^{1,\eta}_{t}f|\Vert_{all}  \leq C\epsilon_0
\left(\sup_{t_0\geq 0}\|N_*^{db}\left(P_t\partial_t 
S_{t\pm t_0}^{1,\eta}f\right)\|_2 + \sup_{t\neq 0}\|\partial_tS^{1,\eta}_tf\|_2\right) +C\Vert f\Vert_{2}.
\label{eq3.9c}\end{equation}
Combining \eqref{eq3.7}, \eqref{eq3.9b} and \eqref{eq3.9c}, we have 
$$\sup_{t\neq 0} \Vert \partial_t S^{1,\eta}_{t}f\Vert_{2}\,  \leq \,C\Vert
f\Vert_{2}+C\epsilon_0\left(\sup_{t_0\geq 0}\|N_*^{db}\left(P_t\partial_t S_{t\pm t_0}^{1,\eta}f\right)\|_2 +
\sup_{t\neq 0}\|\partial_tS^{1,\eta}_tf\|_2\right).$$
Since $f \in C_0^\infty$, there is a large cube $Q$ centered at $0$ containing the support of $f$.
By Lemma \ref{l4.nt1} $(iv)$, taking a supremum over
all $f\in C_0^\infty(Q)$, with $\|f\|_{L^2(Q)} = 1$, we have
$$\sup_{t\neq 0} \Vert \partial_{t}S^{1,\eta}_{t}\Vert_{L^2(Q)\to L^2(\mathbb{R}^n)}\,  \leq \,C\left(
1+\epsilon_0
\sup_{t\neq 0}\|\partial_tS^{1,\eta}_t\|_{L^2(Q) \to L^2(\mathbb{R}^n)}\right).$$
Using Lemma \ref{l2.approx} $(vi)$, we may hide the small term to obtain
\begin{equation}
\sup_{t\neq 0} \Vert \partial_{t}S^{1,\eta}_{t}\Vert_{L^2(Q)\to L^2(\mathbb{R}^n)}\,  \leq \,C
\label{eq3.10a}\end{equation}
uniformly in $Q$.   Thus, letting $\ell(Q) \to \infty$, and then $\eta \to 0$, we obtain by Lemma
\ref{l2.approx} $(iv)$ that
\begin{equation}\label{eq3.11a}\sup_{t\neq0} \|\partial_t S_t^1 \|_{2\to 2} \leq C.\end{equation}
In addition, \eqref{eq3.10a}, Lemma \ref{l4.nt1} $(iv)$ and a limiting argument as 
$\ell(Q) \to \infty$ imply that
$$\sup_{t_0\geq 0}\|N_*^{db}\left(P_t \partial_t S_{t\pm t_0}^{1,\eta} f\right)\|_2\leq C
\|f\|_2,\,\,\,f\in L^2(\mathbb{R}^n).$$
The latter estimate,  \eqref{eq3.9c}, \eqref{eq3.10a}  and Lemma \ref{l2.approx} $(v)$
yield the bound for the first term in
\eqref{eq3.1}.  The bound for the second term in \eqref{eq3.1} follows from \eqref{eq3.4},
the bound just established for $\||t\nabla\partial_tS^1_t\||_{op},$ 
the fact that $N_*\left(P_t \partial_t S_{t+t_0}f\right) \leq
CM\left(N_*(\partial_t S_t f)\right)$, Lemma \ref{l4.nt1} $(i)$ and \eqref{eq3.11a}.

The estimates \eqref{eq3.6} and \eqref{eq3.7} are the heart of the matter, and will be proved in sections \ref{s4} and \ref{s5}, respectively.

We return now to the proof of the lemma.

\begin{proof}[Proof of Lemma \ref{l3.3}] We prove \eqref{eq3.5} first.  
We have that $ \| | t\nabla \partial_t S_t f|\|^2 = $
\begin{equation*}\begin{split}\lim_{\varepsilon \to 0}
\| | t\nabla \partial_t S_t f|\|^2(\varepsilon) &\equiv \lim_{\varepsilon \to 0}\int_{\mathbb{R}^n}
\int^{1/\varepsilon}_\varepsilon \nabla
\partial_t S_t f\cdot \overline{\nabla \partial_t S_t f} tdt\\ &=-\frac{1}{2}\lim_{\varepsilon \to 0}
\int_{\mathbb{R}^n}\int^{1/\varepsilon}_\varepsilon
\partial_t (\nabla \partial_t S_t f\cdot
\overline{\nabla \partial_tS_t f}) t^2dt+ \text{ ``OK"},\end{split}\end{equation*}
where we may use Lemma \ref{l2.10}$(ii)$ to dominate the ``OK" boundary terms
by $C\|f\|^2_2$.  By Cauchy's inequality, we then obtain that
\begin{equation*}\| | t\nabla \partial_t S_tf| \|^2(\varepsilon)\leq \delta \,\| |t\nabla \partial_t S_t f| \|^2
(\varepsilon)+\frac{C}{\delta} \||t^2\nabla
\partial ^2_t S_t f| \|^2(\varepsilon) +C\|f\|_2^2,\end{equation*} where $\delta$ is at our disposal. For $\delta$ small, we can hide the first
term. The second term is bounded by 
$\| |t\partial_t^2S_t f|\|$, as may be seen by splitting $\mathbb{R}^{n+1}_+$ into Whitney boxes, and applying
Caccioppoli's inequality. The bound \eqref{eq3.5} now follows.

The proof of \eqref{eq3.5a} is similar.  We write
$$\||t\nabla \partial_t S_t^{1,\eta}f\||^2 = \int_0^{2\eta} \!\int_{\mathbb{R}^n} + \int_{2\eta}^\infty 
\!\int_{\mathbb{R}^n} \equiv I + II.$$  Term $II$ may be 
handled just like \eqref{eq3.5}, since by definition \eqref{eq2.def},
$$|t\nabla \partial_t S_t^\eta f|
\leq C\left(\varphi_\eta\ast\left(1_{s>\eta}|s\nabla \partial_s S_s f|\right)\right)(t),\,\,\,\,t>2\eta,$$ and  
$u(x,t) \equiv \partial_t^2 S_t^\eta f(x)$ solves
$Lu=0$ in the half space $\{t>\eta\}$.  We omit the details.
To bound term $I$, we note that by definition \eqref{eq2.def}, 
$\partial_t S_t^\eta f(x)= L^{-1}(D_{n+1} f_\eta)(x,t)$,
where $f_\eta(y,s)\equiv f(y) \varphi_\eta(s),$
so that
$$|I| \leq C\eta \iint |\nabla L^{-1} (D_{n+1} f_\eta)|^2 dx dt \leq C\eta \left(\int|\varphi_\eta(t)|^2dt\right)
\|f\|_2^2 = C\|f\|_2^2,$$
where we have used that $\nabla L^{-1} \dv : L^2(\mathbb{R}^{n+1}) \to L^2(\mathbb{R}^{n+1}).$

Next, we prove \eqref{eq3.4}. By the ellipticity of the sub-matrix $A_{\|}$, we have that
\begin{equation*}\| \nabla_{\|}S_{t_0}f\|_2 \leq C\| A_{\|}\nabla_{\|}S_{t_0} f\|_2.\end{equation*} 
Now let $\vec{g}\in
C_0^\infty(\mathbb{R}^n ,\mathbb{C}^n)$, with $\| \vec{g}\|_{2}=1$.  By the Hodge decomposition 
\cite[p. 116]{AT}, we have that
$\vec{g}=\nabla_{x}F+\vec{h}$, where $F\in \dot{L}^2_1(\mathbb{R}^n)$, $\| \nabla_x F\|_2\leq C\| \vec{g}\|_2$ ($C$ depending
only on ellipticity), $h\in
L^2(\mathbb{R}^n)$ and $\dv_{\|}(A_{\|})^\ast
\vec{h}=0$ in the sense that $\int A_{\|}\nabla_{\|}\zeta \cdot \overline{\vec{h}}=0$ for all $\zeta \in
\dot{L}^2_1$.  Lemma 2.9, with $m = -1$, ensures that $S_{t_0}f\in \dot{L}^2_1$, (albeit without
quantitative bounds).  Thus, for $f\in C_0^\infty(\mathbb{R}^n)$, we have
$$\langle A_{\|}\nabla_{\|}S_{t_0}f,\vec{g}\rangle = \langle A_{\|}\nabla_{\|}S_{t_0} f,\nabla_{\|}F\rangle,$$ and it suffices to bound the latter expression with $F \in C^\infty_0$.  Now, 
\begin{multline*}  \langle A_{\|}\nabla_{\|}S_{t_0} f,\nabla_{\|}F\rangle =
-\int^\infty_0\partial_t\langle A_{\|}\nabla_{\|}e^{-t^2L_{\|}}S_{t + t_0}f,\nabla_{\|}e^{-t^2(L_{\|})^\ast}F\rangle
dt\\ =2\! \int^\infty_0\!\left\{\langle A_{\|}\nabla_{\|}tL_{\|}e^{-t^2L_{\|}}S_{t+t_0}f,\nabla_{\|}e^{-t^2(L_{\|})^\ast}F\rangle \,+\, \langle A_{\|}\nabla_{\|}e^{-t^2L_{\|}}S_{t+t_0}f,
\nabla_{\|}t(L_{\|})^\ast e^{-t^2(L_{\|})^\ast}F\rangle \right\}dt\\
-\,\int^\infty_0 \langle A_{\|}\nabla_{\|}e^{-t^2L_{\|}}\partial_tS_{t+t_0}f,\nabla_{\|}e^{-t^2(L_{\|})^\ast }F\rangle dt\,\, =\,\,
I+II-III.\end{multline*}
Integrating by parts, we see that
\begin{multline}\label{eq3.8} |I+II| =4\left|\int^\infty_0 \int_{\mathbb{R}^n}
\left(L_{\|}e^{-t^2L_{\|}}S_{t+t_0}f(x)\right)\,\left(\overline{(L_{\|})^\ast
e^{-t^2(L_{\|})^\ast}F(x)}\right)\,tdxdt\right|\\ \leq 4\| | te^{-t^2L_{\|}}L_{\|}S_{t+t_0}f|\| \,\||t(L_{\|})^\ast
e^{-t^2(L_{\|})^\ast}F|\|\leq C\| | te^{-t^2L_{\|}}L_{\|}S_{t+t_0}f|\|
\,\|\nabla F\|_2,\end{multline} since, by \cite{AHLMcT}, applied to $(L_{\|})^*$, we have that
$\| |t(L_{\|})^\ast e^{-t^2(L_{\|})^\ast}F|\| \leq C\|\nabla F\|_2.$ 
We consider now the first factor on the right side of
\eqref{eq3.8}. Since $u(x,t)\equiv S_{t+t_0}f(x)$ solves $Lu=0$, we have 
\begin{equation*}
L_{\|}S_{t+t_0}f=\sum^n_{i=1} D_iA_{i,n+1} D_{n+1} S_{t+t_0}f+\sum_{j=1}^{n+1}
A_{n+1,j}D_jD_{n+1} S_{t+t_0}f\,\equiv\, \Sigma_1 +\Sigma_2,\end{equation*}
in the weak sense of Lemma \ref{l2.identity}. Since 
$e^{-t^2L_{\|}}:L^2\to L^2$
uniformly in $t$, we obtain
\begin{equation*}\| | te^{-t^2L_{\|}}\Sigma_2 |\| \leq C\| |t\nabla \partial_t S_{t+t_0}f|\| \leq C \| |t\nabla
\partial _t S_t f|
\|\end{equation*} which is one of the allowable terms in the bound that we seek. Also,
\begin{equation}\label{eq3.9} te^{-t^2L_{\|}}\Sigma_1 = R_t
\partial _t S_{t+ t_0}f +\sum^n_{i=1}
(te^{-t^2L_{\|}}D_iA_{i,n+1})P_t\partial_tS_{t+ t_0}f,\end{equation} 
where, by the familiar ``Gaffney estimate"(e.g., \cite{AHLMcT}, pp. 636-637), the operator
\begin{equation*}R_t\equiv \sum_{i=1}^n \left(te^{-t^2L_{\|}}D_iA_{i,n+1}-(te^{-t^2L_{\|}}
D_iA_{i,n+1})P_t \right)\end{equation*} satisfies
the bound \eqref{eq2.11} for every $m\geq 1$ (indeed, it satisfies a stronger exponential decay estimate). Moreover,
$R_t1=0$, and $R_t:L^2\to L^2$.  Thus, by Lemma~\ref{l2.15} we have
\begin{equation*}| \|R_t\partial _tS_{t+ t_0}f\| | \leq C |\|t\nabla \partial_tS_{t+ t_0}f\| |\leq C|\| t\nabla
\partial_tS_t f\| |\end{equation*} as desired. In addition, by \cite{AHLMcT}, we have that $|te^{-t^2L_{\|}}\dv_{\|}\vec{b}|^2
\frac{dx dt}{t}$ is a Carleson measure for all $\vec{b}\in 
L^\infty (\mathbb{R}^n,\mathbb{C}^n)$. Therefore, 
by Carleson's Lemma, the triple bar norm of the last term in \eqref{eq3.9}
is dominated by
$\|N_\ast (P_t\partial_t S_{t+t_0}f)\|_2.$

It remains to handle the term III. Integrating by parts in $t$, we obtain
\begin{equation}\label{eq3.10}-III  = \int^\infty_0 \langle A_{\|}\nabla_{\|}e^{-t^2L_{\|}}\partial^2_t
S_{t+t_0}f,\nabla_{\|}e^{-t^2(L_{\|})^\ast} F\rangle tdt \,\, + \, \text{ ``easy"},\end{equation} where the two ``easy
terms" arise when $\partial_t$ hits either $e^{-t^2L_{\|}}$ or $e^{-t^2(L_{\|})^\ast }$. These two easy terms
may be handled by an argument similar to, but simpler than the one used to treat \eqref{eq3.8} above. The main term in \eqref{eq3.10} is dominated by
\begin{equation*} | \|te^{-t^2L_{\|}}\partial^2_t S_{t+t_0}f\| | \,\, |\|t(L_{\|})^\ast e^{-t^2(L_{\|})^\ast}F\| |
\leq C |\|t\partial^2_t S_tf\| | \,\, \| \nabla F\|_2,\end{equation*} where we have used the 
$L^2$ boundedness of $e^{-t^2L_{\|}}$ to estimate the first factor, and \cite{AHLMcT} to handle the second.

Finally,  \eqref{eq3.4a} may be proved in the same way as \eqref{eq3.4} with one minor modification.
Since $L S_t^\eta f(x) = f_\eta(x,t) \equiv f(x) \varphi_\eta(t)$, the application of Lemma \ref{l2.identity}
produces, in addition to the analogues of $\Sigma_1$ and $\Sigma_2$, 
an error term $f_\eta (\cdot,t+t_0).$
But $$\||te^{-t^2L_\|} f_\eta (\cdot, t+t_0)\|| \leq C 
\left(\eta \int |\varphi_\eta(t+t_0)|^2 dt\right)^{1/2} 
\|f\|_{L^2(\mathbb{R}^n)} = C \|f\|_2,$$
and \eqref{eq3.4a} follows.\end{proof}

We finish this section with a variant of the square function estimates.
\begin{lemma}\label{l3.11} Suppose that $L,L^*$ satisfy the standard assumptions,
and have ``Good Layer
Potentials".  Then for $m \geq 0$, we have the square function bound
\begin{equation*}| \|t^{m+1} \partial^{m+1}_t (S_t\nabla )\cdot {\bf f}\| |\leq C_m \| {\bf f}\|_2
,\end{equation*} where ${\bf f} \in L^2(\mathbb{R}^n,\mathbb{C}^{n+1})$.\end{lemma}

\begin{proof} By $t$-independence and Caccioppoli's inequality in Whitney boxes, we may reduce to the case $m=0$. By
$t$-independence and \eqref{eq1.9}, we may replace $\nabla$ by $\nabla_\|$.  By ellipticity of the
$n\times n$ sub-matrix $A_{\|}$, 
and the Hodge decomposition of \cite[p. 116]{AT}, as in the proof of Lemma \ref{l3.3}, it
suffices to show that
\begin{equation}\label{eq3.12} | \| t\partial_t (S_t\nabla_{\|} )\cdot A_{\|}\nabla_{\|}F\| |\leq C\| \nabla_\|
F\|_2,\end{equation} with $F \in \{S_0 \psi: \psi \in C_0^{\infty}\}$ (which is dense in $\dot{L}^2_1$, by
the bijectivity of the mapping $S_0:L^2\to \dot{L}^2_1)$.  
 In the weak sense of Lemma \ref{l2.identity}, we have
\begin{equation*} \overline{(L_{\|})^\ast _y} \Gamma (x,t,y,s)=\sum^n_{i=1}\frac{\partial}{\partial y_i} \left(\overline{A^\ast_{i,n+1} (y)}\partial_s\Gamma
(x,t,y,s)\right) \,+\,\sum^{n+1}_{j=1} \overline{A^\ast _{n+1,j}(y)}\frac{\partial }{\partial y_j}\partial _s \Gamma (x,t,y,s).\end{equation*} By $t$-independence, we therefore have that
\begin{equation*}\partial _t (S_t \nabla_{\|} )\cdot 
A_{\|}\nabla_{\|}F =\sum^n_{i=1}\partial^2_tS_t
A_{n+1,i}D_iF+\partial^2_t(S_t\overline{\partial_{\nu^\ast}})F,\end{equation*} where $\overline{\partial_{\nu^\ast}}
=-\sum^{n+1}_{j=1}\overline{A^\ast _{n+1,j}}D_j$.   We set
$u(\cdot,\tau) = S_\tau \psi, \tau<0$, so that $u(\cdot,0) \equiv F$. 
Using ``Good Layer Potentials", we obtain in particular that
\begin{equation}\label{eq3.Rellich}\|\nabla u(\cdot,0)\|_2 \leq C \|\nabla_{\|}F\|_2.
\end{equation}  Since 
$(S_t\partial_{\nu^*}) = \mathcal{D}_t$,  Corollary \ref{cor4.ntgauss-green}
implies that
\begin{equation*} \partial^2_t (S_t\overline{\partial_{\nu^\ast}})F=\partial^2_tS_t(\partial_\nu u(\cdot ,0)).\end{equation*}
Consequently, the left hand side of \eqref{eq3.12} is dominated by
\begin{equation*}\sum^n_{i=1} |\|t\partial^2_t S_t A_{n+1,i}D_iF \| | +| \|t\partial^2_t S_t
(\partial_\nu u(\cdot ,0))\| | \leq C\| \nabla_{\|}F\|_2,\end{equation*} where in the last step we have
used \eqref{eq1.9} and \eqref{eq3.Rellich}.\end{proof}

\section{Proof of Theorem \ref{t1.10}: the square function estimate \eqref{eq3.6}\label{s4}}

In this section we prove estimate \eqref{eq3.6}. To be precise, suppose that
$\varphi_\delta =\delta^{-1} \varphi (\cdot/\delta)$ is the 
kernel of a nice approximate identity in $1$ dimension, as in the definition of
$S_t^{\eta}$ \eqref{eq2.def}.
We shall prove that, for all $f\in C_{0}^{\infty}(\mathbb{R}^n)$,   
for all $\Psi\in C_0^\infty(\mathbb{R}^{n+1}_+)$, with $\||\Psi\|| \leq 1$,
and for all $\delta > 0$ sufficiently small, if $\Psi_\delta(x,t) \equiv 
\varphi_\delta \ast \Psi(x,\cdot)(t)$,  then  
\begin{equation}
\label{eq4.1}\iint_{\mathbb{R}^{n+1}_+}  
t \partial_t^2 S_t^{1,\eta} f(x) \,\overline{\Psi_\delta(x,t)} \,\frac{dx\,dt}{t} \leq 
C\epsilon_{0}\left({\bf M^+} + {\bf M^-}\right)
+\,C\Vert f\Vert_{2},
\end{equation}
where 
\begin{equation}\label{eq4.defM} {\bf M^+} \equiv \left(|\| t\nabla\partial_t S^{1,\eta}_{t} f\| |_{+} +
\|N_*\left(P_t \partial_t S^{1,\eta}_{t}f\right)\| _2 +\sup_{t\geq 0}
\|\nabla S_t^{1,\eta} f\|_2 +\|f\|_2\right),
\end{equation} and ${\bf M^-}$ is the corresponding quantity for the lower half-space.
The proof of the analogous estimate in $\mathbb{R}^{n+1}_-$ 
is identical, and we omit it.  By Lemma \ref{l2.approx} $(iii)$,  we may take first the limit as
$\delta \to 0$, and then  the supremum
over all such $\Psi$ to obtain \eqref{eq3.6}.
  
The proof is by perturbation.  Setting $\epsilon(z)\equiv A^1(z)-A^0(z)$,
we have \begin{equation*}
L_{0}^{-1}-L_{1}^{-1}  =L_{0}^{-1}L_{1}L_{1}^{-1}-L_{0}^{-1}L_{0}L_{1}^{-1}
= -L_{0}^{-1}\dv\epsilon\nabla L_{1}^{-1}.
\end{equation*} 
Since $|\Vert t \partial^2_t S^0_{t} f |\Vert
\leq C \|f\|_2$, we have also that 
$\sup_{\eta>0}|\Vert t \partial^2_t S^{0,\eta}_{t} f |\Vert
\leq C \|f\|_2$, as may be seen by arguing as in  the proof of \eqref{eq3.5a}.
Thus, it is enough to consider the difference $t\partial_{t}^{2}\left(S^{1,\eta}_{t}
-S^{0,\eta}_{t}\right)$. 
By definition \eqref{eq2.def}, 
\begin{equation}\label{eq4.etaidentity}
\partial_tS_t^{i,\eta} f(x) =\left((D_{n+1}\varphi_\eta) \ast  S_{(\cdot)}^i f (x)\right)(t) 
=L_i^{-1}(D_{n+1}f_\eta)(x,t), \,\,\, i = 1,2,\end{equation} where
$f_\eta (y,s) \equiv f(y) \varphi_\eta(s)$, and $\varphi_\eta =\eta^{-1} \varphi (\cdot/\eta)$ is 
as above.  We then have
\begin{eqnarray*}\partial_t^2S^{1,\eta}_t f(x) - \partial_t^2S^{0,\eta}_t f(x) 
&=& \partial_t \left(L_0^{-1} \dv \epsilon \nabla L_1^{-1}(D_{n+1} f_\eta)\right)(x,t)\\
&=&\partial_t \left(L_0^{-1} \dv \epsilon \nabla D_{n+1} S_{(\cdot)}^{1,\eta} f \right)(x,t),\end{eqnarray*}
so that
\begin{eqnarray}\iint_{\mathbb{R}^{n+1}_+} \left(t\partial_t^2 S_t^{1,\eta}f(x) - 
t \partial_t^2 S_t^{0,\eta} f(x)\right) \overline{\Psi_\delta(x,t)}
\,\frac{dx\,dt}{t}= \qquad \qquad \qquad \nonumber \\ 
\label{eq4.2}  \qquad \quad \iint_{\mathbb{R}^{n+1}}
\epsilon(y) \nabla \partial_s S_s^{1,\eta} f (y) \cdot \overline{\nabla (L_0^*)^{-1}(D_{n+1}
\Psi_\delta)(y,s)}\,dy ds.\end{eqnarray}
Essentially following \cite{FJK}, and using \eqref{eq4.etaidentity}, we decompose 
 \begin{multline*} 
\nabla (L_0^*)^{-1}(D_{n+1}\Psi_\delta)(y,s)=\int 
\nabla_{y,s}\partial_sS^{L_0^*,\delta}_{s-t}\left(\Psi(\cdot,t)\right)(y) \,dt \\
= \int_{t>2|s|} \left\{\nabla_{y,s}\partial_sS^{L_0^*,\delta}_{s-t}\left(\Psi(\cdot,t)\right)(y) -
 \left(\nabla_{y,s}\partial_s S^{L_0^*,\delta}_{s-t}\right)\big|_{s=0}
 \left(\Psi(\cdot,t)\right)(y)\right\}\, dt\\\!\!\!\!\!\!\!\!\!\!\!
\!\!\! \!\!\!\!  +\, \int_{t>2|s|} \left(\nabla_{y,s}\partial_sS^{L_0^*,\delta}_{s-t}\right)\big|_{s=0}
 \left(\Psi(\cdot,t)\right)(y) \,dt \\ \qquad\qquad
 \qquad+\,\int_{t\leq 2|s|}\left(\frac{\sqrt{t}-\sqrt{|s|}}{\sqrt{t}}\right) \,
\nabla_{y,s}\partial_sS^{L_0^*,\delta}_{s-t}\left(\Psi(\cdot,t)\right)(y)\, dt \\\qquad\qquad\qquad\qquad
+\, \int \left(\frac{|s|}{t}\right)^{1/2}
\nabla_{y,s}\partial_sS^{L_0^*,\delta}_{s-t}\left(\Psi(\cdot,t)\right)(y)\, dt\\-\,
\int _{t>2|s|} \left(\frac{|s|}{t}\right)^{1/2}\nabla_{y,s}\partial_s
S^{L_0^*,\delta}_{s-t}\left(\Psi(\cdot,t)\right)(y)\, dt 
\,\,\equiv\,\, \bf{i + ii + iii + iv - v}.
\end{multline*}
In turn, this induces a corresponding decomposition 
in \eqref{eq4.2}:
\begin{equation*}
I +II + III + IV - V
\equiv \iint_{\mathbb{R}^{n+1}}
\epsilon(y) \nabla \partial_s S_s^{1,\eta} f (y) \cdot
\overline{(\bf{i + ii + iii + iv - v})}\,dy ds.
\end{equation*}
All but term $II$ will be easy to handle, and we shall deal with these
easy terms as in \cite{FJK}.  The main term here (and in \cite{FJK}) is $II$, but in our situation, matters are much more delicate,
since for us $A^0$ is not constant.  The approach of \cite{FJK} depends critically on the fact
that solutions of constant coefficient equations are, in particular, twice differentiable, a fact which fails utterly in the present setting (unless at least one of the derivatives falls on the $t$-variable). 
We shall require new methods, which exploit the technology of the solution of the Kato problem, to
deal with term $II.$ 

We dispose of the easy terms in short order.
To begin, $$ IV = \iint_{\mathbb{R}^{n+1} }|s|^{1/2}
\epsilon(y) \nabla \partial_s S_s^{1,\eta} f (y)\cdot 
\nabla (L_0^*)^{-1} \left(D_{n+1} \left(\varphi_\delta\ast\frac{\Psi}{\sqrt{t}}\right)\right)(y,s) \,dy ds.$$
 Since $\nabla L_{0}^{-1}\dv:L^{2}(\mathbb{R}^{n+1})\rightarrow L^{2}(\mathbb{R}^{n+1})$, we have that 
 $|IV|\leq C\epsilon_{0}|\Vert
t\nabla\partial_{t}S^{1,\eta}_{t}f\Vert|_{all}.$ 
Given the following lemma, I, III and V may be handled by Hardy's inequality, yielding also the bound
$|I| + |III| + |V| \leq C\epsilon_{0}|\Vert t\nabla\partial_{t}S^{1,\eta}_{t}f\Vert|_{all}$. We omit the
details.

\begin{lemma}\label{l4.5}We have
\begin{align} \label{eq4.6} &\| \nabla D_{n+1} S^ {L_0^*,\delta}_{s-t}
-\nabla D_{n+1}  S^{L_0^*,\delta}_{-t}\| _{2 \to 2} \leq
C\frac{|s|}{t^2} ,\quad |s|<t/2,\,\,\delta < 1000^{-1} t \\
\label{eq4.7} & \| \nabla \partial_\tau S^{L_0^*,\delta}_{\tau } \| _{2\to 2}\leq 
\frac{C}{|\tau |} ,\quad \tau \neq 0\end{align}
\end{lemma}

\begin{proof}[Proof of the Lemma]  If $|\tau| > 100 \delta,$ estimate \eqref{eq4.7} is essentially just the case $m=0$ of Lemma~\ref{l2.10}.  Otherwise, we obtain the 
better bound $C \delta^{-1}$, using definition \eqref{eq2.def} and the hypothesis
that $L_0,L_0^*$ have bounded layer potentials.  Estimate
\eqref{eq4.6} is obtained from the case $m=1$ of Lemma~\ref{l2.10}, and the identity
\begin{equation*}\nabla D_{n+1} S^{L_0^*,\delta}_{s-t}-
\nabla D_{n+1} S^{L_0^*,\delta}_{-t} =
\int_0^{s}\nabla \partial_\tau^2S^{L_0^*,\delta}_{\tau-t}\,d\tau.\end{equation*} \end{proof}

It remains to handle II, which equals 
\begin{eqnarray}\nonumber
\iint_{\mathbb{R}^{n+1}}\left\{\int_{-t/2}^{t/2} 
\epsilon(y) \nabla \partial_s S_s^{1,\eta} f (y)\,ds\right\}\cdot
\overline{
 \left(\nabla D_{n+1} S^{L_0^*,\delta}_{-t}\right)\left(\Psi(\cdot,t)\right)(y)}\,dy dt\quad\\\label{eq4.8}
\quad=- \iint_{\mathbb{R}_+^{n+1}}\left(\partial_tS_t^{0}\nabla\right)\cdot
\epsilon\nabla\left(S^{1,\eta}_{t/2}f-S^{1,\eta}_{-t/2}f\right)(x)\,
\overline{\Psi_\delta(x,t)}\,dx dt,
\end{eqnarray}
where we have used that for $\eta > 0$, 
$\nabla S_t^\eta$ does not jump across the boundary. 
Since $\Psi$ is compactly supported in $R^{n+1}_+$, for $\delta$ sufficiently small,
$$t^{-1/2}| \Psi_\delta (x,t)| \leq C\int \varphi_\delta (t-s) |\Psi(x,s)| s^{-1/2} ds.$$ 
Thus, it is enough to bound $\||t\left(\partial_tS_t^{0}\nabla\right)\cdot
\epsilon\nabla S^{1,\eta}_{t/2}f \||, $ plus a similar term 
with $-t/2$ in place of $t/2$, which may be handled in the same way. 
The desired bound then follows immediately from 
the change of variable $t\to 2t$ and \eqref{eq4.9i} below.

\begin{lemma}\label{l4.9} Suppose that $a\in \mathbb{R}\backslash \{ 0\}$, and define
${\bf M^+}$ as in \eqref{eq4.defM}. Then 
\begin{eqnarray}\label{eq4.9i}
\|| t \left(\partial_t S^0_{at}\nabla \right)\cdot \epsilon \nabla S^{1,\eta}_{t} f\||&\leq &
C(a)\epsilon_0\, {\bf M^+}\\\label{eq4.9ii}
\|| t^2 \left(\partial^2_t S^0_{at}\nabla \right)\cdot \epsilon \nabla S^{1,\eta}_{t} f\| |&\leq &
C(a)\epsilon_0\, {\bf M^+}.
\end{eqnarray}
Moreover, the analogous bound holds in the lower half space.
\end{lemma}
 
\begin{proof}[Proof of Lemma~\ref{l4.9}] This lemma is the deep fact underlying
estimate \eqref{eq3.6}, and the proof is rather delicate.
For the sake of notational simplicity, we treat 
only the case $a=1$, as the general case is handled by an almost identical argument. 
We begin by showing that \eqref{eq4.9ii} implies \eqref{eq4.9i}.  Set 
$${\bf J}(\sigma) \equiv \int_{\sigma}^{1/\sigma}\!\int_{\mathbb{R}^n}
\left| \partial_t \left(S^0_{t}\nabla \right)\cdot \epsilon \nabla S^{1,\eta}_{t} f\right|^2 dx\, t dt.$$
After integrating by parts in $t$, we obtain that
$${\bf J}(\sigma) =-\Re e\int_{\sigma}^{1/\sigma}\!\int_{\mathbb{R}^n}\frac{\partial}{\partial t}
\left\{\left(\partial_tS^{0}_{t}\nabla\right)\cdot\epsilon\nabla\!
S^{1,\eta}_{t}f\right\}
\overline{\left\{\left(\partial_tS^{0}_{t}\nabla\right)\cdot\epsilon\nabla\!
S^{1,\eta}_{t}f\right\}}\,dx\,t^{2}dt \,+\,\text{ ``OK"},$$
where by Lemma \ref{l2.10} $(i)$, the ``OK" boundary terms
are dominated by $C \epsilon_0^2 \sup_{t>0}
\|\nabla S_t^{1,\eta} f\|_2^2 .$
By Cauchy's inequality, modulo the ``OK" terms, 
\begin{eqnarray*}{\bf J}(\sigma)\!&
\leq&\!\frac{1}{2}{\bf J}(\sigma) +|\Vert t\left(\partial_{t}S^0_{t}\nabla\right)\cdot\epsilon
t\nabla\partial_{t}S^{1,\eta}_{t}f\Vert|^{2}
 +|\Vert
t^{2}\left(\partial_{t}^{2}S^0_{t}\nabla\right)\cdot\epsilon\nabla S^{1,\eta}_{t}f\Vert|^{2}\\
&&\qquad \equiv\frac{1}{2}{\bf J}(\sigma) + I + II .
\end{eqnarray*}
The term $\frac{1}{2}{\bf J}(\sigma)$ 
 may be hidden on the left hand side. By Lemma~\ref{l2.10} $(i)$ with $m=0$,
 term $I$  is no larger than 
$C\epsilon_{0}^{2}|\Vert t\nabla\partial_{t}S^{1,\eta}_{t}f\Vert|^{2}.$
The square root of the main term, $II$, is estimated 
in \eqref{eq4.9ii}.  Taking the latter for granted momentarily, 
we obtain \eqref{eq4.9i} by letting $\sigma \to 0.$

We now turn to the proof of \eqref{eq4.9ii}, again with $a=1$.
We make the splitting:
\begin{multline*} t^2\partial^2_t (S^0_{t}\nabla )\cdot \epsilon \nabla S^{1,\eta}_{t}f \,\,=\,\,
\sum^{n+1}_{i=1} \sum^n_{j=1} t^2 \partial^2_t (S^0_{t}D_i)\epsilon _{ij}
D_j S^{1,\eta}_{t}f\\+\,\,\,\sum^{n+1}_{i=1}
t^2 \partial^2_t (S^0_{t}D_i)\epsilon_{i,n+1}D_{n+1}S^{1,\eta}_{t}f
\,\,\equiv\,\, V_t f+\widetilde{V}_t
f.\end{multline*}
We treat $\widetilde{V}_t$ first. For ${\bf f}: \mathbb{R}^n \to \mathbb{C}^{n+1},$ set 
$$\theta_t{\bf f}\equiv t^2\partial^2_t (S^0_{t}\nabla )\cdot {\bf f},$$ and let
$\vec{\epsilon}\equiv (\epsilon_{1,n+1},\epsilon_{2,n+1},\dots, \epsilon_{n+1,n+1})$. Then, 
using a well known trick of \cite{CM}, we write 
\begin{equation*}\widetilde{V}_tf = \left\{\theta _t \vec{\epsilon} 
-(\theta _t\vec{\epsilon})P_t \right\} \partial_tS^{1,\eta}_{t}f 
\,+\, (\theta_t \vec{\epsilon}) P_t \partial_tS^{1,\eta}_{t} f \,\equiv  \,R_t^{\epsilon}\partial_t 
S^{1,\eta}_{t} f  +(\theta _t\vec{\epsilon})P_t 
\partial_tS^{1,\eta}_{t}f,\end{equation*}
where as usual $P_t$ is a nice approximate identity.  By Lemmas~\ref{l3.11}, \ref{l2.9}, \ref{l2.12} and Carleson's Lemma, the triple bar norm of the second summand is
no larger than
$C\epsilon_0 \| N_\ast \left(P_t \partial_t S^{1,\eta}_{t}f\right)\|_2.$
In addition, by Lemma~\ref{l2.15}, we have that
\begin{equation*} | \| R_t^{\epsilon} \partial_t S^{1,\eta}_{t} f\| | \,\leq \, C\epsilon _0 |\| t\nabla_\|
\partial_t S^{1,\eta}_{t}f\| |\,\leq \,C\epsilon_0|\|t\nabla\partial_t S^{1,\eta}_{t}f\| |.\end{equation*}

It remains to control $|\| V_t f\| |$, which is the primary difficulty.
By definition,
\begin{equation*}V_{t}\,=\,\theta_{t}\tilde{\epsilon}\nabla_{\|}S^{1,\eta}_{t}\, \equiv\, 
t^2\partial^2_t (S^0_{t}\nabla )\cdot\tilde{\epsilon}\nabla_{\|}S^{1,\eta}_{t} \, ,\end{equation*} where
$\tilde{\epsilon}$ is the $(n+1)\times n$ matrix $(\epsilon_{ij})_{1\leq i\leq n+1,1\leq j\leq n}$. Recall
that $A^1_{\|}$ is the $n\times n$ sub-matrix of $A^1$ with $(A^1_{\|})_{ij}=A^1_{ij}$, $1\leq i,j\leq n$, and that $(L_1)_{\|}\equiv
-\dv_{\|}A^1_{\|} \nabla_{\|}$. Then
\begin{equation*}V_{t}\,=\, \theta_{t} \tilde{\epsilon} \nabla_{\|} \left(
I-\left(I+t^2(L_1)_{\|}\right)^{-1}\right) 
S^{1,\eta}_{t} \,+\,\theta_{t} \tilde{\epsilon}\, \nabla_{\|}(I+t^2(L_1)_{\|})^{-1}
S^{1,\eta}_{t}\,\equiv\,
Y_t +Z_t.\end{equation*}

We first consider $Y_t$. Note that $\left( I-\left(I+t^2(L_1)_{\|}\right)^{-1}\right) 
=t^2(L_1)_{\|}\left(I+t^2(L_1)_{\|}\right)^{-1}$, so
\begin{equation*}Y_{t}=\theta_{t} \tilde{\epsilon} t^2\nabla_{\|}\left(I+t^2(L_1)_{\|}\right)^{-1}(L_1)_{\|} 
S^{1,\eta}_{t}.\end{equation*} As above, set $f_\eta(x,t) \equiv f(x) \varphi_\eta(t)$.
In the weak sense of Lemma \ref{l2.identity}, we then have
\begin{equation*}(L_1)_{\|} S^{1,\eta}_{t}f=\sum^n_{i=1} D_iA^1_{i,n+1} \partial_t S^{1,\eta}_{t}f+
\sum^{n+1}_{j=1} A^1_{n+1,j}D_j\partial
_tS^{1,\eta}_tf + f_\eta,\end{equation*}
and we denote by 
$Y^{(1)}_t +Y^{(2)}_t +Y_t^{(3)}$ the corresponding splitting of $Y_t$. Now, by Lemma~\ref{l2.10}, $\theta_{t}:L^2\to L^2$, and
it is well known that $t\nabla_{\|}(I+t^2(L_1)_{\|})^{-1}:L^2\to L^2$. Thus 
\begin{equation*}|\| Y^{(2)}_t f\| | \leq C\epsilon_0 | \| t\nabla \partial_t S^{1,\eta}_{t}f\| |,\end{equation*}
and also, as in the proof of \eqref{eq3.5a}, $$\||Y^{(3)}_t\|| \leq C\epsilon_0\||tf_\eta\|| \leq C\epsilon_0 \|f\|_{L^2(\mathbb{R}^n)}.$$

We make a further decomposition of $Y^{(1)}_t$ as follows:
\begin{equation*} Y^{(1)}_t = \left(U_t\vec{a}-(U_t\vec{a})P_t \right)\partial_tS^{1,\eta}_{t}
+(U_t\vec{a})P_t\partial_t
S^{1,\eta}_{t}\,\equiv \,\widetilde{R}_t \partial_t S^{1,\eta}_{t}+
(U_t\vec{a})P_t \partial_t S^{1,\eta}_{t}\,,\end{equation*} where
\begin{equation}\label{eq4.10} U_t\vec{g}\,\equiv\, \theta_{t}\tilde{\epsilon}
t^2\nabla_{\|}\left(I+t^2(L_1)_{\|}\right)^{-1}\dv_\| \vec{g} ',\end{equation} and
$\vec{a}\equiv (A^1_{1,n+1}, A^1_{2,n+1},\dots ,A^1_{n,n+1})$. We now claim that
\begin{equation}\label{eq4.11} |\| U_t \| |_{op}\leq C\epsilon_0\end{equation} Let us momentarily defer the proof of
this claim.  It is a standard fact that for two sets $E$ and $E'\subseteq
\mathbb{R}^n$, with $\vec{g}$ supported in $E'$, we have
\begin{equation}\label{eq4.12} \left\| t^2\nabla_{\|} \left(1+t^2(L_1)_{\|}\right)^{-1} \dv_{\|} \vec{g}\right\|_{L^2(E)}\leq C\exp \left\{
\frac{-\dist (E,E')}{Ct}\right\} \| \vec{g}\|_{L^2 (E')}\end{equation} (the corresponding fact for the operator $t\nabla_\|\left(1+t^2(L_1)_{\|}\right)^{-1}$ is proved in \cite{AHLMcT} for example, and
\eqref{eq4.12} may be readily deduced from this fact plus the same argument). Thus, by Lemma~\ref{l2.13}, the operator
$U_t$ satisfies \eqref{eq2.11}, with a bound on the order of $C\epsilon_0$, whenever $t\leq c\ell (Q)$. Therefore, by
Lemma~\ref{l2.12} and Carleson's Lemma, we have that
\begin{equation*} | \| (U_t\vec{a})P_t \partial _t S^{1,\eta}_{t} f\| | \leq C\epsilon_0 \| N_\ast (P_t
\partial_t S^{1,\eta}_{t}f)\| _2.\end{equation*} 
Moreover, by Lemmas \ref{l2.15} and \ref{l2.19}, we have that
$$\||\widetilde{R}_t \partial_t S^{1,\eta}_{t} f \|| \leq C \epsilon_0 \||t \nabla_{\|} \partial_t S^{1,\eta}_{t} f \|| 
\leq C \epsilon_0 \|| t \nabla\partial_t S^{1,\eta}_{t} f \||.$$

To finish our treatment of $Y_t$, it remains to prove \eqref{eq4.11}.  We continue
to defer the proof of this estimate for the moment, and proceed to discuss the term $Z_t$.
We write 
 \begin{multline*}
Z_t\, = \,\theta_{t} \tilde{\epsilon}\,\nabla_{\|} \left(I + t^2 (L_1)_\|\right)^{-1}
(S^{1,\eta}_{t} -S^{1,\eta}_{0}) \\ + \,
\theta_{t} \tilde{\epsilon}\,\nabla_{\|}\left( \left(I + t^2 (L_1)_\|\right)^{-1} - I\right)S^{1,\eta}_{0} \,+\,
\theta_{t} \tilde{\epsilon}\,\nabla_{\|}S^{1,\eta}_{0}\,\,\equiv\,\,Z_t^{(1)} + Z_t^{(2)} + Z_t^{(3)}.
\end{multline*}
By Lemma~\ref{l3.11} with $m=1$, we have that
\begin{equation*} | \| Z^{(3)}_t f\| | \leq C\epsilon_0 \,\sup_{t>0} \| \nabla
S^{1,\eta}_{t} f\|_2 .\end{equation*} Also,
\begin{equation*}Z^{(2)}_t =\theta_{t}\tilde{\epsilon}\,
\nabla_{\|} \left(I + t^2 (L_1)_\|\right)^{-1}t^2\dv_{\|} A^1_{\|}\nabla_{\|}S^{1,\eta}_{0}\equiv
U_t A^1_{\|}\nabla_{\|}S^{1,\eta}_{0}\end{equation*} (see \eqref{eq4.10}), so by the deferred estimate \eqref{eq4.11} we have that
\begin{equation*}|\| Z^{(2)}_t f\| | \leq C\epsilon_0 \sup_{t>0}\| \nabla S^{1,\eta}_{t}
f\|_{2}.\end{equation*} Integrating by parts, we obtain 
\begin{multline*}Z^{(1)}_t =\theta_{t}\tilde{\epsilon} \,\nabla_{\|} \left(I + t^2 (L_1)_\|\right)^{-1}\!\int^t_0 \partial_s S^{1,\eta}_{s} ds\,=\,-\theta_{t} \tilde{\epsilon}\,
\nabla_{\|}\left(I + t^2 (L_1)_\|\right)^{-1}\!\int^t_0 s\partial^2_s S^{1,\eta}_{s}
ds\\+\,\,\theta_{t}\tilde{\epsilon}\,
\nabla_{\|}\left(I + t^2 (L_1)_\|\right)^{-1}
t\partial_t S^{1,\eta}_{t}\,\,\equiv\,\, \Omega^{(1)}_t +\Omega^{(2)}_t.\end{multline*} By
Lemma~\ref{l2.13}, and the fact that $\nabla_{\|}(I+t^2(L_1)_{\|})^{-1}1=0$, we have that the operator 
$$R_t\equiv
\theta_{t}\tilde{\epsilon}\,t\nabla_{\|}\,\left(I + t^2 (L_1)_\|\right)^{-1}$$ 
satisfies the hypothesis of Lemma~\ref{l2.15}, with a bound on the
order of $C\epsilon_0$, so that  
\begin{equation*}|\| \Omega^{(2)}_t f\| | \leq C\epsilon_0 \,| \| t\nabla\partial_t S^{1,\eta}_{t}f\| |.
\end{equation*} Furthermore,
\begin{equation*} \Omega^{(1)}_t =-\int^t_0\frac{s}{t}\,\theta_{t}\tilde{\epsilon}\, t\nabla_\|
\left(I + t^2 (L_1)_\|\right)^{-1}s\partial^2_sS^{1,\eta}_{s}
\frac{ds}{s},\end{equation*} so by Lemma~\ref{l2.20}, we have 
\begin{equation*}|\| \Omega^{(1)}_t f\| |\leq C\epsilon_0 \,|\|t\partial^2_tS^{1,\eta}_{t}f\| |.
\end{equation*} Modulo
\eqref{eq4.11}, this concludes the proof of Lemma~\ref{l4.9}, and hence also that of \eqref{eq3.6}.

We conclude the present section by proving~\eqref{eq4.11}. The proof will depend on some technology from the proof of
the Kato square root conjecture. By ellipticity, it is enough to show that \begin{equation*} |\Vert
U_{t}A_{\|}^1\vec{g}\Vert|\leq C\epsilon_{0}\|\vec{g}\|_{2}\end{equation*}
 for $\vec{g}\in L^{2}(\mathbb{R}^{n},\mathbb{C}^{n})$. But
$$U_{t}A^1_{\|}=\theta_{t}\tilde{\epsilon}\,t^{2}\nabla_{\|}\left(I + t^2 (L_1)_\|\right)^{-1}\dv_{\|}A^1_{\|},$$ so by the Hodge decomposition
\cite[p. 116]{AT}, we may replace $\vec{g}$ by $\nabla_{\|}F$, where $\Vert\nabla_{\|}F\Vert_{2}\leq C\Vert g\Vert_{2}$. As usual, by density we may suppose
that $F \in C_0^{\infty}.$  Now
\begin{equation*}
U_{t}A^1_{\|}\nabla_{\|}F  \,=\,-\theta_{t}\tilde{\epsilon}\,
\nabla_{\|}\left(I + t^2 (L_1)_\|\right)^{-1}(t^{2}(L_1)_{\|})F 
\,=\,\theta_{t}\tilde{\epsilon}\,\nabla_{\|}\left(\left(I + t^2 (L_1)_\|\right)^{-1}-I\right)F.
\end{equation*}
We recall that $\theta_t =t^2\partial^2_t (S^0_{t}\nabla)\,\cdot \, $, so by 
Lemma~\ref{l3.11} with $m=1$, 
\begin{equation*} |\Vert\theta_{t}\tilde{\epsilon}\,\nabla_{\|}F\Vert|\leq
C\epsilon_{0}\Vert\nabla_{\|}F\Vert_{2}.\end{equation*}
 The main term is \begin{equation*}
\theta_{t}\tilde{\epsilon}\,\nabla_{\|}\left(I + t^2 (L_1)_\|\right)^{-1}F\equiv\frac{1}{t}R_{t}F,\end{equation*}
 where by Lemmas~\ref{l2.9}, \ref{l2.10}, and \ref{l2.13}, and the fact that 
$\nabla_{\|}(I+t^{2}L_{2}')^{-1}1=0$, we have that $R_{t}$ satisfies the hypotheses of Lemma~\ref{l2.17}, with a
bound on the order of $C\epsilon_{0}$. Therefore, it suffices to prove the Carleson measure estimate \begin{equation*}
\int_{0}^{\ell (Q)}\!\!\int_{Q}\left|\frac{1}{t}R_{t}\Phi(x)\right|^{2}\frac{dxdt}{t}\leq C\epsilon_{0}|Q|,\end{equation*}
 where $\Phi(x)\equiv x$. To this end, we write \begin{equation}
\frac{1}{t}R_{t}\Phi=\theta_{t}\tilde{\epsilon}\,\nabla_{\|}\left(\left(I + t^2 (L_1)_\|\right)^{-1}-I\right)\Phi+
\theta_{t} \tilde{\epsilon}\,\nabla_{\|}\Phi.\label{eq4.13}\end{equation}
 But $\nabla_{\|}\Phi=\mathbb{I}$, the $n\times n$ identity matrix. Thus, Lemmas~\ref{l3.11}, \ref{l2.9} and \ref{l2.12}
yield the bound \begin{equation*}
\int_{0}^{\ell (Q)}\!\!\int_{Q}|\theta_{t}\tilde{\epsilon}\nabla_{\|}\Phi|^{2}\frac{dxdt}{t}\leq C\epsilon_{0}|Q|.\end{equation*}
 The remaining term in \eqref{eq4.13} equals \begin{equation*}
\theta_{t}\tilde{\epsilon}\,t^{2}\nabla_{\|}\left(I + t^2 (L_1)_\|\right)^{-1}\dv_{\|}A_{\|}
\nabla_{\|}\Phi 
\,=\,\theta_{t}\tilde{\epsilon}\,t^{2}\nabla_{\|}\left(I + t^2 (L_1)_\|\right)^{-1}\dv_{\|}A_{\|} 
\,\equiv \,T_{t}A_{\|}
\end{equation*}
We now invoke a key fact in the proof of the Kato conjecture. By \cite{AHLMcT}, there exists, for each $Q$, a mapping
$F_{Q}=\mathbb{R}^{n}\rightarrow \mathbb{C}^{n}$ such that \begin{equation}
\begin{split}\label{eq4.14}\text{(i)} & \quad\int_{\mathbb{R}^{n}}|\nabla_{\|}F_{Q}|^{2}\leq C|Q|\\
\text{(ii)} & \quad\int_{\mathbb{R}^{n}}|(L_1)_{\|}F_{Q}|^{2}\leq C\frac{|Q|}{\ell(Q)^{2}}\\
\text{(iii)} & \quad\sup_{Q}\int_{0}^{\ell(Q)}\fint_{Q}|\vec{\zeta}(x,t)|^{2}\frac{dxdt}{t}\\
 & \quad\quad\leq
C\sup_{Q}\int_{0}^{\ell(Q)}\fint_{Q}|\vec{\zeta}(x,t)E_{t}\nabla_{\|}F_{Q}(x)|^{2}\frac{dxdt}{t},\end{split}
\end{equation}
 for every function $\vec{\zeta}:\mathbb{R}_{+}^{n+1}\rightarrow\mathbb{C}^{n}$, where $E_{t}$ denotes the dyadic
averaging operator, i.e. if $Q(x,t)$ is the minimal dyadic cube (with respect to the grid induced by $Q$) containing $x$,
with side length at least $t$, then \begin{equation*} E_{t}g(x)\equiv\fint_{Q(x,t)}g.\end{equation*}
 Here $\nabla_{\|}F_{Q}$ is the Jacobian matrix $(D_{i}(F_{Q})_{j})_{1\leq i,j\leq n}$, and the product
$$\vec{\zeta}E_{t}\nabla_{\|}F_{Q}=\sum_{i=1}^{n}\zeta_{i}E_{t}D_{i}F_{Q}$$ is a vector. Given the existence of a family of
mappings $F_{Q}$ with these properties, as in \cite[Chapter 3]{AT}, we see by (iii), applied with
$\vec{\zeta}(x,t)=T_{t}A_{\|}$, that it is enough to show that \begin{equation*}
\int_{0}^{\ell(Q)}\int_{Q}|T_{t}A_{\|}(x)\,E_{t}\nabla_{\|}F_{Q}(x)|^{2}\frac{dxdt}{t}\leq C\epsilon_{0}|Q|.\end{equation*}
 But as in \cite{AT}, we may exploit the idea of \cite{CM} to write
\begin{equation*}
\begin{split}(T_{t}A_{\|})E_{t}\nabla_{\|}F_{Q} & =\left\{(T_{t}A_{\|})E_{t}-T_{t}A_{\|}\right\}\nabla_{\|}F_{Q}+T_{t}A_{\|}\nabla_{\|}F_{Q}\\
 & =(T_{t}A_{\|})(E_{t}-E_tP_{t})\nabla_{\|}F_{Q}+\left\{(T_{t}A_{\|})E_tP_{t}-T_{t}A_{\|}\right\}\nabla_{\|}F_{Q}+T_{t}A_{\|}\nabla_{\|}F_{Q}\\
 & \equiv R_{t}^{(1)}\nabla_{\|}F_{Q}+R_{t}^{(2)}\nabla_{\|}F_{Q}+T_{t}A_{\|}\nabla_{\|}F_{Q},\end{split}
\end{equation*}
 where $P_{t}$ is a nice approximate identify.  The last term is easy to handle. We have that 
\begin{equation*}
T_{t}A_{\|}\nabla_{\|}F_{Q}=-\theta_{t}\tilde{\epsilon}t\nabla_{\|}\left(I + t^2 (L_1)_\|\right)^{-1}t(L_1)_{\|}F_{Q}.\end{equation*}
 Therefore, since $\theta_{t}$ and $t\nabla_{\|}(I+t^{2}(L_1)_{\|})^{-1}$ are uniformly bounded on $L^{2}$, we obtain that
\begin{equation*}
\int_{0}^{\ell(Q)}\int_{Q}|T_{t}A_{\|}\nabla_{\|}F_{Q}|^{2}\frac{dxdt}{t}  \leq
C\epsilon_{0}\int_{\mathbb{R}^{n}}|(L_1)_{\|}F_{Q}|^{2}\int_{0}^{\ell(Q)}tdtdx\leq C\epsilon_{0}|Q|,
\end{equation*}
 where in the last step we have used \eqref{eq4.14}(ii).

It is also easy to handle $R_{t}^{(1)}\nabla_{\|}F_{Q}$. Indeed $E_{t}=E_{t}^{2}$, so
that \begin{equation*} R_{t}^{(1)}=(T_{t}A_{\|})E_{t}(E_{t}-P_{t})\end{equation*}
 By the definition of $T_{t}$, Lemma~\ref{l2.13} and Lemma~\ref{l2.19}, we have that
\begin{equation}\label{eq6.smallbound}
\Vert(T_{t}A_{\|})E_{t}\Vert_{2\rightarrow2}\leq C\epsilon_{0}.\end{equation} 
Thus, \begin{equation*}
\int_{0}^{\ell(Q)}\int_{Q}|R_{t}^{(1)}\nabla_{\|}F_{Q}|^{2}dx\frac{dt}{t}  \leq
C\epsilon_{0}\iint_{\mathbb{R}_+^{n+1}}|(E_{t}-P_{t})\nabla_{\|}F_{Q}|^{2}\frac{dxdt}{t}
 \leq C\epsilon_{0}|Q|,
\end{equation*}
 where in the last step we have used \eqref{eq4.14}(i), as well as the boundedness on $L^{2}$ of 
\begin{equation*}
g\rightarrow\left(\int_{0}^{\infty}|(E_{t}-P_{t})g|^{2}\frac{dt}{t}\right)^{\frac{1}{2}}.\end{equation*}
 It remains to treat the contribution of the term $R_{t}^{(2)}\nabla_{\|}F_{Q}$. By \eqref{eq4.14}(i), it will be enough to
establish the square function bound \begin{equation*} |\Vert R_{t}^{(2)}\nabla_{\|}F_{Q}\Vert|\leq
C\epsilon_{0}\Vert\nabla_{\|}F_{Q}\Vert_2.\end{equation*}
To this end, we write \begin{equation}
R_{t}^{(2)}\nabla_{\|}F_{Q}=R_{t}^{(2)}(I-P_{t})\nabla_{\|}F_{Q}\,
+\,R_{t}^{(2)}P_{t}\nabla_{\|}F_{Q},\label{eq4.15}\end{equation}
where $I$ denotes the identity operator.
The last term is easy to handle. We note that $R_{t}^{(2)}1=0$, and therefore by Lemmas~\ref{l2.9},
\ref{l2.10}, \ref{l2.13} and
\ref{l2.19}, the operator $R_{t}^{(2)}$ satisfies the hypotheses of Lemma~\ref{l2.15} with bound on the order of
$C\epsilon_{0}$. Thus, \begin{equation*}
|\Vert R_{t}^{(2)}P_{t}\nabla_{\|}F_{Q}\Vert|  \leq C\epsilon_{0}\,|\Vert t\nabla_{\|}P_{t}\nabla_{\|}F_{Q}\Vert| \leq C\epsilon_{0}\,\Vert\nabla_{\|}F_{Q}\Vert_2,
\end{equation*}
where the last inequality is standard Littlewood-Paley theory. 

By the definition of $R_{t}^{(2)}$, we may further decompose the first summand on
the right side of \eqref{eq4.15} as \begin{equation*}
(T_{t}A_{\|})E_tQ_{t}\nabla_{\|}F_{Q}\,-\,T_{t}A_{\|}\nabla_{\|}(I-P_{t})F_{Q}
\equiv {\bf I} - {\bf II},
\end{equation*} where $Q_{t}\equiv P_{t}(I-P_{t})$ satisfies $|\|
Q_{t}\||_{op}\leq C$. Then by \eqref{eq6.smallbound}, we have 
$$|\|{\bf I}\||\leq C\epsilon_{0}\|\nabla_{\|}F_{Q}\|_{2}.$$

Next, by definition of $T_{t}$, we see that\begin{multline*} {\bf II}= 
\,\theta_{t}\tilde{\epsilon}\,\nabla_{\|}\left(\left(I + t^2 (L_1)_\|\right)^{-1}-I\right)(I-P_{t})F_{Q} \,=\, 
-\,\theta_{t}\tilde{\epsilon}\,\nabla_{\|}F_{Q}\\+\,\theta_{t}\tilde{\epsilon}\,\nabla_{\|}P_{t}F_{Q}
\,\,+\,\,\theta_{t}\tilde{\epsilon}\,\nabla_{\|}\left(I + t^2 (L_1)_\|\right)^{-1}(I-P_{t})F_{Q} \equiv 
\,{\bf II}_{1}+{\bf II}_{2}+{\bf II}_{3}.\end{multline*} By
Lemma~\ref{l3.11}, $$\|| {\bf II}_{1}\||\leq C\epsilon_{0}\|\nabla_{\|}F_{Q}\|.$$ Moreover, by Lemma~\ref{l2.10} and the fact that
$\|t\nabla_{\|}(1+t^{2}(L_1)_{\|})^{-1}\|_{2\to2}\leq C$, we obtain that\begin{equation*} |\| {\bf II}_{3}\||\leq
C\epsilon_{0}|\Vert t^{-1}I_{1}(I-P_{t})\sqrt{-\Delta}F_{Q}\Vert |\leq
C\epsilon_{0}\|\nabla_{\|}F_{Q}\|_{2},\end{equation*} where $I_{1}=(-\Delta)^{-1/2}$ is the fractional integral
operator of order one on $\mathbb{R}^{n}$, and where we have used the Littlewood-Paley inequality
$$|\Vert t^{-1}I_{1}(I-P_{t})\Vert
|_{op}\leq C.$$ The latter estimate holds by Plancherel's Theorem, since
\begin{equation}
\left|\frac{1}{t|\xi|}(1-\hat{\phi}(t\xi))\right|\,\leq \,
C\min\left(t|\xi|,\frac{1}{t|\xi|}\right),\label{eq4.16}\end{equation}
if $\phi_t(x)=t^{-n}\phi_t\left(x/t\right)),$
the convolution kernel of $P_{t}$, is chosen so that 
$\int_{\mathbb{R}^{n}}x\phi_t(x)dx=0$.

Finally, it remains only to consider the term ${\bf II}_{2}$. Now 
$${\bf II}_{2}=\theta_{t}\tilde{\epsilon}P_{t}\nabla_{\|}F_{Q},$$ 
so we need that $|\|\theta_{t}\tilde{\epsilon}P_{t}\||_{op}\leq C\epsilon_{0}$.  By Lemmas~\ref{l3.11},
\ref{l2.12} and \ref{l2.9},  $|\theta_{t}\tilde{\epsilon}|^{2}t^{-1}dxdt$ is a Carleson measure with norm at most $C\epsilon_{0}^{2}$, so it is enough to bound
$|\|\theta_{t}\tilde{\epsilon}P_{t}-(\theta_{t}\tilde{\epsilon})P_{t}\||_{op}.$
We may choose $P_t$ to be of the form $P_{t}=\tilde{P}_{t}^{2}$,  where $\tilde{P}_t$ is of the same type.  Set
$$R_{t}\equiv\theta_{t}\tilde{\epsilon}
\tilde{P}_{t}-(\theta_{t}\tilde{\epsilon})\tilde{P}_{t},$$ which satisfies the
hypothesis of Lemma~\ref{l2.15} with bound $C\epsilon_{0}$. Thus, \begin{equation*}
|\|\theta_{t}\tilde{\epsilon}P_{t}-(\theta_{t}\tilde{\epsilon}) P_{t}\||_{op}\equiv|\|
R_{t}\tilde{P}_{t}\||_{op}\leq C\epsilon_{0}|\| t\nabla\tilde{P}_{t}\||_{op}\leq C\epsilon_{0}.\end{equation*} This concludes the proof of Lemma \ref{l4.9}, and hence 
that of the square function bound~\eqref{eq3.6}.
\end{proof}

\section{Proof of Theorem \ref{t1.10}: the singular integral estimate \eqref{eq3.7} \label{s5}}

We shall consider separately the cases $t>0$ and $t<0$, and since the proof is the same in each case we treat only the
former. More precisely, we shall prove
\begin{equation}\label{eq5.1}\sup_{0<\eta<10^{-10} }
\sup_{t>0}\|\partial_t S^{1,\eta}_{t}f\|_2 \leq C\| f\|_2 +C\epsilon_0
\left({\bf M^+} + {\bf M^-}\right),\end{equation} 
where ${\bf M}^\pm$ are defined in \eqref{eq4.defM}. 
We begin by reducing matters to the case $t\geq 4\eta$.  Suppose that
$0\leq t < 4\eta$.  We claim that
$$|\partial_t S_t^{1,\eta} f(x) - D_{n+1} S_{4\eta}^{1,\eta}f(x)| \leq C Mf(x).$$
Indeed, let $K^\eta_t(x,y)$ denote the kernel of $\partial_t S_t^{1,\eta}$, i.e.,
$$K_t^\eta(x,y) \equiv \partial_t\left(\varphi_\eta \ast \Gamma_1(x,\cdot,y,0)\right)(t).$$
To prove the claim, it is enough to establish the following estimate: 
$$|K_t^\eta(x,y) - K^\eta_{4\eta}(x,y)| \leq C \left(\frac{1_{\{|x-y|\leq 10 \eta\}}}{\eta |x-y|^{n-1}} +
\frac{\eta}{|x-y|^{n+1}}1_{\{|x-y|>10\eta\}}\right).$$
In turn, the case $|x-y|\leq 10\eta$ of the latter bound follows directly from \eqref{eq4.ntkernel1}.
On the other hand, if $|x-y|>10\eta$, we have by Lemma \ref{l2.1} that
\begin{multline*}|K_t^\eta(x,y) - K^\eta_{4\eta}(x,y)| =\left|\int \varphi_\eta(s) \left(
D_{n+1} \Gamma(x,t-s,y,0) - D_{n+1}\Gamma(x,4\eta -s,y,0)\right)ds\right|\\
\leq C \int|\varphi_\eta(s)| \frac{|4\eta-t|}{|x-y|^{n+1}} ds\leq C\frac{\eta}{|x-y|^{n+1}}.\end{multline*}

Having proved the claim, we fix $t_0\geq 4\eta$, and use \eqref{eq1.3} to obtain, for each $y\in
\mathbb{R}^n$,
\begin{equation*}\begin{split}|(D_{n+1}S^{1,\eta}_{t_0}) 
f(y)|&\leq
C\left(\fiint_{B((y,t_0),t_0/2)}
|\partial_\tau S^{1,\eta}_{\tau}f(x)|^2 dxd\tau
\right)^{\frac{1}{2}}\\ &\leq C\left( \fiint_{B((y,t_0),t_0/2)} |\partial_\tau S^{1,\eta}_{\tau
}f-\partial_\tau S^{0,\eta}_{\tau }f|^2 dxd\tau 
\right)^{\frac{1}{2}}+\,\ok,\end{split}\end{equation*} where $\|\ok\|_{L^2(\mathbb{R}^n)}\leq C\|f\|_2$ uniformly in
$t_0$, by our hypotheses regarding $L_0$, and where we have used that 
$u_\eta(x,t) \equiv S_t^{1,\eta}f(x)$ solves $L_1 u_\eta = 0$ in $\{t>\eta\}$.
Consequently, \begin{equation*}\| (D_{n+1} S^{1,\eta}_{t_0})f\|^2_2\leq C
\|f\|_2 \,+\,C\frac{1}{t_0}\int^{3t_0/2}_{t_0/2}\!\!\int_{\mathbb{R}^n}
|\partial_\tau S^{1,\eta}_{\tau }f-\partial_\tau S^{0,\eta}_{\tau}f|^2dxd\tau .\end{equation*} 
As in the section \ref{s4},  \begin{equation*}  \partial_\tau S^{1,\eta}_{\tau }f(x)-\partial_\tau S^{0,\eta}_{\tau}f(x)
=\partial_\tau\left(L_0^{-1}\dv \epsilon \nabla S_{(\cdot)}^{1,\eta} f\right)(x,\tau).
\end{equation*}
Thus, it is enough to prove that for every
$\Psi \in C_0^\infty\left(\mathbb{R}^n \times (\frac{t_0}{2},\frac{3t_0}{2})\right),$
with $t_0^{-1/2}\|\Psi\|_2 =1$,
and for each $\eta > 0 $ and $\delta>0$ sufficiently small, we have
\begin{equation} \label{eq5.2a}\left|\frac{1}{t_0}\iint_{\mathbb{R}^{n+1}} 
\epsilon(y)\nabla S_s^{1,\eta}f(y)\cdot\overline{\nabla (L_0^*)^{-1}
\left(D_{n+1}\Psi_\delta\right)(y,s)} dy ds\right|
\leq C\epsilon_0({\bf M^+} + {\bf M^-}), 
\end{equation}
where again $\Psi_\delta \equiv \varphi_\delta \ast \Psi$.  We may then 
obtain \eqref{eq5.1} by taking first a limit as $\delta \to 0$, and then a supremum over all such $\Psi.$

To prove \eqref{eq5.2a}, we begin by splitting the integral on the left hand side into
\begin{equation}
\frac{1}{t_0}\left\{ \int_{-t_0/4}^{t_0/4}\!\int_{\mathbb{R}^n} \,+\, \int_{t_0/4}^{4t_0}\!\int_{\mathbb{R}^n}
\,+\, \int_{4t_0}^\infty\!\int_{\mathbb{R}^n} \,+\,  \int_{-\infty}^{-t_0/4}\!\!\int_{\mathbb{R}^n}\right\}
\equiv I + II + III + IV.\label{eq5.3a}\end{equation}   
Since $\nabla (L_0^*)^{-1}\dv$ is bounded on $L^2(\mathbb{R}^{n+1})$, by Cauchy-Schwarz and our assumptions on $\Psi$,
we have that $$|II| \leq C \epsilon_0 \left(t_0^{-1} 
\int_{t_0/4}^{4t_0}\!\int_{\mathbb{R}^n}|\nabla S_s^{1,\eta}f(y)|^2 dy ds \right)^{1/2}
\leq C \epsilon_0 \sup_{t>0} \|\nabla S_t^{1,\eta}f\|_2.$$

Next we consider terms $III$ and $IV$.  These may be handled in the same way, so we treat only
$III$ explicitly.  We use \eqref{eq4.etaidentity} to write 
\begin{equation}\nabla (L_0^*)^{-1}(D_{n+1}\Psi_\delta)(y,s)=\int 
\nabla_{y,s}\partial_sS^{L_0^*,\delta}_{s-\tau}\left(\Psi(\cdot,\tau)\right)(y) \,d\tau,\label{eq5.deltaidentity}
\end{equation}
so that \begin{multline*}
III = t_0^{-1} \int\!\int_{4t_0}^\infty\!\int_{\mathbb{R}^n} 
\left(\partial_\tau S^0_{\tau-s} \nabla\right) \cdot
\epsilon \nabla S_s^{1,\eta} f (x) \Psi_\delta (x,\tau) dx ds d\tau \\=
 \frac{1}{t_0}\int\!\int_{2\tau}^\infty\!\int_{\mathbb{R}^n}\,\,\, -\quad\frac{1}{t_0}
 \int\!\int_{2\tau}^{4t_0}\!\int_{\mathbb{R}^n} 
 \equiv \widetilde{III} \,- \text { error}.
\end{multline*}
In the error term, $s-\tau \approx s \approx \tau \approx t_0$, if $\delta$ is sufficiently small,
given the support constraints on $\Psi$.  Thus by Cauchy-Schwarz and Lemma \ref{l2.10} $(i)$,
the absolute value of the error term is bounded by $C \epsilon_0 \sup_{t>0}\|\nabla S_t^{1,\eta}f\|_2.$
The remaining term is
\begin{multline*}
\widetilde{III}=t_0^{-1} \int\!\int_{2\tau}^\infty\!\int_{\mathbb{R}^n} 
\left(\partial_\tau S^0_{\tau-s} \nabla\right) \cdot
\epsilon \nabla S_s^{1,\eta} f (x) \Psi_\delta (x,\tau) dx ds d\tau\\
=t_0^{-1} \int\!\lim_{R \to \infty}\int_{\mathbb{R}^n}\! \left\{\int_{2\tau}^{2R}
\left(\partial_\tau S^0_{\tau-s} \nabla\right) \cdot
\epsilon \nabla S_s^{1,\eta} f (x) ds \right\}\,\Psi_\delta (x,\tau)  dx d\tau\\
\equiv\,t_0^{-1} \int\!\lim_{R \to \infty}\int_{\mathbb{R}^n}H_R(x,\tau)\,\Psi_\delta (x,\tau)  dx d\tau,
\end{multline*}
where the expression in curly brackets equals
\begin{multline*}H_R(x,\tau) = -\int_{\tau}^R\partial_t\left(\int_{2t}^{2R}
\left(\partial_t S^0_{t-s} \nabla\right) \cdot
\epsilon \nabla S_s^{1,\eta} f (x) ds\right) dt \\
=-\int_{\tau}^R\partial_t\left(\int_{t}^{2R-t}
\left(D_{n+1}S^0_{-s} \nabla\right) \cdot
\epsilon \nabla S_{t+s}^{1,\eta} f (x) ds\right) dt\\= \int_{\tau}^R
\left(D_{n+1} S^0_{-t} \nabla\right) \cdot
\epsilon \nabla S_{2t}^{1,\eta} f (x)  dt -\int_{\tau}^R
\left(D_{n+1} S^0_{t-2R} \nabla\right) \cdot
\epsilon \nabla S_{2R}^{1,\eta} f (x)  dt \\
-\int_{\tau}^R\left(\int_{t}^{2R-t}
\left(D_{n+1}S^0_{-s} \nabla\right) \cdot
\epsilon \nabla \partial_tS_{t+s}^{1,\eta} f (x) ds\right) dt \\\equiv H_R'(x,\tau)-H_R''(x,\tau)
-H_R'''(x,\tau).
\end{multline*}
Since $|t-2R| \approx R$,  we have that by Lemma \ref{l2.10} $(i)$,
$$\sup_{\tau,R: 0<\tau<R}\|H_R''(\cdot,\tau)\|_2 \leq  C \epsilon_0 \sup_{t>0}\|\nabla S_t^{1,\eta}f\|_2,$$
from which the desired bound for the corresponding part of $\widetilde{III}$ follows readily.
Similarly, we may treat the contribution of $H_R'(x,\tau)$ by a direct application of the following
Lemma, which is really the deep result in this section.

\begin{lemma}\label{l5.2} Let $a,b$ denote non-zero real constants. We then have that
\begin{equation*} \sup_{0\leq \tau_1<\tau _2< \infty} \left\|\int^{\tau _2}_{\tau _1}
\left(D_{n+1}S^0_{at}\nabla \right)\cdot\epsilon \nabla S^{1,\eta}_{bt} f dt\right\|_2  
\leq \, C(a,b)\epsilon_0 \left({\bf M}^+ + {\bf M}^-\right).\end{equation*}\end{lemma}

We defer for the moment the proof of this Lemma, and consider now
\begin{equation*} H'''_R (x,\tau )=\int^R_\tau \int^{2R}_{2 t} 
\left(\partial_tS^0_{t-s}\nabla \right)\cdot \epsilon
\partial_s \nabla S_{s}^{1,\eta} f(x)\,dsdt.\end{equation*} Then for $h\in L^2(\mathbb{R}^n)$, with $\| h\|_2=1$, we have
\begin{equation}\label{eq5.3} \left|\langle h,H'''_R 
(\cdot ,\tau )\rangle \right|=\left| \int^R_\tau \int ^{2R}_{2t}\langle
\nabla D_{n+1}S^{L_0^*}_{s-t} h,\epsilon \,\partial _s \nabla S^{1,\eta}_{s} f\rangle \,dsdt\right|,\end{equation} where we have
used that $ad\!j(S^0_{t-s}) =S^{L_0^*}
_{s-t}$ (recall that $ad\!j$ indicates that we have taken the adjoint
in the $x,y$ variables only, whereas $S^{L_0^*}_{t}$ is the single layer potential operator associated to $L_0^*$). Thus, \eqref{eq5.3} is dominated by
\begin{equation*}C\epsilon_0
\left( \int^\infty_0\int^\infty_{2 t} \|\nabla \partial_s S^{L_0^*}_{s-t} h\|^2_2 \,dsdt\right)^{\frac{1}{2}} 
\left(\int^\infty_0 \| \partial _s \nabla S^{1,\eta}_{s} f\|^2_2\int^{s/2}_0dtds\right)^{\frac{1}{2}}
\equiv C\epsilon_0\,
B_1\cdot B_2.\end{equation*} Note that $B_2=C| \| s\nabla \partial _s S^{1,\eta}_{s} f\| |$. Similarly, the
change of variable $s\to s+t$ yields that $B_1=|\| s\partial_s \nabla S^{L_0^*}_{s}h\| | \leq C\| h\|_2=C$. 
A suitable bound follows for the contribution of $H'''_R$.

It remains to consider the term $I$ in \eqref{eq5.3a}, which we shall also treat via Lemma~\ref{l5.2}. 
Again using \eqref{eq5.deltaidentity}, and that for small $\delta$,
$\Psi_\delta$ is 
supported in $\{t_0/2<\tau<3t_0/2\}$, we write
\begin{multline*}
I= t_0^{-1} \int\!\int_{-t_0/4}^{t_0/4}\!\int_{\mathbb{R}^n} 
\left(\partial_\tau S^0_{\tau-s} \nabla\right) \cdot
\epsilon \nabla S_s^{1,\eta} f (x) \Psi_\delta (x,\tau) dx ds d\tau \\=
 \frac{1}{t_0}\int\!\int_{-\tau/2}^{\tau/2}\!\int_{\mathbb{R}^n}\,\,\, -\quad \frac{1}{t_0}
 \int\!\int_{t_0/4<|s|<\tau/2}\!\int_{\mathbb{R}^n} 
 \equiv \widetilde{I} \,- \text { error}.
\end{multline*}
By Cauchy-Schwarz and Lemma \ref{l2.10} $(i)$,
the absolute value of the error term is bounded by $C \epsilon_0 \sup_{t>0}\|\nabla S_t^{1,\eta}f\|_2,$
since $\tau - s \approx \tau \approx t_0$.  The remaining term splits into
\begin{multline*}
\widetilde{I}_+=
t_0^{-1} \int\!\int_{\mathbb{R}^n}\! \left\{\int_{0}^{\tau/2}
\left(\partial_\tau S^0_{\tau-s} \nabla\right) \cdot
\epsilon \nabla S_s^{1,\eta} f (x) ds \right\}\,\Psi_\delta (x,\tau)  dx d\tau\\
\equiv\,t_0^{-1} \int\!\int_{\mathbb{R}^n}F(x,\tau)\,\Psi_\delta (x,\tau)  dx d\tau,
\end{multline*}
plus a similar term $\widetilde{I}_-$, which may be treated by the same arguments,
in which the expression in curly brackets
has domain of integration $(-\tau/2,0)$.   Now,
\begin{multline*}F(\cdot,\tau) =\int^\tau_{0} 
\partial_t \left( \int^{t/2}_0 \left(\partial _t S^0_{t-s}\nabla \right)\cdot
\epsilon \nabla S^{1,\eta}_{s} fds\right) dt\\=\int^\tau_{0} \partial_t \left( \int^t_{t/2}\left(D_{n+1}
S^0_{s}\nabla \right)\cdot\epsilon \nabla
S_{t-s}^{1,\eta} fds\right) dt\\  = \int^\tau_0 \left(\partial_tS^0_{t}\nabla \right)\cdot
\epsilon \nabla S^{1,\eta}_{0}
fdt -\int^\tau_0 \left(D_{n+1} S^0_{t/2} \nabla \right)\cdot\epsilon \nabla S^{1,\eta}_{t/2} fdt\\ + \int^\tau _0
\int^{t/2}_0 \left(\partial _t S^0_{t-s}\nabla \right)\cdot 
\epsilon \nabla \partial_s S^{1,\eta}_{s} fdsdt \,\,\equiv\,\,
F'-F''+F'''.\end{multline*}
We may estimate the contribution of $F''$ directly via Lemma~\ref{l5.2}. Also,
\begin{equation*} F'(\cdot,\tau)=\left(S^0_{\tau }\nabla \right)\cdot\epsilon \nabla S^{1,\eta}_{0}f-\left(S^0_{0^+}\nabla \right)\cdot\epsilon \nabla
S^{1,\eta}_{0}f,\end{equation*} so by our hypotheses concerning $L_0$,
\begin{equation*} \sup_\tau \| F' (\cdot ,\tau )\|_{L^2(\mathbb{R}^n)} \leq C\epsilon_0 \sup_{t >0} \| 
\nabla S^{1,\eta}_{t} f\|_2.\end{equation*} We therefore obtain 
a permissible bound for the contribution of $F'$.
We also have that 
\begin{multline}\label{eq5.4}F'''(\cdot,\tau)=\int^\tau_0 \int^{t/2}_0 \!\partial_t 
\left(\left(S^0_{t-s}-S^0_{t}\right)\nabla \right)
\cdot\epsilon\nabla \partial_s S^{1,\eta}_{s}f dsdt\\+ 
\int^{\tau}_0 \left(\partial_t S^0_{t}\nabla \right)\cdot\epsilon 
\nabla S^{1,\eta}_{t/2}fdt - \int^\tau_0
\left(\partial_t S^0_{t}\nabla \right)\cdot\epsilon 
\nabla S^{1,\eta}_{0}fdt.\end{multline} In turn, the last term equals $-F'$, and the middle summand
may be handled via
Lemma~\ref{l5.2}.
The first summand on the right hand side of \eqref{eq5.4} equals
\begin{equation*}-\int^\tau_0 \int^{t/2}_0\int^s_0 \partial^2_t (S^0_{t-\sigma}\nabla )\cdot 
\epsilon \nabla \partial_s
S^{1,\eta}_{s}f(x) d\sigma ds dt.\end{equation*} Dualizing against 
$h\in L^2(\mathbb{R}^n)$, with $\| h\|_2=1$, we see
that it is enough to consider
\begin{equation*}\begin{split} &\left| \int^\tau _0 \int^\infty_0\int^\infty_0 1_{\left\{\sigma<s<t/2\right\}}
\langle \nabla D^2_{n+1} S^{L_0^*}_{\sigma-t}h,\epsilon D_{n+1}\nabla  S^{1,\eta}_{s}f\rangle
d\sigma dsdt\right|\\ &\quad \leq C\epsilon _0 \left( \int^\infty_0 \int^\infty_0\int^\infty_0 1_{\left\{
\sigma<s<t/2\right\}} s^{-\frac{1}{2}} t^{\frac{3}{2}} \| \nabla D^2_{n+1}S^{L_0^*}_{\sigma-t}h\|^2_2\,
d\sigma dsdt\right)^{\frac{1}{2}}\\&\qquad\qquad \times\quad\left( \int^\infty_0
\int^\infty_0\int^\infty_0 1_{\left\{ \sigma<s<t/2\right\} } s^{\frac{1}{2}} t^{-\frac{3}{2}} \| \partial_s \nabla
S^{1,\eta}_{s} f\|^2_2 d\sigma dsdt\right)^{\frac{1}{2}}\\ &\quad \equiv C\epsilon _0 B_3\cdot B_4.\end{split}\end{equation*} Now, 
\begin{equation*}B_4= \left( \int^\infty_0 \| \partial_s \nabla S^{1,\eta}_{s}f\|^2_2 \left(
\int^s_0d\sigma\int^\infty_{2s} \frac{s^{1/2}}{t^{3/2}}dt\right) ds\right)^{\frac{1}{2}} 
=C|\| s\nabla \partial_sS^{1,\eta}_{s}f\| |.\end{equation*}  
Similarly, the change of variable $t\to
t+\sigma$ yields the bound
\begin{multline*}B_3 =\left( \int^\infty_0 \int^\infty_0 \int^\infty_0 1_{\left\{
\sigma<s<(t+\sigma)/2\right\} } s^{-\frac{1}{2}}(t+\sigma)^{\frac{3}{2}} 
\| \nabla D^2_{n+1} S^{L_0^*}_{-t} h\|^2_2
d\sigma dsdt\right)^{\frac{1}{2}}\\ \leq C\left( \int^\infty_0 t^{\frac{3}{2}}\| \nabla 
\partial^2_t S^{L_0^*}_{-t} h\|^2_2
\int^t_0 s^{-\frac{1}{2}}
\int^s_0 d\sigma ds dt\right)^{\frac{1}{2}}= C| \| t^2\nabla \partial^2_t S^{L_0^*}_{-t} h\| |\leq C\|
h\|_2=C,\end{multline*}
and the desired estimate for the contribution of $F'''$ now follows.

To complete the proof of estimate \eqref{eq3.7}, it therefore remains to prove Lemma~\ref{l5.2}. 
\begin{proof}[Proof of Lemma \ref{l5.2}]For the sake of
simplicity of notation, we shall treat the case $a=2$, $b=1$, as the general case follows via the same argument.

As above we dualize against $h\in L^2 (\mathbb{R}^n)$, so that it is enough to consider
\begin{multline} \label{eq5.5} 
\int^{\tau _2}_{\tau_1}\langle \nabla \partial_t S^{L_0^*}_{-2t} h,
\epsilon \nabla S^{1,\eta}_{t}f\rangle
dt =- \int^{\tau_2}_{\tau_1} \langle \nabla \partial^2_tS^{L_0^*}_{-2t}h,
\epsilon \nabla S^{1,\eta}_{t} f\rangle tdt\\  -\int^{\tau_2}_{\tau_1} \langle 
\nabla \partial_t S^{L_0^*}_{-2t}
h,\epsilon \nabla \partial_t S^{1,\eta}_{t}f\rangle tdt\,+\,\text{boundary},\end{multline} 
where we have integrated by parts in $t$, and where the boundary term is
dominated by 
\begin{equation*} C\epsilon _0 \left(\sup_{\tau >0} 
\| \tau \nabla \partial_\tau S^{L_0^*}_{-2\tau} h\|_2\right)\left(
\sup_{\tau >0} \| \nabla S^{1,\eta}_{\tau}f\|_2\right)
\,\leq \,C\epsilon_0 \sup_{\tau >0}\| \nabla S^{1,\eta}_{\tau} f\|_2,\end{equation*}
as desired. Here, the last inequality follows from Lemma~\ref{l2.10} $(ii)$. 
Moreover, by Cauchy-Schwarz, the
middle term on the right hand side of  \eqref{eq5.5} is no larger than
\begin{equation*}C\epsilon _0 \, |\| t\nabla \partial_t S^{L_0^*}_{-2 t} h\| |\cdot |\| t\nabla \partial
_t S^{1,\eta}_{t}f\| |
\,\leq\, C\epsilon_0 | \| t\nabla \partial_t S^{1,\eta}_{t}f\| |.\end{equation*}
In the first term on the right hand side of 
\eqref{eq5.5}, we integrate by parts again in $t$, to obtain
\begin{equation}\label{eq5.6}\frac{1}{2} \int^{\tau_2}_{\tau _1} \langle \nabla \partial^3_t S^{L_0^*}_{-2 t}
h,\epsilon
\nabla S^{1,\eta}_{t} f\rangle t^2dt\, +\,\text{Errors},\end{equation}
where the error terms correspond to the last two terms in 
\eqref{eq5.5} and are handled in a similar fashion.
Turning to the main term in \eqref{eq5.6}, we note that
\begin{equation*} \frac{1}{2}\partial^3_t S^{L_0^*}_{-2 t}\, h=\partial_s\partial^2_t S^{L_0^*}_{-t-s} \,h\vert
_{s=t}.\end{equation*}
Now set $g\equiv \partial^2_t S^{L_0^*}_{-t}\,h$. Let $u$ solve
\begin{equation*}\begin{cases} L^\ast _0 u=0 &\text{in } \mathbb{R}^{n+1}_-\\
u(\cdot,0)=g \end{cases}.\end{equation*}
By invertibility of the layer potentials for $L^\ast_0$, and by uniqueness, we have that
\begin{equation*} u(\cdot,-s)=
\mathcal{D}^{L_0^*}_{-s} \left( \frac{1}{2}I +K^{L_0^*}\right)^{-1}g.\end{equation*}
On the other hand, we also have that
$u(\cdot,-s)=\partial^2_t S^{L_0^*}_{-t-s}\,h.$
Consequently,
\begin{equation*}\partial_s\nabla u(\cdot,-s)=\partial _s\nabla \mathcal{D}_{-s}^{L_0^*} \left(\frac{1}{2}
I+K^{L_0^*}\right)^{-1} g
=\partial_s\nabla \partial^2_t S^{L_0^*}_{-t-s}\,h.\end{equation*}
Setting $s=t$, we have that
\begin{equation*}\frac{1}{2}\nabla 
\partial^3_t S^{L_0^*}_{-2t}h=-\,D_{n+1}\nabla \mathcal{D}_{-t}^{L_0^*} 
\left( \frac{1}{2}I+ K^{L_0^*}\right) ^{-1}g
= -\,D_{n+1}\nabla \mathcal{D}^{L_0^*}_{-t}
\left(\frac{1}{2}I+K^{L_0^*}\right)^{-1}\partial_t^2 S^{L_0^*}_{-t}\,
h.\end{equation*}
But, $\mathcal{D}^{L_0^*}_{-t}=(S^{L_0^*}_{-t}
\overline{\partial_{\nu_0}})$, where $\overline{\partial_{\nu_0}}$ denotes conjugate exterior
co-normal differentiation for $L_0$. Thus, 
\begin{equation*} ad\!j\left(\nabla D_{n+1} \mathcal{D}^{L_0^*}_{-t}\right)=
\left(\partial_{\nu_0}\partial_tS^0_{t}\nabla
\right).\end{equation*}
Therefore, the main term in \eqref{eq5.6} equals in absolute value
\begin{multline*}\left| \int^{\tau_2}_{\tau_1} \left\langle \left(\frac{1}{2} I+K^{L_0^*}\right)^{-1}
\partial^2_t S^{L_0^*}_{-t}\, h,\left(\partial_{\nu_0}
D_{n+1}S^0_{t}\nabla \right)\cdot\epsilon \nabla S^{1,\eta}_{t}f\right\rangle
t^2dt\right|\\ \leq C|\| t\partial^2_t S^{L_0^*}_{-t}h\| | \cdot|\| t^2
\left(\nabla D_{n+1}S^0_{t}\nabla \right)\cdot \epsilon \nabla S^{1,\eta}_{t}
f\| |\leq C| \| t^2\left(\nabla \partial_t S^0_{t}\nabla \right)\cdot
\epsilon \nabla S^{1,\eta}_{t}f\| |.\end{multline*}
To conclude the proof of Lemma \ref{l5.2}, it then suffices to prove that
\begin{equation}\label{eq5.main}| \| t^2\left(\nabla \partial_t S^0_{t}\nabla \right)\cdot
\epsilon \nabla S^{1,\eta}_{t}f\| | \leq C \epsilon_0\, {\bf M}^+.
\end{equation}
To this end, we first prove a lemma that will allow us to reduce matters to
\eqref{eq4.9i}.
\begin{lemma}\label{l5.discretize}For $k \in \mathbb{Z}$, set $t_k \equiv 2^{k-1}.$ Then
\begin{equation}\label{eq5.main12}
\sum_{k=-\infty}^\infty \int_{2^{k-1}}^{2^{k+2}}\! \!\int_{\mathbb{R}^n} |\nabla S_t^{1,\eta} f (x)
-\nabla S_{t_k}^{1,\eta}f(x)|^2 \frac{dx dt}{t} \leq C |\|t\nabla \partial_tS_t^{1,\eta}f\||_2.\end{equation}
\end{lemma}
Let us momentarily take the lemma for granted, and deduce
\eqref{eq5.main}.  Combining Lemma \ref{l2.10} $(i)$, Lemma \ref{l2.2sidegrad} and 
Lemma \ref{l5.discretize},
we may replace the square of the left hand side of \eqref{eq5.main} by
$$\sum_{k=-\infty}^\infty \int_{2^{k}}^{2^{k+1}}\! \!\int_{\mathbb{R}^n}
| t^2\left(\nabla \partial_t S^0_{t}\nabla \right)\cdot
\epsilon \nabla S^{1,\eta}_{t_k}f(x)|^2\frac{dx dt}{t}.$$ 
Since $u_k(\cdot,t) \equiv \left(\partial_t S^0_{t}\nabla \right)\cdot
\epsilon \nabla S^{1,\eta}_{t_k}f $ solves $L_0 u_k = 0$ in the upper half space,
we may use Caccioppoli's inequality in Whitney boxes to reduce matters to
considering
$$\sum_{k=-\infty}^\infty \int_{2^{k-1}}^{2^{k+2}}\! \!\int_{\mathbb{R}^n}
| t\left( \partial_t S^0_{t}\nabla \right)\cdot
\epsilon \nabla S^{1,\eta}_{t_k}f(x)|^2\frac{dx dt}{t}.$$
Applying Lemma \ref{l2.10} $(i)$ and Lemma \ref{l5.discretize} again, along with  \eqref{eq4.9i},
we obtain \eqref{eq5.main}.

\begin{proof}[Proof of Lemma \ref{l5.discretize}]
The left hand side of \eqref{eq5.main12} equals
\begin{multline*}\sum_{k=-\infty}^\infty \int_{2^{k-1}}^{2^{k+2}}\! \!\int_{\mathbb{R}^n}
\left|\frac{1}{\sqrt{t}}\int_{t_k}^t\nabla \partial_s S_s^{1,\eta} f (x)ds\right|^2 dx dt\\ \leq 
C \sum_{k=-\infty}^\infty \iint_{\mathbb{R}^{n+1}}
\left|\fint_{t_k}^t 1_{\{2^{k-1}\leq s < 2^{k+2}\}}\sqrt{s}
\nabla \partial_s S_s^{1,\eta} f (x)ds\right|^2 dt dx.\end{multline*}
The desired bound now follows from the Hardy-Littlewood maximal theorem.
\end{proof}

This concludes the proof Lemma \ref {l5.2}, and thus also that of Theorem
\ref{t1.10}\end{proof}

\section{Proof of Theorem \ref{t1.11}:  boundedness\label{s6}}

Let $L\equiv -\dv A \nabla,$ where $A$ is real, symmetric, $L^\infty$, $t$-independent 
and uniformly elliptic.
In this section, we show that the layer potentials associated to $L$ are
bounded;  we defer the proof of invertibility to the next section.
By the classical de Giorgi-Nash Theorem, estimates
\eqref{eq1.2} and \eqref{eq1.3} hold for solutions of $Lu=0$. By Lemma~\ref{l3.3} and Lemma
\ref{l4.nt1}, in order to establish boundedness of the layer potentials, it
suffices to prove
\begin{equation}\label{eq6.1} \sup_{t\neq 0} \| \partial_t S_t f\|_2\leq C\|f\|_2\end{equation}
and \begin{equation}
\label{eq6.2}\int^\infty_{-\infty} \int_{\mathbb{R}^n}|t\partial^2_t S_t f|^2dx \frac{dt}{|t|}\leq C\| f\|_2.\end{equation}
By Lemma~\ref{l2.1}, the kernel
$K_t (x,y)\equiv \partial_t \Gamma (x,t,y,0)$
satisfies the standard Calder\'on-Zygmund estimates \begin{subequations}
\label{eq6.3}
\begin{equation}\label{eq6.3a}|K_t(x,y)|\leq \frac{c}{|x-y|^n}\end{equation}
\begin{equation}\label{eq6.3b} |K_t(x,y+h)-K_t(x,y)|+|K_t(x+h,y)-K_t(x,y)|\leq C\frac{|h|^\alpha}{|x-y|^{n+\alpha}},\end{equation}\end{subequations}
uniformly in $t$, where the later inequality holds for some $\alpha >0$ whenever $|x-y|>2 |h|$.
In addition, the kernel
$$\psi _{t}(x,y)\equiv t\partial^{2}_{t}\Gamma (x,t,y,0)$$
satisfies the standard Littlewood-Paley kernel conditions
\begin{equation}\begin{split}\label{eq6.4}|\psi _{t}(x,y)|&\leq \frac{|t|}{(|t|+|x-y|)^{n+1}}\\
|\psi_t (x,y+h)-\psi_t (x,y)|&\leq \frac{C|t|\,|h|^\alpha}{(|t|+|x-y|)^{n+1+\alpha}} \leq \frac{C|h|^\alpha}{(|t|+|x-y|)^{n+\alpha}}\end{split}\end{equation}
for some $\alpha >0$, whenever $|h|\leq \frac{1}{2}|x-y|$ or $|h|\leq |t|/2$.

The bound \eqref{eq6.2} will be deduced from the following ``local" Tb Theorem for square functions

\begin{theorem}\label{t6.5} Let $\theta_t f(x)\equiv \int_{\mathbb{R}^n}\psi_t (x,y)f(y)dy$, where $\psi_t (x,y)$ satisfies
\eqref{eq6.4}. Suppose also that there exists a system $\{b_Q\}$ of functions indexed by cubes $Q\subseteq
\mathbb{R}^n$ such that for each cube $Q$
\begin{enumerate}\item[(i)] $\int_{\mathbb{R}^n} |b_Q|^2\leq C|Q|$
\item[(ii)] $\int^{\ell (Q)}_0\int_Q |\theta _t b_Q (x)|^2\frac{dxdt}{t}\leq C|Q|$
\item[(iii)] $\frac{1}{C} |Q|\leq \Re e \int_Q b_Q$.\end{enumerate}
Then we have the square function bound
\begin{equation*}|\| \theta_t f\| |\leq C\| f\|_2.\end{equation*}\end{theorem}
We omit the proof here.  
A direct proof of the present formulation of Theorem~\ref{t6.5} may be found in \cite{A2} or
\cite{H2}, although we note that the theorem and its proof were already
implicit in the proof of the Kato square root conjecture \cite{HMc}, \cite{HLMc} and \cite{AHLMcT}; 
see also the works \cite{Ch}, \cite{S} and \cite{AT} for some important antecedents.  

We shall deduce estimate \eqref{eq6.1} as a consequence of the following extension of a local Tb Theorem for singular
integrals introduced by M. Christ \cite{Ch} in connection with the theory of analytic capacity. A 1-dimensional
version of the present result, valid for ``perfect dyadic" Calder\'on-Zygmund kernels, appears in \cite{AHMTT}.   A self-contained proof of the more general formulation below may be found in
\cite{H3}.  Alternatively, the result of \cite{AY} may be combined with that of \cite{AHMTT} to deduce
the general case (in the slightly sharper form $q=2$).
 In the sequel, we  let $T^{tr}$  denote the transpose of the operator
$T$.

\begin{theorem}\label{t6.6} Let $T$ be a singular integral operator associated to a kernel $K$ satisfying
\eqref{eq6.3}, and suppose that $K$ satisfies the generalized truncation condition $K(x,y)\in L^\infty
(\mathbb{R}^n\times \mathbb{R}^n)$. Suppose also that there exist pseudo-accretive systems $\{ b^1_Q\}$, $\{ b^2_Q\}$
such that $b^1_Q$ and  $b^2_Q$ are supported in $Q$, and
\begin{enumerate}\item[(i)] $\int_Q \left(|b^1_Q|^q+|b^2_Q|^q\right)\leq C|Q|$, for some $q>2$
\item[(ii)] $\int_Q \left(|Tb^1_Q|^2 + |T^{\tr} b^2_Q |^2\right)\leq C|Q|$
\item[(iii)] $\frac{1}{C}|Q|\leq \min\left(\Re e \int_Q b^1_Q,
\Re e \int_Qb^2_Q\right)$.\end{enumerate}
Then $T:L^2(\mathbb{R}^n)\to L^2(\mathbb{R}^n)$, with bound independent of $\| K\|_\infty $.\end{theorem}

Let us first show that Theorem~\ref{t6.5} implies~\eqref{eq6.2}. 
As usual, we may restrict our attention to the case
$t>0$. As above let $\psi_t(x,y)\equiv t\partial_t^2\Gamma (x,t,y,0)$, so that
\begin{equation*}t\partial^2_t S_tf(x)\equiv \theta_t f(x)=\int_{\mathbb{R}^n}\psi_t(x,y)f(y)dy.\end{equation*}
By Theorem~\ref{t6.5}, it suffices to construct a system $\{ b_Q\}$ satisfying the hypotheses (i), (ii) and (iii) of
the Theorem.

Our functions $b_Q$ will be normalized Poisson kernels. Given a cube
$Q\subset \mathbb{R}^n$, let $x_Q$ denote its center, and let 
$\ell(Q)$
denote its side length. We define
\begin{equation*} A^+_Q \equiv (x_Q,\ell(Q))\in
  \mathbb{R}^{n+1}_+,\quad A^-_Q \equiv (x_Q, -\ell (Q))\in
  \mathbb{R}^{n+1}_-.\end{equation*}
Given $X^+\in \mathbb{R}^{n+1}_+$, $X^-\in \mathbb{R}^{n+1}_-$, let 
$k^{X^+}_+(y)$, $k_-^{X^-}(y)$ denote, respectively, the Poisson kernels
  for $L$ in the upper and lower half spaces, and let $G^+(X,Y)$,
  $G^-(X,Y)$ denote the corresponding Green functions, so that
\begin{equation*}k_+^{X^+}(y)\equiv \frac{\partial G^+}{\partial \nu^+_y}
  (X^+,y,0),\quad
k^{X^-}_-(y)\equiv \frac{\partial G^-}{\partial
  \nu^-_y}(X^-,y,0),\end{equation*}
where $\frac{\partial }{\partial \nu^+_y}$, $\frac{\partial}{\partial
  \nu^-_y}$ denote the co-normal derivatives at the
point $y\in \partial \mathbb{R}^{n+1}_+$, $\partial
\mathbb{R}_-^{n+1}$ respectively. We now set
\begin{equation}\label{eq6.7} b_Q\equiv |Q|\,k^{A^-_Q}_-.\end{equation}
We recall the following fundamental result of Jerison and
Kenig~\cite{JK1} (see also \cite[pp 63-64]{K}), which amounts to   the solvability
of (D2) in the lower half-space:

\begin{theorem}\label{t6.8}\cite{JK1} Suppose that $L=-\dv A\nabla$,
  where $A$ is real, symmetric, $(n+1)\times (n+1)$, $t$-independent, $L^\infty$ and uniformly elliptic. 
Then there exists $\varepsilon_1 \equiv \varepsilon_1(n,\lambda,\Lambda)$ 
such that for all $0\leq \varepsilon <\varepsilon_1$ and for every cube $Q$,
  \begin{equation}\label{eq6.9}
\int_{\mathbb{R}^n} (k^{A^-_Q}_- (y))^{2+\varepsilon} dy\leq C_\varepsilon|Q|^{-1-\varepsilon}.
\end{equation}
\end{theorem}
We remark that \eqref{eq6.9} is usually stated in terms of an integral over $Q$, but in fact the
global bound follows from the local one and duality, since
by \cite{JK1}, \cite{K} the 
local version of \eqref{eq6.9} and the $L^p$ version of \eqref{eq1.3} yield the estimate
\begin{equation*}
|u(A_Q^-)|\leq C\sup_{t<0} \| u(\cdot ,t)\|_{L^p(\mathbb{R}^n)}\leq 
C\| g\|_{L^p (\mathbb{R}^n)},\end{equation*}
where $u(x,t) = \int_{\mathbb{R}^n} k_-^{x,t}(y) g(y) dy,$ and $p$ is the dual exponent to
$2+\varepsilon$.
  
  We now note that hypothesis (i) of Theorem~\ref{t6.5} follows immediately from \eqref{eq6.7} and \eqref{eq6.9}.  Moreover,
(iii) follows immediately from \eqref{eq6.7} and the following well
known estimate of Caffarelli, Fabes, Mortola and Salsa \cite{CFMS} (also \cite[Lemma 1.3.2, p. 9]{K}):
  \begin{equation}\label{eq6.10} \omega^{A^-_Q}_- (Q)\geq \frac{1}{C},\end{equation}
  where $\omega^{X^-}_-$ denotes harmonic measure for $L$ at $X^-\in \mathbb{R}^{n+1}_-$.
  
  It remains to verify that $b_Q$ as defined in \eqref{eq6.7} satisfies hypothesis (ii) of Theorem~\ref{t6.5}. To this end,
let $(x,t)\in R^+_Q\equiv Q \times (0,\ell (Q))$. Then,
since for fixed $(x,t)\in \mathbb{R}^{n+1}_+$, 
we have that $\partial^2_t \Gamma (x,t,\cdot , \cdot )$ is a solution
of $Lu=0$ in $\mathbb{R}^{n+1}_-$,
\begin{equation}\label{eq6.11a}\theta_t b_Q(x) = |Q|\,t\int \partial^2_t \Gamma (x,t,y,0)\, k^{A^-_Q}_-(y) dy
=|Q|\, t \,\partial^2_t \Gamma (x,t,A^-_Q),\end{equation}
by Theorem \ref{t6.8} (i.e., \cite{JK1}) and uniqueness in (D2)
(e.g., Lemma \ref{l4.ntunique} $(i)$, although of course, uniqueness in the present setting of real symmetric coefficients appears already in \cite{JK1}, \cite{K}). Therefore, by \eqref{eq2.5} and 
translation invariance in $t$, we have that
\begin{equation*}|\theta_tb_Q(x)|\leq C\frac{t}{\ell (Q)},\end{equation*}
from which hypothesis (ii) follows readily. Thus, given Theorem~\ref{t6.5}, we conclude that
\begin{equation*}\int^\infty_0\int_{\mathbb{R}^n} |t\partial^2_t S_tf(x)|^2 \frac{dxdt}{t} \leq C\| f\|^2_2.\end{equation*}
The corresponding square function estimate in the lower half-space follows by the same argument, if we replace
$k^{A^-_Q}_-$ by $k^{A^+_Q}_+$ in the definition of $b_Q$. We then obtain \eqref{eq6.2} as desired.

Next, we show that Theorem~\ref{t6.6} implies \eqref{eq6.1}. We consider only the case $t>0$, the other case being handled by a
similar argument. Again, it suffices to construct systems $\{b^1_Q\}$, $\{b^2_Q\}$, now with $b^1_Q$ and $b^2_Q$
supported in $Q$, satisfying hypotheses (i), (ii) and (iii) of Theorem~\ref{t6.6}.

In fact, we shall use the same construction as before, except that we truncate the function outside of $Q$, i.e. we set
\begin{equation}\label{eq6.11} b^1_Q\equiv |Q|\,k^{A^-_Q}_- 1_Q=b_Q1_Q,\quad
b^2_Q \equiv |Q|\, k^{A^+_Q}_+ 1_Q\end{equation}
As before, (iii) and (i) follow immediately from  
\cite{CFMS}, and \eqref{eq6.9}, respectively.

It remains to establish (ii). We observe first that, as in \eqref{eq6.11a},
\begin{equation*}| \partial_tS_t b_Q (x)|=|Q|\left|
\int_{\mathbb{R}^n}\partial_t \Gamma (x,t,y,0) k^{A^-_Q}_-
(y)dy\right|=|Q|\,|\partial_t\Gamma (x,t,A^-_Q)| \leq C,\end{equation*}
uniformly in $(x,t)\in \mathbb{R}^{n+1}_+$, where $b_Q$ is defined as in \eqref{eq6.7}, and we have used \eqref{eq2.5} and the fact that $t>0$. 
We now claim that, for $x\in Q$ and $t>0$, the same $L^\infty$ bound holds
for $\partial_tS_t (\varphi_Qb_Q)(x)$, where $\varphi_Q\in C^\infty_0$, $\varphi_Q\equiv 1$ on $5Q$, $\supp \varphi_Q \subseteq 6Q$,
with $\| \nabla_{\|}\varphi_Q\|_\infty \leq C/\ell(Q)$. Indeed, fixing $(x,t)\in Q\times (0,\infty)$, and setting
$u=\partial_t\Gamma (x,t,\cdot ,\cdot )$, we have that 
\begin{equation*}\partial_t S_t (\varphi_Q b_Q)(x)=|Q|\int_{\mathbb{R}^n}u(y,0)\varphi_Q (y) \frac{\partial G^-}{\partial \nu^-_y}(A^-_Q,
y,0)dy.\end{equation*}
We now extend $\varphi_Q$ smoothly into the lower half-space so that $\varphi_Q(y,s)\equiv 1$ on $5Q\times (0,-\ell (Q)/4)$,
$\varphi_Q(y,s)$ vanishes in $\mathbb{R}^{n+1}_-\backslash [6Q\times (0,-\ell (Q)/2)]$, and $$\| \nabla \varphi_Q\|_{L^\infty
(\overline{\mathbb{R}^{n+1}_-})}\leq C\ell (Q)^{-1}.$$ Since $G^-(A^-_Q , \cdot ,\cdot )$ and $u$ are both solutions of
$Lu=0$ in $\supp\varphi_Q$, we obtain from Green's formula (whose use 
may be justified in the sense of 
Lemma \ref{l4.ntconverge} $(iii)$) that
\begin{multline*}\partial_t S_t (\varphi_Qb_Q)(x)=|Q|\iint_{\mathbb{R}_-^{n+1}}A_{y,s}\nabla
G^-(A^-_Q,y,s)\nabla \varphi_Q (y,s) u(y,s)dyds\\
 -\, |Q|\iint_{\mathbb{R}^{n+1}_-}G^- \nabla \varphi_Q\cdot A\nabla udyds
\,\equiv\, I+II.\end{multline*}

We first consider term II. Let $D_Q\equiv \supp \nabla \varphi_Q$. By the definition of $\varphi_Q(y,s)$, 
a standard estimate for $G^-$,  and Cauchy-Schwarz, we have that
\begin{multline*} |II|
\leq C\ell (Q)^{(n+1)/2} \left( \iint_{D_Q} |\nabla \partial_t \Gamma (x,t,y,s)|^2 dyds\right)^{\frac{1}{2}}\\
\leq C\ell(Q)^{(n-1)/2} \left(\iint_{\tilde{D}_Q} |\partial_t \Gamma (x,t,y,s)|^2
dyds\right)^{\frac{1}{2}},\end{multline*}
where the last inequality follows by Caccioppoli's inequality, and where $\tilde{D}_Q$ is a fattened version of
$D_Q$. But for $x\in Q$, $t>0$, and $(y,s)\in \tilde{D}_Q$, we have by \eqref{eq2.5} that
\begin{equation*}|\partial_t \Gamma (x,t,y,s)|\leq C|Q|^{-1},\end{equation*}
hence $|II|\leq C.$
Similarly,
\begin{equation*}|I|\leq 
C \ell (Q)^{(n-1)/2} \left( \iint_{D_Q} |\nabla G^-(A^-_Q ,y,s)|^2\right)^{\frac{1}{2}}
\leq C,\end{equation*}
again by Caccioppoli. Altogether then, 
$\sup_{t>0} \| \partial_t S_t (\varphi_Qb_Q)\| _{L^\infty (Q)}\leq C,$
and therefore
\begin{equation}\label{eq6.12} \sup_{t>0}\int_Q |\partial_t S_t (\varphi_Q b_Q)|^2\leq C|Q|.\end{equation}
To prove (ii), it will be enough to observe that, for any kernel $K(x,y)$ satisfying \eqref{eq6.3}(a), we have
\begin{equation}\label{eq6.13} \int_Q \left| \int K(x,y)1_{6Q\backslash Q}(y)f(y)dy\right|^2 dx\leq C\int_{6Q\backslash
Q} |f|^2.\end{equation}
Indeed, given \eqref{eq6.13}, we may replace $\varphi_Q$ by $1_Q$ in \eqref{eq6.12} (with controlled error), and (ii) follows. The proof of
\eqref{eq6.13} is omitted. Since $\Gamma (x,t,y,0)=\Gamma (y,-t,x,0)$, a similar argument yields the
corresponding bound for $(\partial_tS_t)^{tr}(b^2_Q)$, and \eqref{eq6.1} now follows.

\section{Proof of Theorem \ref{t1.11}:  invertibility\label{s7}}
We now consider invertibility of the layer potentials
in the case of real symmetric coefficients.  The proof will follow the strategy of Verchota \cite{V},
using the well known
``Rellich identities" combined with the method of continuity.  In our case, the 
continuity argument will exploit
Theorem \ref{t1.10}.  
\begin{proof}[Proof  of Invertibility]
From self-adjointness and integration by parts, we obtain the equivalence
\begin{equation}\label{eq9.1}
\|\partial_\nu u\|_{L^2(\mathbb{R}^n)} \approx \|\nabla_x u\|_{L^2(\mathbb{R}^n)},
\end{equation}
for solutions of $Lu= 0$ in $\mathbb{R}^{n+1}_\pm$ for which
$\widetilde{N}_*(\nabla u) \in L^2$, where the implicit constants depend only upon ellipticity 
(see, e.g., \cite{K} for details).
In particular, \eqref{eq9.1} holds for $u(\cdot,t) \equiv S_t f,$ with $f\in L^2$.
By the jump relation formulae Lemma \ref{l4.ntjump},
\eqref{eq9.1} becomes
\begin{equation}\label{eq9.2}\|\left(\pm\frac{1}{2}I + \widetilde{K} \right)f\|_2 \approx 
\|\nabla_x S_0f\|_2.\end{equation}
Thus, by the triangle inequality and \eqref{eq9.2} we have
\begin{equation}\label{eq9.3}
\|f\|_2 \leq C \|\left(\frac{1}{2}I + \widetilde{K} \right)f\|_2 
\end{equation}
and also
\begin{equation}\label{eq9.4}
\|f\|_2 \leq C \|\nabla_x S_0f\|_2,
\end{equation}
where the constants in \eqref{eq9.3} and \eqref{eq9.4} depend only on ellipticity.
Moreover, if we set $$L_\sigma \equiv -\dv A_\sigma \nabla, \quad 0\leq \sigma \leq 1,$$ where
$$A_\sigma \equiv (1-\sigma)\mathbb{I} + \sigma A,$$  and $\mathbb{I}$ denotes the
$(n+1) \times (n+1)$ identity matrix, then \eqref{eq9.3} and \eqref{eq9.4}
hold, uniformly in $\sigma$, for the layer potentials associated to $L_\sigma$;  indeed,
we have uniform control of the ellipticity constants for $A_\sigma$.  By the result of 
Section \ref{s6}, we of course have boundedness of the layer potentials associated to
$L_\sigma$, again with uniform constants depending only upon ellipticity and dimension.
Thus, once we have established invertibility of the layer potentials associated to
$L_\sigma,$ for a given $\sigma$, the corresponding Layer Potential Constants will depend only upon ellipticity and dimension, since, in particular, the quantitative bounds for the inverses are
precisely the constants in \eqref{eq9.3} and \eqref{eq9.4}.
We may therefore establish invertibility of $\frac{1}{2}I + \widetilde{K}:L^2 \to L^2$
and $S_0:L^2 \to \dot{L}^2_1$ as follows.  Since $L_0$ clearly has Good Layer Potentials,
we may invoke Theorem \ref{t1.10} to deduce that $L_\sigma$ has Good Layer potentials,
for $0\leq\sigma < \epsilon_0$, for some $\epsilon_0$ depending only upon ellipticity and dimension.
By our previous observation concerning the uniform control of the layer potential constants,
we may then iterate this procedure, advancing each time by the same distance $\epsilon_0$,
so that we reach $A = A_1$ in finitely many steps.
\end{proof}

\section{Appendix:  constant coefficients \label{sappend}}

Suppose that $L= -\dv a\nabla$, where $a$ is a constant 
complex elliptic matrix.  Following \cite{FJK}, we observe that $L$ has Fourier symbol
$$q(i\xi,i\tau)= \sum_{j,k =1}^{n+1} a_{j,k} \xi_j\xi_k = a_{n+1,n+1}(\tau-\tau_+(\xi))
(\tau-\tau_-(\xi)),$$
where $\xi_{n+1} \equiv \tau$, and $\tau_\pm : \mathbb{R}^n \to \mathbb{C}$
are each homogeneous of degree $1$, $C^\infty(S^{n-1}),$ with 
\begin{equation}\label{eq10.0}
\Im m\, \tau_+ (\xi) \geq \mu ,\quad \Im m\, \tau_-(\xi) \leq -\mu,
\end{equation}
for some $\mu > 0$.  In particular,
\begin{equation}
\label{eq10.1}
|\tau_+(\xi) - \tau_-(\xi)|\approx |\xi|, \quad \xi \in\mathbb{R}^n.\end{equation}
The fundamental solution $\Gamma(x,t)$ is a convolution kernel with Fourier
symbol $q(i\xi,i\tau)^{-1}$.  Inverting the Fourier symbol in $t$ only, and then using 
the method of residues, we obtain
$$\widehat{\Gamma}(\cdot,t)(\xi)= \frac{1}{2\pi}\int_{-\infty}^\infty \frac{e^{it\tau}}{q(i\xi,i\tau)}d\tau 
=- \frac{e^{it\tau_+(\xi)}1_{\{t>0\}} 
+e^{it\tau_-(\xi)}1_{\{t<0\}}}{ia_{n+1,n+1}\left(\tau_+(\xi) - \tau_-(\xi)\right)},$$
so by \eqref{eq10.1} and the accretivity of $a_{n+1,n+1}$, we  have in particular that
$$| \,\widehat{\Gamma}(\cdot,0)(\xi)\,|\approx |\xi|^{-1}.$$ Consequently,
$S_0: L^2 \to \dot{L}^2_1$ is bounded and invertible, by Plancherel's Theorem.  
One also readily verifies via Plancherel's Theorem that
$$\sup_{t\neq 0} \|\nabla S_t\|_{op} \leq C,\quad \||t\partial_t^2 S_t\||_{op} \leq C.$$
Finally, we note that $f\to \left(\frac{1}{2}I + \widetilde{K}\right)f = \partial_\nu S_t f|_{t=0^+}$
is invertible on $L^2$.  Indeed, the corresponding Fourier symbol is
$$-\lim_{t\to 0^+}e_{n+1} \cdot a \widehat{\nabla \Gamma}(\cdot,t)(\xi) =  
\frac{a_{n+1,n+1}\tau_+(\xi)+\sum_{j=1}^n a_{n+1,j} \xi_j}{a_{n+1,n+1}
\left(\tau_+(\xi) - \tau_-(\xi)\right)},$$
and by \cite{AQ}, Lemma 4, 
the modulus of the numerator $\approx |\xi|.$  By the accretivity of $a_{n+1,n+1}$ and \eqref{eq10.1},
the same holds for the denominator, and the invertibility follows.
Of course, a similar observation holds for $-\frac{1}{2}I + \widetilde{K}$.

\end{document}